\documentclass[a4paper, 11pt]{amsart}

\let\oldtocsection=\tocsection

\let\oldtocsubsection=\tocsubsection

\let\oldtocsubsubsection=\tocsubsubsection

\renewcommand{\tocsection}[2]{\hspace{0em}\oldtocsection{#1}{#2}}
\renewcommand{\tocsubsection}[2]{\hspace{2em}\oldtocsubsection{#1}{#2}}
\renewcommand{\tocsubsubsection}[2]{\hspace{2em}\oldtocsubsubsection{#1}{#2}}

\title{Day convolution for algebraic patterns}
\author{Thomas Blom, F\'elix Loubaton and Jaco Ruit}
\address{Max-Planck-Institut f\"ur Mathematik, Vivatsgasse 7, Bonn, Germany}
\email{blom@mpim-bonn.mpg.de}
\address{CNRS, Aix-Marseille Universit\'e, 163 Avenue de Luminy, Marseille, France}
\email{felix.loubaton@cnrs.fr}
\address{Max-Planck-Institut f\"ur Mathematik, Vivatsgasse 7, Bonn, Germany}
\email{ruit@mpim-bonn.mpg.de} 

\usepackage{multicol}
\usepackage{amssymb}
\usepackage{amsfonts}
\usepackage{amsthm}

\usepackage{kpfonts} 
\usepackage[mathscr]{eucal}

\usepackage[english]{babel}

\usepackage{tikz-cd}{}  
\usetikzlibrary{backgrounds}
\usetikzlibrary{patterns}
\usetikzlibrary{decorations.pathmorphing}
\usetikzlibrary{decorations.pathreplacing}
\usepackage{quiver}
\usepackage[shortlabels]{enumitem}
\usepackage[linktocpage]{hyperref}
\usepackage[noabbrev,capitalize]{cleveref}
\usepackage{aliascnt}
\usepackage{xstring}
\usepackage{xcolor}
\usepackage{bbm}
\usepackage{breakurl}
\usepackage{todonotes}
\usepackage[normalem]{ulem}
\usepackage{mathdots}
\usepackage{float}

\usepackage[left=3cm,right=3cm,top=3.2cm,bottom=3.2cm]{geometry} 
\newtheorem{thmintro}{Theorem}

\newtheorem{theorem}{Theorem}
\newcommand{\jnewtheorem}[2]{
	\newaliascnt{#1}{theorem}
	\newtheorem{#1}[#1]{#2}
	\aliascntresetthe{#1}
}
\numberwithin{theorem}{section}
\jnewtheorem{lemma}{Lemma}
\jnewtheorem{proposition}{Proposition}
\jnewtheorem{corollary}{Corollary}

\theoremstyle{definition}
\jnewtheorem{definition}{Definition}
\jnewtheorem{example}{Example}
\jnewtheorem{remark}{Remark}
\jnewtheorem{construction}{Construction}
\jnewtheorem{notation}{Notation}
\jnewtheorem{convention}{Convention}
\jnewtheorem{warning}{Warning}

\crefformat{item}{#2#1#3}
\crefformat{enumi}{#2#1#3}
\crefformat{equation}{#2(#1)#3}
\crefformat{subsecconventions}{#2Conventions#3}

\hypersetup{
	pdftitle={Day convolution for algebraic patterns},
	pdfauthor={Thomas Blom, F\'elix Loubaton and Jaco Ruit},
	bookmarksnumbered=true,     
	bookmarksopen=true,         
	bookmarksopenlevel=1,       
	colorlinks=true,   
	linkcolor = {cyan!60!black},
	urlcolor = {cyan!60!black},
	citecolor = {cyan!60!black},   
	pdfencoding=unicode,      
	pdfstartview=Fit,           
	pdfpagemode=UseOutlines,    
	pdfpagelayout=OneColumn,
	breaklinks
}

\newcommand{\op}{\mathrm{op}}
\newcommand{\map}{\mathrm{Hom}}
\newcommand{\Map}{\map}
\def\colim{\qopname\relax m{colim}}

\newcommand{\Fun}{\mathrm{Fun}}
\newcommand{\Cat}{\mathrm{Cat}}
\newcommand{\ev}{\mathrm{ev}}
\newcommand{\Algd}{\mathrm{Algad}}
\newcommand{\Cocart}{\mathrm{Cocart}}
\newcommand{\LFib}{\mathrm{LFib}}

\newcommand{\cS}{\mathscr{S}}
\renewcommand{\D}{\mathscr{D}}
\renewcommand{\P}{\mathscr{P}}
\newcommand{\cP}{\P}
\newcommand{\Q}{\mathscr{Q}}
\newcommand{\cQ}{\Q}

\renewcommand{\O}{\mathscr{O}}

\newcommand{\cO}{\O}
\newcommand{\cB}{\mathscr{B}}
\newcommand{\cC}{\C}
\newcommand{\cD}{\D}
\newcommand{\Seg}{\mathrm{Seg}}

\newcommand{\CSeg}{\mathrm{CSeg}}
\newcommand{\Hom}{\mathrm{Hom}}
\newcommand{\DblCat}{\mathrm{DblCat}}
\newcommand{\VirtDblCat}{\mathrm{VirtDblCat}}
\newcommand{\Fact}{\mathrm{Fact}}
\newcommand{\AlgPatt}{\mathrm{AlgPatt}}
\newcommand{\DblPatt}{\mathrm{DblPatt}}
\newcommand{\Set}{\mathrm{Set}}

\newcommand{\PSh}{\mathrm{PSh}}
\newcommand{\Sq}{\mathrm{Sq}}	
\newcommand{\id}{\mathrm{id}}
\newcommand{\CCAT}{\mathbb{C}\mathrm{at}}
\newcommand{\Span}{\mathrm{Span}}
\newcommand{\SSpan}{\mathbb{S}\mathrm{pan}}
\newcommand{\DCospan}{\mathbb{C}\mathrm{ospan}}
\newcommand{\Sp}{\mathrm{Sp}}
\newcommand{\el}{\mathrm{el}}
\newcommand{\elem}{\el}
\newcommand{\activ}{\act}
\newcommand{\inert}{\mathrm{int}}
\newcommand{\act}{\mathrm{act}}
\newcommand{\intarrow}{\rightarrowtail}
\newcommand{\actarrow}{\rightsquigarrow}
\newcommand{\dbl}{\mathrm{dbl}}
\newcommand{\alg}{\mathrm{alg}}
\newcommand{\Ar}{\mathrm{Ar}}
\newcommand{\F}{\mathbb{F}}
\newcommand{\Spc}{\mathscr{S}}
\newcommand{\Orb}{\mathrm{Orb}}

\newcommand{\Forest}{\Phi}
\newcommand{\Tree}{\Omega}
\newcommand{\C}{\mathscr{C}}

\usepackage{trimclip}
\makeatletter
\newcommand{\twoheaddownarrow}{\mathrel{\mathpalette\twoheaddownarrow@\relax}}
\newcommand{\twoheaddownarrow@}[2]{%
  \begingroup
  \sbox\z@{$\m@th#1\downarrow$}%
  \vphantom{\copy\z@}
  \ooalign{\copy\z@\cr\hidewidth\clipdownarrow@{#1}\hidewidth\cr}%
  \endgroup
}
\newcommand{\clipdownarrow@}[1]{%
  \raisebox{0.25\ht\z@}{\clipbox{-3pt 0pt -3pt {0.5\height}}{\copy\z@}}%
}
\makeatother

\makeatletter
\def\slashedarrowfill@#1#2#3#4#5{%
  $\m@th\thickmuskip0mu\medmuskip\thickmuskip\thinmuskip\thickmuskip
  \relax#5#1\mkern-7mu%
  \cleaders\hbox{$#5\mkern-2mu#2\mkern-2mu$}\hfill
  \mathclap{#3}\mathclap{#2}%
  \cleaders\hbox{$#5\mkern-2mu#2\mkern-2mu$}\hfill
  \mkern-5mu#4$%
}
\def\rightslashedarrowfill@{%
  \slashedarrowfill@\relbar\relbar\mapstochar\rightarrow}
\newcommand\xslashedrightarrow[2][]{%
  \ext@arrow 0055{\rightslashedarrowfill@}{#1}{#2}}
\makeatother

\begin{document}

\begin{abstract}
We characterize the exponentiable objects for a wide range of structures prevalent in $\infty$-categorical algebra, extending the construction of Day convolution to more general structures than $\infty$-operads. More precisely, we give a criterion that is both necessary and sufficient for many of these structures encountered in practice, such as (equivariant) $\infty$-operads and virtual double $\infty$-categories. We work within the framework of algebraic patterns of Chu--Haugseng that describe these structures in terms of weak Segal fibrations. As part of the proof, we give a new description of weak Segal fibrations in terms of generalized Segal spaces on certain ``tree'' categories. We also define the ``underlying graph'' of a weak Segal fibration, extending the notion of the underlying $\infty$-category for $\infty$-operads, and explicitly describe the underlying graph of exponential objects in weak Segal fibrations.
\end{abstract}

\maketitle


\setcounter{tocdepth}{1}

\tableofcontents

\section{Introduction}

One of the cornerstones of modern higher algebra is the Day convolution symmetric monoidal structure on functor categories.
Given a symmetric monoidal category $(\cC,\otimes)$, Day \cite{Day1970ClosedCategoriesFunctors} constructed a \emph{convolution} monoidal structure $\circledast$ on $\Fun(\cC,\mathrm{Set})$ given by the coend formula
\[(F \circledast G)(c) = \int^{(c_1,c_2) \in \cC \times \cC} F(c_1) \times G(c_2) \times \Hom(c_1 \otimes c_2, c).\]
Glasman \cite{Glasman2016DayConvolution} considered a higher-categorical analogue of Day's convolution monoidal structure, which was subsequently generalized to the setting of $\infty$-operads by Lurie \cite[\S 2.2.6]{HA} and Hinich \cite[\S 2.8]{Hinich2020YonedaLemmaEnriched}.

Given two $\infty$-operads $\cP$ and $\mathscr{Q}$, their \emph{Day convolution} is, if it exists, an $\infty$-operad $[\cP,\mathscr{Q}]$ characterized by the universal property
\[\Hom_{\mathrm{Op}_\infty}(\cO \times \cP, \mathscr{Q}) \simeq \Hom_{\mathrm{Op}_\infty}(\cO, [\cP,\mathscr{Q}]);\]
that is, it is an \emph{internal hom-object} or \emph{exponential object} in the $\infty$-category $\mathrm{Op}_\infty$ of $\infty$-operads.
Given the importance of Day convolution in higher algebra, it is natural to wonder whether such internal hom-objects also exist for other operad-like structures, such as equivariant $\infty$-operads and virtual double $\infty$-categories.
In this paper, we characterize the \emph{exponentiable objects} for many such operad-like structures; that is, the objects $\cP$ for which the functor $- \times \cP$ admits a right adjoint.

Our main result (\cref{thmA}) is such a characterization in the context of \emph{algebraic patterns}.
The theory of algebraic patterns, introduced by Chu--Haugseng \cite{ChuHaugseng2021HomotopycoherentAlgebraSegal}, provides a very general framework for talking about operad-like structures. An algebraic pattern is an $\infty$-category $\cO$ equipped with a factorization system $(\cO^\inert, \cO^\activ)$ of \emph{inert} and \emph{active} morphisms and a full subcategory $\cO^\elem \subset \cO^\inert$ of \emph{elementary} objects. Lurie's definition of an $\infty$-operad was generalized to any algebraic pattern in \cite{ChuHaugseng2021HomotopycoherentAlgebraSegal} by considering  certain functors over $\O$ called \emph{weak Segal fibrations}.

Depending on the choice of algebraic pattern $\cO$, weak Segal fibrations can describe, among others, the following operad-like structures:
\begin{itemize}
    \item (generalized) $\infty$-operads, as introduced by Lurie \cite{HA},
    \item non-symmetric $\infty$-operads,
    \item equivariant $\infty$-operads \cite{NardinShah},
    \item virtual double $\infty$-categories (cf.\ \cite{GepnerHaugseng2015EnrichedCategoriesNonsymmetric}, where they are called generalized non-symmetric $\infty$-operads).
\end{itemize}

Throughout this paper, and the rest of this introduction, we have decided to change the name \emph{weak Segal fibration} to \emph{algebrad}.
The reason is that in this paper, we will give several equivalent definitions of weak Segal fibrations/algebrads.
In these other contexts they are not necessarily a type of fibration (see e.g.\ \cref{thmD} below), so the original terminology seemed less suitable to us.
We chose the term \emph{algebrad} since it reflects that this notion is a generalization of an \emph{operad}.

\subsection{A Conduché criterion for weak Segal fibrations}

In \cref{thmA}, we describe a criterion for detecting exponentiable objects in the $\infty$-category $\Algd(\cO)$ of \emph{$\cO$-algebrads} (i.e.\ weak Segal fibrations over $\cO$). It is similar to the \emph{Conduché criterion} for detecting exponentiable objects $p \colon \cC \to \cB$ in $\Cat_{\infty/\cB}$, which we now briefly recall.
Given a morphism $f$ in a category $\cC$ and any factorization $p(f) \simeq g \circ h$ of $p(f)$ in $\cB$, viewed as a functor $[2] \to \cB$, we may form the $\infty$-category
\[\mathrm{Fact}(f \mid g \circ h) \coloneqq \Fun_{/\cB}([2],\cC) \times_{\Fun_{/\cB}([1],\cC)} \{f\}\]
of factorizations of $f$ lying over the factorization $p(f) = g \circ h$. The Conduché criterion (see \cite[Proposition B.3.2]{HA} or \cite[Lemma 2.2.8]{AyalaFrancis2020FibrationsInftyCategories}) now states that $p \colon \cC \to \cB$ is exponentiable in $\Cat_{\infty/\cB}$ if and only if for any such $f$, $g$ and $h$, the $\infty$-category $\mathrm{Fact}(f \mid g \circ h)$ is weakly contractible.

Our main result is that an object $\cP \to \cO$ in $\Algd(\cO)$ is exponentiable if the Conduché criterion holds for a specific class of factorizations in $\cO$.

\begin{thmintro}\label{thmA}
    Let $\cO$ be an algebraic pattern and $\cP \to \cO$ an algebrad. Then $\cP$ is exponentiable, as an object in $\Algd(\cO)$, if the following  condition is satisfied:
    \begin{enumerate}[(CC)]
        \item\label{CondCrit} for any composable pair of active morphisms
            \[\begin{tikzcd}[cramped]
            	x & y & e
            	\arrow["h", squiggly, from=1-1, to=1-2]
            	\arrow["g", squiggly, from=1-2, to=1-3]
            \end{tikzcd}\]
            in $\cO$ such that $e$ is an elementary object, and any lift $f$ of $g \circ h$, the $\infty$-category $\mathrm{Fact}(f \mid g \circ h)$ is weakly contractible. 
    \end{enumerate}
\end{thmintro}

In particular, if $\cP \to \cO$ is exponentiable in $\Cat_{\infty/\cO}$, then it is also exponentiable in $\Algd(\cO)$.
For $\infty$-operads, \cref{thmA} was also obtained by Hinich \cite[\S 2.8]{Hinich2020YonedaLemmaEnriched}, where he calls $\infty$-operads satisfying the condition \cref{CondCrit} \emph{flat}.
However, even in this case our proof is completely different and in particular does not rely on the model-categorical arguments from \cite[Appendix B]{HA}; see \cref{intro:other results} for a detailed comparison between our proof and previous proof strategies.

A natural follow-up question is whether this criterion is also necessary; that is, whether \cref{thmA} can be upgraded to an ``if and only if'' statement. 
Under a condition on $\cO$ called \emph{robustness} (see \cref{def:new robust}), our methods will show that this is indeed the case.

\begin{thmintro}\label{thmB}
    Let $\cO$ be a robust algebraic pattern. Then any exponentiable object in $\Algd(\cO)$ satisfies condition \cref{CondCrit} from \cref{thmA}.
\end{thmintro}

In \cref{subsec:examples-robust}, we show that there exist robust algebraic patterns $\cO$ such that $\Algd(\cO)$ is equivalent to the $\infty$-categories of
\begin{itemize}
    \item $\infty$-operads,
    \item equivariant $\infty$-operads, and
    \item virtual double $\infty$-categories.
\end{itemize}
For all of these $\infty$-categories, we then obtain a complete characterization of their exponentiable objects, which will be discussed in \cref{subsection:examples of CC} and \cref{subsection:examples of CC converse}. We could not find a proof of similar characterizations in the current literature; see \cref{intro:other results} for a thorough overview of what was known before.

\begin{remark}
    \cref{thmA} also provides a criterion for when a morphism $\mathscr{Q} \to \cP$ in $\Algd(\O)$ is exponentiable.
    To see this, note that by a proof similar to \cite[Corollary 4.1.17]{BarkanHaugsengSteinebrunner}, the category $\cP$ admits an algebraic pattern structure such that $\Algd(\cO)_{/\cP} \simeq \Algd(\cP)$.
    From this it follows that $\mathscr{Q} \to \cP$ is exponentiable in $\Algd(\cO)$ if and only if it is exponentiable as an object in $\Algd(\cP)$, to which we can apply \cref{thmA}.
    In fact, in our proof of \cref{thmA} we directly characterize the exponentiable morphisms in $\Algd(\cO)$, and not just the exponentiable objects.
\end{remark}

\subsection{The underlying graph of an exponential object}

An important feature of Day convolution for symmetric monoidal $\infty$-categories is that the underlying $\infty$-category of the $\infty$-operad $[\cC^\otimes,\cD^\otimes]$ is simply the functor category $\Fun(\cC,\cD)$.
This fact generalizes to algebrads.

To any algebrad $\cP \to \cO$ over an algebraic pattern $\cO$, one can assign its \emph{underlying graph}, which is a functor $\Gamma\cP \colon \cO^\elem \to \Cat_\infty$ generalizing the notion of \emph{underlying $\infty$-category} for $\infty$-operads.
The $\infty$-category $\Fun(\cO^\elem,\Cat_\infty)$ is cartesian closed, and it turns out exponential objects in $\Algd(\cO)$ are compatible with the internal hom $[-,-]$ of $\Fun(\cO^\elem,\Cat_\infty)$.
More precisely, we show the following.

\begin{thmintro}\label{thmC}
    Let $\cO$ be an algebraic pattern and let $\cP$ be exponentiable in $\Algd(\cO)$.
    Then for any algebrad $\cQ \to \cO$, the canonical comparison map
    \[\Gamma [\cP,\cQ] \to [\Gamma \cP, \Gamma \cQ]\]
    is an equivalence.
\end{thmintro}

\subsection{Algebrads as complete Segal presheaves on the tree category}

Our proof of \cref{thmA} and \cref{thmB} uses an alternative description of the $\infty$-category $\Algd(\cO)$ of $\cO$-algebrads which is of independent interest.
Namely, we show that to any algebraic pattern, one can associate a certain $\infty$-category of \emph{layered trees} $\Tree[\cO]$ such that $\Algd(\cO)$ is equivalent to the $\infty$-category of \emph{complete Segal} presheaves on $\Tree[\cO]$.\footnote{Our choice of terminology and notation is based on the theory of dendroidal sets, introduced in \cite{MoerdijkWeiss2007DendroidalSets} as a model for $\infty$-operads. However, applying our construction to the category $\F_*$ of finite pointed sets yields the category $\Tree[\F_*]$ of \emph{layered} trees, which admits a map to the Moerdijk--Weiss tree category $\Omega$ but is not equivalent to it.}
The objects in $\Tree[\cO]$ are strings $t_0 \actarrow \cdots \actarrow t_n$ of active morphisms in $\cO$ such that $t_n$ is elementary.
A morphism between two such strings $t_0 \actarrow \cdots \actarrow t_n$ and $s_0 \actarrow \cdots \actarrow s_m$ consists of a map $\phi \colon [n] \to [m]$ in $\Delta$ together with a diagram
\[\begin{tikzcd}
	{t_0} & {t_1} & \cdots & {t_{n-1}} & {t_n} \\
	{s_{\phi(0)}} & {s_{\phi(1)}} & \cdots & {s_{\phi(n-1)}} & {s_{\phi(n)}}
	\arrow[squiggly, from=1-1, to=1-2]
	\arrow[squiggly, from=1-2, to=1-3]
	\arrow[squiggly, from=1-3, to=1-4]
	\arrow[squiggly, from=1-4, to=1-5]
	\arrow[tail, from=2-1, to=1-1]
	\arrow[squiggly, from=2-1, to=2-2]
	\arrow[tail, from=2-2, to=1-2]
	\arrow[squiggly, from=2-2, to=2-3]
	\arrow[squiggly, from=2-3, to=2-4]
	\arrow[tail, from=2-4, to=1-4]
	\arrow[squiggly, from=2-4, to=2-5]
	\arrow[tail, from=2-5, to=1-5]
\end{tikzcd}\]
in $\cO$ whose vertical morphisms are inert.
A presheaf on $\Tree[\cO]$ is called \emph{Segal} if for any string $t_0 \actarrow \cdots \actarrow t_n$, the Segal map
\[X(t_0 \actarrow \cdots \actarrow t_n) \to X(t_0 \actarrow t_1) \times_{X(t_1)} \cdots \times_{X(t_{n-1})} X(t_{n-1} \actarrow t_n)\]
is an equivalence. We say that such an $X$ is \emph{complete} if for any elementary object $e$ in $\cO$, the Segal space
\[[n] \mapsto X(\underbrace{e = \cdots = e}_{n\text{--times}})\]
is complete.
We write $\CSeg(\Tree[\cO]) \subseteq \PSh(\Tree[\cO])$ for the full subcategory of complete Segal presheaves.

\begin{thmintro}\label{thmD}
    Let $\cO$ be an algebraic pattern.
    Then there is a natural equivalence
    \[\Algd(\cO) \simeq \CSeg(\Tree[\cO]).\]
\end{thmintro}

Concretely, the equivalence $\Algd(\cO) \to \CSeg(\Tree[\cO])$ is given by a \emph{nerve} construction with respect to a certain functor $\Tree[\cO] \to \Algd(\cO)$; see \cref{ssec:unraveling} for details.
Similar equivalences were studied by Barwick \cite{Barwick} in the context of operator categories and by Kern \cite{Kern} for a special class of algebraic patterns called \emph{combinatorial} algebraic patterns.
Furthermore, Felix Naß has an unpublished proof of this equivalence in the case of soundly extendable patterns that uses different methods.

In practice, we will demonstrate \cref{thmD} by passing through an auxiliary gadget. Using Juran's  double $\infty$-categorical perspective on factorization systems \cite{BrankoJuran}, we may equivalently view every algebraic pattern $\O$ as a particular double $\infty$-category $\O^\mathrm{dbl}$ with a distinguished class of elementary objects. This intermediate double $\infty$-categorical point of view on algebraic patterns will turn out to be convenient to simplify arguments and formulate statements throughout the paper. For instance, we may construct the tree $\infty$-category $\Tree[\O]$ by unstraightening the underlying functor $\Delta^\op \to \Cat$ of $\O^{\mathrm{dbl}}$.

The $\infty$-category $\Tree[\cO]$ behaves quite similarly to the simplex category $\Delta$.
This will allow us to prove \cref{thmA} and \cref{thmB} by a strategy similar to that of Ayala--Francis \cite[Lemma 2.2.8]{AyalaFrancis2020FibrationsInftyCategories}.

In fact, we will demonstrate that $\Tree[\O]^\op$ can again be given the structure of an algebraic pattern. The complete Segal presheaves on $\Tree[\O]$ are precisely the Segal objects for this pattern, in the sense of Chu--Haugseng \cite{ChuHaugseng2021HomotopycoherentAlgebraSegal}, that admit an additional univalence condition. In particular, this gives a way to iterate the tree construction (see \cref{ex:iterated trees}) and to produce a tower of (non-full) inclusions 
$$
\Algd(\O) \to \Algd(\Tree[\O]^\op) \to \Algd(\Tree^2[\O]^\op) \to \dotsb.
$$

\subsection{Relation to other results}\label{intro:other results}
Exponentiability in the $\infty$-category of $\infty$-operads was studied by Lurie \cite[\S 2.2.6]{HA} and Hinich \cite[\S 2.8]{Hinich2020YonedaLemmaEnriched}, who proved \cref{thmA} in this setting.
This was subsequently generalized by Nardin--Shah \cite[\S 3]{NardinShah}, who proved a version of \cref{thmA} in the setting of $G$-$\infty$-operads.
In \cite{ChuHaugseng2023FreeAlgebrasDay}, Chu--Haugseng prove a version of Day convolution for \emph{cartesian} algebraic patterns $\cO$ when the source is an $\cO$-monoidal $\infty$-category and the target is the $\infty$-category of spaces.
Our \cref{thmA} generalizes all these cases.
Moreover, to our knowledge there is no place in the literature that establishes \emph{necessary} conditions for exponentiability (our \cref{thmB}). 
We will now compare our proof strategy to that of Lurie, Hinich and Nardin--Shah and explain why the same strategy cannot be used to prove a result at the level of generality of \cref{thmA}.

The proof strategies of Nardin--Shah and Hinich are essentially the same as that of Lurie: they verify that the conditions of \cite[Theorem B.4.2]{HA} are satisfied.
Roughly, this goes as follows:
Given a map of $\infty$-operads $f \colon \cQ \to \cP$, one obtains a pullback functor
\[f^* \colon \Cocart^\inert(\cP) \to \Cocart^\inert(\cQ)\]
where $\Cocart^\inert(-)$ denotes the sub-$\infty$-category of $(\Cat_{\infty})_{/-}$ spanned by functors that admit cocartesian lifts of inerts, and maps between them that preserve these.
One easily observes that $f^*$ takes $\infty$-operads over $\cP$ to $\infty$-operads over $\cQ$.
Moreover, the functor $f^*$ admits a right adjoint $f_* \colon \Cocart^\inert(\cQ) \to \Cocart^\inert(\cP)$ if and only if $\cQ^\activ \to \cP^\activ$ is exponentiable (cf.\ \cref{prop:expo-inert-fibrations}).
The exponentiability of $f \colon \cQ \to \cP$ is then proved by verifying that this right adjoint $f_*$ takes $\infty$-operads over $\cQ$ to $\infty$-operads over $\cP$.

Observe that this strategy can only work if $\cQ^\activ \to \cP^\activ$ is exponentiable, because otherwise the right adjoint $f_* \colon \Cocart^\inert(\cQ) \to \Cocart^\inert(\cP)$ does not exist.
However, our main result \cref{thmA} only requires the Conduché criterion for pairs of maps $x \actarrow y \actarrow e$ that end with an elementary object, not for all maps in $\cP^\activ$.
In the case of $\infty$-operads, these conditions are equivalent by \cite[Lemma 2.8.2]{Hinich2020YonedaLemmaEnriched}, but Hinich's argument does not go through for most other algebraic patterns.
In \cref{subsection:non-triv exp map between vdc}, we give an explicit counterexample of a map between virtual double $\infty$-categories $\cQ \to \cP$ that is exponentiable but for which $\cQ^\activ \to \cP^\activ$ is not exponentiable in $\Cat$.
In particular, the proof strategy of Lurie, Nardin--Shah and Hinich cannot be used to characterize all exponentiable morphisms between virtual double $\infty$-categories.
(It is also worth pointing out that condition (5) of \cite[Theorem B.4.2]{HA} does not hold for most algebraic patterns.)

Instead, our strategy is to write $\Algd(\cO)$ as the localization $\CSeg(\Tree[\cO])$ of a presheaf $\infty$-category (\cref{thmD}).
Since any morphism $f \colon X \to Y$ in a presheaf $\infty$-category is exponentiable, we lose the requirement that $\cQ^\activ \to \cP^\activ$ is exponentiable this way.
We then prove \cref{thmA} by studying when the right adjoint $f_* \colon \PSh(\Tree[\cO])_{/X} \to \PSh(\Tree[\cO])_{/Y}$ preserves complete Segal objects.

Finally, let us mention that similar results exist in the strict categorical setting. Day introduced \textit{promonoidal symmetric monoidal categories} in \cite{Day1970ClosedCategoriesFunctors}. These structures were characterized by Pisani as operads that satisfy a 1-categorical version of condition \cref{CondCrit}, and moreover, Pisani showed these are precisely the exponentiable operads, thereby proving a version of \cref{thmB} for ordinary operads \cite{Pisani2014SeqMulticategories}. In the setting of ordinary virtual double categories, it was recently shown by Arkor \cite{Arkor} that all pseudo double categories are exponentiable as virtual double categories. More generally, a characterization of all exponentiable virtual double categories was recently announced by Thompson \cite{EaEThompson}.
The $\infty$-categorical versions of these statements are a consequence of \cref{thmA} and \cref{thmB}, see also \cref{exa:double-cat-is-expo}.
We suspect that versions of \cref{thmA} and \cref{thmB} in the strict setting can be deduced by restricting to truncated objects in $\Seg(\Tree[\cO])$ or $\CSeg(\Tree[\cO])$, but we will not pursue this here.

\subsection{Outline of the paper}

In \cref{sec:Algebrads}, we recall some basic facts on algebraic patterns and discuss our main examples.
We then reformulate this theory in terms of double $\infty$-categories, using Juran's \cite{BrankoJuran} equivalence between $\infty$-categories equipped with a factorization system and factorization double $\infty$-categories.
In \cref{sec:forest-and-tree-cats}, we define the category of trees $\Tree[\cO]$ for any algebraic pattern and study its basic properties.
Here we also define robustness for algebraic patterns and show that most of our examples satisfy this property.
Combining the results from \cref{sec:Algebrads} and \cref{sec:forest-and-tree-cats}, we establish \cref{thmD} in \cref{sec:algebrads_as_segal_presheaves_on_the_tree_category}.
We then provide sufficient conditions for the exponentiability of objects in $\CSeg(\Tree[\cO])$ in \cref{sec:exponentiable-objects}, exploiting the simplicial nature of the category $\Tree[\cO]$.
Using \cref{thmD}, we translate this characterization back to the setting of algebrads over $\cO$ and establish \cref{thmA}.
In \cref{sec:Underlying-graph}, we study the underlying graphs of our exponential objects and prove \cref{thmC}.
After that, we shift our attention to proving \cref{thmB}. As a preliminary step, we introduce \textit{robust} algebraic patterns in \cref{section:robust patterns}. We then demonstrate \cref{thmB} in \cref{section:necessity}. We work out condition (CC) for various examples of patterns in \cref{subsection:examples of CC,subsection:examples of CC converse}.
We conclude the paper with the short \cref{sec:examples} where we discuss a few examples of exponentiable $\cO$-algebrads.
In particular, we show that to any $\infty$-category $\cC$ one can associate its \emph{virtual} cospan double $\infty$-category $\DCospan(\cC)$ (even if $\cC$ does not admit pushouts) and that $\DCospan(\cC)$ is always an exponentiable object in the category of virtual double $\infty$-categories.

\subsection{Future work} In a planned future work, we aim to use the results of this paper to prove a universal property of mapping out of virtual double $\infty$-categories of cospans.
The goal is to establish the same property as was obtained by Dawson--Par\'e--Pronk in the strict setting \cite{DawsonParePronk}. Our strategy involves exponentiating double $\infty$-categories by virtual cospan double $\infty$-categories. As a stepping stone, we already show in \cref{subsection:cospans} that virtual cospan double $\infty$-categories are indeed exponentiable.

The latter two authors are developing the basic aspects of category theory for algebrads in a sequel to this work. The Day convolution of algebrads proven here, can be used to define presheaf algebrads. One may then show a suitable version of the Yoneda lemma, and develop a theory of Kan extensions and cocompletions for algebrads. When specializing the theory to the pattern describing operads, the operadic Kan extensions of Lurie \cite[\S 3.1]{HA} will be recovered.

\subsection*{Acknowledgments}
The authors are grateful to the Max Planck Institute for Mathematics in Bonn for its hospitality and support during the writing of this paper.

\subsection*{Conventions}\label[subsecconventions]{conventions}
Throughout the rest of this article, we will make use of the following conventions:
\begin{enumerate}
    \item From now on, we will drop the `$\infty$' symbol from our notation. For example, we will refer to $\infty$-categories as categories, to $\infty$-operads as operads and to $(\infty,n)$-categories as $n$-categories.
    \item The category of spaces or ($\infty$-)groupoids is denoted by $\cS$.
    \item If $\C$ is a category, then we will write $\PSh(\C) \coloneqq \Fun(\C^\op, \cS)$ for the category of presheaves on $\C$.
    \item If $\C$ is a category, then we will write $\Ar(\C) \coloneqq \Fun([1],\C)$ for the category of arrows in $\C$.
    \item If $x$ and $y$ are objects of a category $\C$, then we write $\Hom_\C(x,y)$ or $\Hom(x,y)$ for the space of maps from $x$ to $y$.
    \item If $\cC$ is a category that admits cartesian products and $c$ is an object in $\cC$ such that $c \times -$ admits a right adjoint, then we will denote this right adjoint by $[c,-]$ and call it an \emph{internal hom-object} or \emph{exponential object}.
    \item Given a sequence $t = (t_0 \to \cdots \to t_n)$ of composable maps in a category, we will write $t_{i-1}$ for the $i$-th object of this sequence. Given $0 \leq i \leq j \leq n$, we will write $t_{i,j}$ for the subsequence $(t_i \to \cdots \to t_j)$, and we write $t_{\leq i}$ for $(t_{0} \to \cdots \to t_{i})$ and $t_{\geq j}$ for $(t_j \to \cdots \to t_n)$.
    \item Suppose that we are given a commutative square
    \[
        \begin{tikzcd}
            a\arrow[r]\arrow[d] & b \arrow[d] \\
            c \arrow[r] & d
        \end{tikzcd}
    \]
    in a category $\C$. If $\C$ has pullbacks, then the unique map $a \to b\times_d c$ will be called the \textit{gap} map associated to the square. Dually, if $\C$ has pushouts, then the unique map $b \cup_a c \to d$ will be called the \textit{cogap} map associated to the square.
\end{enumerate}

\section{Algebrads}\label{sec:Algebrads}

We commence by discussing the notions that play a key role throughout this article.

\subsection{Recollections on algebraic patterns}

We briefly recall the theory of algebraic patterns. For details, the reader is referred to \cite{ChuHaugseng2021HomotopycoherentAlgebraSegal,BarkanHaugsengSteinebrunner}.

\begin{definition}
    An \textit{algebraic pattern} is a category $\cO$ equipped with a factorization system $(\cO^\inert,\cO^\act)$ and a full subcategory $\cO^\el$ of $\cO^\inert$. We will write $\AlgPatt$ for the category of algebraic patterns.
\end{definition}

\begin{notation}
    We will use the following notation and terminology:
    \begin{itemize}
        \item The morphisms in  $\cO^\inert$ and $\cO^\activ$ will be called \emph{inert} and \emph{active} and often be denoted by $\intarrow$ and $\actarrow$, respectively.
        \item The objects in $\cO^\elem$ are called \emph{elementary}.
        \item Given an object $x$ in $\cO$, we will write $\cO^\elem_{x/}$ for the fiber product $\cO^\elem \times_{\cO} \cO_{x/}$.
        \item The full subcategories of $\Ar(\cO) $ spanned by the inert and active morphisms will be denoted $\Ar_\inert(\cO)$ and $\Ar_\activ(\cO)$, respectively.
    \end{itemize}
\end{notation}

\begin{definition}[Chu--Haugseng]
    \label{defi:of algebrad for factactization system}
   Let $\cO$ be an algebraic pattern. A functor $p \colon \P \to \O$ is called an \emph{$\cO$-algebrad}, or \emph{algebrad} for short, if the following conditions hold:
   \begin{enumerate}[(1)]
       \item $p$ has cocartesian lifts of inert morphisms in $\O$, 
       \item for all $x\in \O$, the canonical functor 
      $$\P_x\to \lim_{(x\intarrow e)\in \O^{\el}_{x/}}\P_e$$
      is an equivalence,
       
     \item\label{defi:algd fact:item3} for any $x,y$ in $\cO$, any $\overline x \in \cP_x$ and any $\overline y \in \cP_y$, the square
    \[\begin{tikzcd}
    	{\Map_\cP(\overline x, \overline y)} & {\lim\limits_{(\phi \colon y \intarrow e) \in \cO^\elem_{y/}}\Map_\cP(\overline x,\phi_! \overline y)} \\
    	{\Map_\cO(x,y)} & {\lim\limits_{(\phi \colon y \intarrow e) \in \cO^\elem_{y/}} \Map_\cO(x,e)}
    	\arrow[from=1-1, to=1-2]
    	\arrow[from=1-1, to=2-1]
    	\arrow[from=1-2, to=2-2]
    	\arrow[from=2-1, to=2-2]
    \end{tikzcd}\]
    is a pullback square.
   \end{enumerate}
\end{definition}

\begin{remark}
In the work of Chu and Haugseng, {algebrads} were originally introduced as \textit{weak Segal fibrations} \cite[Definition 9.6]{ChuHaugseng2021HomotopycoherentAlgebraSegal}. However, in what follows we will give several equivalent definitions of algebraic patterns and of algebrads. In these other contexts, it seemed to us that the original terminology was less suitable. We therefore chose to modify it, and the term \textit{algebrad} seemed to us to be an appropriate choice, as it reflects that this notion is a generalization of an \textit{operad}.
\end{remark}

\begin{definition}\label{def:category-of-algebrads}
    If $\cP \to \cO$ is an algebrad, then $\cP$ inherits the structure of an algebraic pattern whose inert morphisms are the cocartesian lifts of inert morphisms of $\O$. The active morphisms (resp.\ elementary objects) are the ones lying over active morphisms  (resp.\ elementary objects) of $\O$. This was explained in \cite[Lemma 9.10]{ChuHaugseng2021HomotopycoherentAlgebraSegal}. We write $$\Algd(\cO)\subset \AlgPatt_{/\cO}$$ for the full subcategory spanned by the $\cO$-algebrads in this sense.
\end{definition}

\begin{example}\label{examples:examples of patterns}
    We have the following examples of algebrads:
    \begin{itemize}
        \item \textit{Operads}, as introduced by Lurie \cite{HA}, are precisely the algebrads for the algebraic pattern structure on $\F_*$ with the inerts and actives as in \cite[\S 2.1.1]{HA} and elementaries given by $\{\left<1\right>\}$. We will denote this algebraic pattern by $\F_*^\flat$.
        \item If we add the additional elementary $\left<0\right>$ to $\F_*$, then we obtain a pattern which we will denote by $\F_*^\natural$. The algebrads for this pattern are \textit{generalized operads} \cite[\S 2.3.2]{HA}.
        \item The category $\Delta^\op$ can be given the structure of an algebraic pattern where the inert morphisms are the inclusions of convex subsets, the actives are the morphisms that preserve the minimal and maximal elements, and the only elementary object is $[1]$. Algebrads for this pattern are \textit{non-symmetric operads} (see \cite[Definition 3.1.3]{GepnerHaugseng2015EnrichedCategoriesNonsymmetric} and \cite[Definition 4.1.3.2]{HA}).
        This algebraic pattern structure will be denoted $\Delta^{\op,\flat}$.
        \item Choosing the elementaries of $\Delta^\op$ to be $\{[0],[1]\}$ instead, its algebrads are known as \textit{virtual double categories} (called \textit{generalized non-symmetric operads} in \cite{GepnerHaugseng2015EnrichedCategoriesNonsymmetric}).
        This pattern structure will be denoted $\Delta^{\op,\natural}$ to distinguish it from the previous one.
        \item Let $\F_G$ be the category of finite $G$-sets for $G$ a finite group. The (2,1)-category $\Span(\F_G)$ of spans in $\F_G$ admits the structure of an algebraic pattern whose inert and active maps are of the form 
        \[
            \begin{tikzcd}[column sep = tiny, row sep = tiny]
                & X\arrow[dl]\arrow[dr,"\cong"] \\ 
                Y & & X'
            \end{tikzcd} ~~~~ \text{and} ~~~~
             \begin{tikzcd}[column sep = tiny, row sep = tiny]
                & X\arrow[dl,"\cong"']\arrow[dr] \\ 
                X' & & Y
            \end{tikzcd}
        \]respectively, where the $\cong$-marked maps are isomorphisms. Its elementary objects are the transitive $G$-sets, i.e.\ the \textit{$G$-orbits}. Then as remarked in \cite[Observation  5.2.12]{BarkanHaugsengSteinebrunner}, the algebrads for this pattern deserve to be called \textit{$G$-operads}. We will write $\Span(\F_G)^\flat$ for this pattern structure.
        \item We will also consider the pattern structure $\Span(\F)^\natural$ on $\Span(\F)$ with the same actives and inerts as above (taking $G$ to be trivial), and where the elementary objects are the finite sets of cardinality at most $1$.
        We will see in \cref{exa:CC-necessary-genoperad} below that the inclusion $\F_* \hookrightarrow \Span(\F)$ induces an equivalence $\Algd(\F_*^\natural) \simeq \Algd(\Span(\F)^\natural)$, so algebrads for this pattern also describe generalized operads. (This also follows by applying \cite[Theorem A]{BarkanHaugsengSteinebrunner} to $\F_*^\natural \hookrightarrow \Span(\F)^\natural$.)
    \end{itemize}
\end{example}

\begin{remark}\label{remark:algebrads-all-elementary}
    If every object in $\cO$ is elementary, then $\cP \to \cO$ is an $\cO$-algebrad precisely if it admits cocartesian lifts of inerts.
    In this case $\Algd(\cO)$ is equivalent to the subcategory $\Cocart^\inert(\cO) \hookrightarrow \Cat_{/\cO}$ whose objects are the functors $\cC \to \cO$ that admit cocartesian lifts over $\cO^\inert$ and whose morphisms are those functors that preserve cocartesian lifts over $\cO^\inert$.
\end{remark}

We also recall the following notion from \cite{ChuHaugseng2021HomotopycoherentAlgebraSegal}:

\begin{definition}\label{def:segal objects}
Let $\C$ be a limit complete category. A functor $\P : \O \to \C$
is called a \textit{Segal $\O$-object} in $\C$ if the canonical map 
$$
\P(x) \to \lim_{(x\intarrow e) \in \O^{\el}_{x/}} \P(e)
$$
is an equivalence. We will write $\Seg(\O, \C) \subset \Fun(\O,\C)$ for the full subcategory spanned by the Segal $\O$-objects.
If $\C = \Cat$, then Segal $\O$-objects are referred to as \textit{Segal $\O$-categories}.
\end{definition}

It follows from \cite[Remark 9.11]{ChuHaugseng2021HomotopycoherentAlgebraSegal} that the Grothendieck construction $\Fun(\O,\Cat) \to \Cat_{/\O}$ restricts to a functor 
$\Seg(\O, \Cat) \to \Algd(\O)$ that selects the (non-full) subcategory spanned by the algebrads $\P \to \O$ that are cocartesian fibrations, and maps between algebrads that preserve cocartesian arrows.

\begin{example}
    We have the following examples of Segal categories:
    \begin{enumerate}
        \item \textit{Symmetric monoidal categories} are the Segal categories for the pattern $\F_*$.
        \item \textit{Monoidal categories} are the Segal categories for the pattern $\Delta^{\op, \flat}$.
        \item \textit{Double categories} are the Segal categories for the pattern $\Delta^{\op, \natural}$; see also \cref{def:double category}.
        \item \textit{Symmetric monoidal $G$-categories} are the Segal categories for the pattern $\Span(\F_G)^\flat$; see \cite[\S 2.3]{NardinShah} and \cite[\S 5.2]{BarkanHaugsengSteinebrunner}.
    \end{enumerate}
\end{example}

\subsection{Algebraic patterns as double categories}
\label{section:algebraic pattern}

In this subsection, we will introduce a double categorical perspective on the previous theory. 

\begin{definition}\label{def:double category}
A \textit{double category} is a functor $\P \colon \Delta^{\op}\to \Cat$ such that for any $n$, the functor 
$$\P_n\to \P_1\times_{\P_0}\dotsb\times_{\P_0}\P_1$$
is an equivalence.
A double category $\cP$ will be called \emph{complete} if moreover the functor
\[\cP_0 \to \cP_3 \times_{\cP_{\{0 \leq 2\}} \times \cP_{\{1 \leq 3\}}} (\cP_0 \times \cP_0)\]
is an equivalence.
The objects and morphisms of $\P_0$ are called the \textit{objects} and \textit{vertical morphisms} of $\P$. The objects and morphisms of $\P_1$ are referred to as the \textit{horizontal morphisms} and \textit{2-cells} of $\P$.
The full subcategory of $\Fun(\Delta^\op,\Cat)$ of double categories will be denoted $\DblCat$.
\end{definition}

\begin{example}
    Let $\C$ be a category. The \textit{double category $
    \Sq(\C)$ of squares in $\C$} is defined by 
    $$
    \Sq(\C)_n \coloneqq \Fun([n],\C).
    $$
\end{example}

We recall the following definition of Juran \cite{BrankoJuran}:

\begin{definition}[Juran]
A complete double category $\P$ is called a \emph{factorization double category} if the target map $t \colon \P_1 \to \P_0$ is a left fibration. In other words, if every solid diagram as pictured below
\[\begin{tikzcd}
	x & y \\
	{y'} & z
	\arrow[""{name=0, anchor=center, inner sep=0}, from=1-1, to=1-2]
	\arrow[dashed, from=1-1, to=2-1]
	\arrow[from=1-2, to=2-2]
	\arrow[""{name=1, anchor=center, inner sep=0}, dashed, from=2-1, to=2-2]
	\arrow[between={0.2}{0.8}, Rightarrow, dashed, from=0, to=1]
\end{tikzcd}\]
can be extended to the dashed 2-cell in an essentially unique way. 
\end{definition}

\begin{remark}
We have swapped the roles of the vertical and the horizontal morphisms with respect to the definition of a factorization double category given in \cite{BrankoJuran}.
This convention will be important in \cref{sec:forest-and-tree-cats} when we define the forest and tree categories of an algebraic pattern.
\end{remark}

\begin{remark}\label{rem:slices-factorization-double-category}
If $\cP$ is a factorization double category, then the Segal condition implies that for every $n \geq 1$, the target map $\cP_n \to \cP_0$ induced by $\{n\} \hookrightarrow [n]$ is a left fibration.
In particular, for any $t = (t_0 \to t_1 \to \cdots \to t_n)$ in $\cP_n$, the target projection $(\cP_n)_{t/} \to (\cP_0)_{t_n/}$ is an equivalence.
\end{remark}

The name for these structures is inspired by the following result.

\begin{construction}
    Let $(\C, \C_L, \C_R)$ be a factorization system. We denote by  $\Sq_{L,R}(\C)$ the  sub double category of  $\Sq(\C)$ whose vertical arrows are contained in $\C_L$ and whose horizontal arrows are contained in $\C_R$.
\end{construction}

\begin{proposition}[Juran]\label{prop:factorization system and double factorysation caregory}
    The construction $(\C, \C_L,\C_R) \mapsto \Sq_{L,R}(\C)$ defines a fully faithful functor $\Sq_{L,R} \colon \Fact \to \DblCat$, where $\Fact$ is the category of factorization systems.
    Its essential image consists precisely of the factorization double categories.
\end{proposition}

\begin{proof}
This is \cite[Theorem 3.19]{BrankoJuran}.
\end{proof}

This motivates the following definition.

\begin{definition}\label{def:double-pattern}
    A \textit{double pattern} $\O$ is a factorization double category together with a full subcategory $\cO_0^\elem \subset \cO_0$, called the \textit{elementaries}. In this case, the vertical arrows of $\O$ are called \emph{inert} and denoted $\intarrow$, and the horizontal arrows are called \emph{active} and denoted $\actarrow$.
    
    For any $n \geq 0$, we define 
    $\O_n^{\el}$ as the full subcategory of $\O_n$ fitting in the pullback
    \[\begin{tikzcd}
    	{\O_n^{\el}} & {\O_n} \\
    	{\O_{0}^{\el}} & {\O_{0}},
    	\arrow[from=1-1, to=1-2]
    	\arrow[from=1-1, to=2-1]
    	\arrow[from=1-2, to=2-2]
    	\arrow[from=2-1, to=2-2]
    \end{tikzcd}\]
    where the vertical maps are induced by the inclusion $\{n\} \hookrightarrow [n]$.
    The pair $(\cO_n,\cO_n^\elem)$ defines an algebraic pattern where every map is inert.
\end{definition}

\begin{remark}\label{remark:elementary-slices-On}
    It follows immediately that for any $t$ in $\cO_n$, the equivalence $(\cO_n)_{t/} \to (\cO_0)_{t_n/}$ from \cref{rem:slices-factorization-double-category} restricts to an equivalence $(\cO_n^\el)_{t/} = (\cO_0^\el)_{t_n/}$.
\end{remark}

The following is a direct consequence of \cref{prop:factorization system and double factorysation caregory}:

\begin{corollary}\label{cor:equivalence between double and algebraic pattern}
    The functor $\Sq_{L,R}$ induces an equivalence between the category $\AlgPatt$ of algebraic patterns and the category $\DblPatt$ of double patterns. \qed
\end{corollary}

\begin{notation}
    We will denote this equivalence by
    \[(-)^\dbl : \AlgPatt \simeq \DblPatt : (-)^\alg.\]
\end{notation}

\subsection{Algebrads in the double categorical context}\label{subsec:double-algebrads}

Let us start with defining the analogue of algebrads in the setting of double patterns.

\begin{definition}
    \label{defi:of algebrad for double pattern}
    Let $\cO$ be a double pattern. A functor $p \colon \cP \to \cO$ is called an \emph{$\cO$-algebrad}, or \emph{algebrad} for short, if
    \begin{enumerate}[(1)]
        \item for every $n \geq 0$, the functor $p_n \colon \cP_n \to \cO_n$ is a left fibration,
        \item for any $t$ in $\O_0$, the morphism 
        $$(\P_0)_t\to \lim_{(t\intarrow s)\in (\O_0)^{\el}_{t/}}(\P_0)_s$$
        is an equivalence,
        \item\label{item3:defi-double-algebrad} for any $t$ in $\cO_1$, the square
        $$
        \begin{tikzcd}
        	{(\cP_1)_t} & {\lim\limits_{(t \intarrow s) \in(\cO_1)^\elem_{t/}}(\cP_1)_s} \\
        	{(\cP_0)_{t_0}} & {\lim\limits_{(t \intarrow s) \in(\cO_1)^\elem_{t/}}(\cP_{0})_{s_0}}
        	\arrow[from=1-1, to=1-2]
        	\arrow["{d_1}"', from=1-1, to=2-1]
        	\arrow["{d_1}", from=1-2, to=2-2]
        	\arrow[from=2-1, to=2-2]
        \end{tikzcd}$$
        is cartesian.
    \end{enumerate}
\end{definition}

\begin{remark}
    To see how this square is constructed, let $P_n \colon \cO_n \to \cS$ denote the functors classifying  $p_n \colon \cP_n \to \cO_n$.
    The commutative square
    \[\begin{tikzcd}
    	{\cP_1} & {\cP_0} \\
    	{\cO_1} & {\cO_0}
    	\arrow["{d_1}", from=1-1, to=1-2]
    	\arrow["{p_1}"', from=1-1, to=2-1]
    	\arrow["{p_0}", from=1-2, to=2-2]
    	\arrow["{d_1}", from=2-1, to=2-2]
    \end{tikzcd}\]
    then induces maps $(\cP_1)_{t} = P_1(t) \to P_0(d_1(t)) = (\cP_0)_{t_0}$ that are natural in $t \in \cO_1$.
    The horizontal maps in the square of the condition $(3)$ of \cref{defi:of algebrad for double pattern} are obtained by restricting and right Kan extending along $\cO^\elem_1 \hookrightarrow \cO_1$.
\end{remark}

\begin{definition}
    If $\cP \to \cO$ is an algebrad, then $\cP$ inherits the structure of a double pattern where $\cP^\elem$ consists of the objects lying over an elementary in $\cO$. We write $$\Algd(\cO)\subset \DblPatt_{/\cO}$$ for the full subcategory spanned by the $\cO$-algebrads in this sense.
\end{definition}

We will now compare algebrads for double patterns and algebrads for algebraic patterns.

\begin{lemma}\label{lem:algebrad-n=1-suffices}
    Let $p \colon \cP \to \cO$ be an algebrad. Then for every $n \geq 1$ and $t$ in $\cO_n$, the square
    \[\begin{tikzcd}
    	{(\cP_n)_t} & {\lim\limits_{(t \intarrow s) \in(\cO_n)^\elem_{t/}}(\cP_n)_s} \\
    	{(\cP_{n-1})_{t_{\leq n-1}}} & {\lim\limits_{(t \intarrow s) \in(\cO_n)^\elem_{t/}}(\cP_{n-1})_{s_{\leq n-1}}}
    	\arrow[from=1-1, to=1-2]
    	\arrow["{d_n}"', from=1-1, to=2-1]
    	\arrow["{d_n}", from=1-2, to=2-2]
    	\arrow[from=2-1, to=2-2]
    \end{tikzcd}\]
    is a pullback square. Here, we denote by $t_{\leq n-1}$  the image of an element $t\in \P_n$ under the morphism $d_n\colon\P_n\to \P_{n-1}$.
\end{lemma}

\begin{proof}
    Consider the squares
    \[\begin{tikzcd}
    	{(\cP_n)_t} & {\lim\limits_{(t \intarrow s) \in(\cO_n)^\elem_{t/}}(\cP_n)_s} & {\lim\limits_{(t \intarrow s) \in(\cO_n)^\elem_{t/}}(\cP_1)_{s_{\geq n-1}}} \\
    	{(\cP_{n-1})_{t_{\leq n-1}}} & {\lim\limits_{(t \intarrow s) \in(\cO_n)^\elem_{t/}}(\cP_{n-1})_{s_{\leq n-1}}} & {\lim\limits_{(t \intarrow s) \in(\cO_n)^\elem_{t/}}(\cP_{0})_{s_{n-1}}}.
    	\arrow[from=1-1, to=1-2]
    	\arrow["{d_n}", from=1-1, to=2-1]
    	\arrow[from=1-2, to=1-3]
    	\arrow["{d_n}", from=1-2, to=2-2]
	\arrow["\lrcorner",phantom, very near start, from=1-2, to=2-3]
    	\arrow[from=1-3, to=2-3]
    	\arrow[from=2-1, to=2-2]
    	\arrow[from=2-2, to=2-3]
    \end{tikzcd}\]
    The right square is cartesian since $\cP \to \cO$ is a functor of double categories, so it suffices to show that the outer rectangle is cartesian. Observe that this is also the outer rectangle in
    \[\begin{tikzcd}
    	{(\cP_n)_t} & {(\cP_1)_{t_{\geq n-1}}} & {\lim\limits_{(t \intarrow s) \in(\cO_n)^\elem_{t/}}(\cP_1)_{s_{\geq n-1}}} \\
    	{(\cP_{n-1})_{t_{\leq n-1}}} & {(\cP_0)_{t_{n-1}}} & {\lim\limits_{(t \intarrow s) \in(\cO_n)^\elem_{t/}}(\cP_{0})_{s_{n-1}}}.
    	\arrow[from=1-1, to=1-2]
    	\arrow["{d_n}", from=1-1, to=2-1]
    	\arrow[from=1-2, to=1-3]
    	\arrow["{d_n}", from=1-2, to=2-2]
    	\arrow["\lrcorner", phantom, very near start, from=1-2, to=2-3]
    	\arrow["\lrcorner", very near start, phantom, from=1-1, to=2-2]
    	\arrow[from=1-3, to=2-3]
    	\arrow[from=2-1, to=2-2]
    	\arrow[from=2-2, to=2-3]
    \end{tikzcd}\]
    The left square is cartesian since $\cP$ is a double category, while the right square is cartesian since $\cP \to \cO$ is an algebrad.
\end{proof}

\begin{notation}
Given an algebraic pattern $\cO$ and a map $f \colon x \to y$ in $\cO$, we will write $f^\inert$ and $f^\activ$ for the inert and active parts of its inert-active factorization. Given a functor $\cP \to \cO$ and $\overline x \in \cP_x$, $\overline y \in \cP_y$, we write $$\Map_\cP^f(\overline x, \overline y) \coloneqq \Map_\cP(\overline x, \overline y) \times_{\Map_\cO(x,y)} \{f\}.$$
\end{notation}

\begin{lemma}\label{lem:simplified-WSF}
    Let $\cO$ be an algebraic pattern and $\cP \to \cO$ a functor satisfying the conditions $(1)$ and $(2)$ of \cref{defi:of algebrad for factactization system}. Then $\cP \to \cO$ is an algebrad if and only if for any active morphism $f \colon x \to y$ in $\cO$ and any $\overline{x} \in \cP_{x}$ and $\overline{y} \in \cP_{y}$, the map
    \begin{equation}\label{eq:simplified-WSF}
        \map^f_\cP(\overline x, \overline y) \to \lim_{\phi \in \cO^\elem_{y/}} \map^{(\phi f)^\activ}_\cP((\phi f)^\inert_!\overline x, \phi_! \overline y)
    \end{equation}
    is an equivalence.
\end{lemma}

\begin{proof}
    We need to verify that for any $x,y$ in $\cO$, any $\overline x \in \cP_x$ and any $\overline y \in \cP_y$, the square
    \[\begin{tikzcd}
    	{\Map_\cP(\overline x, \overline y)} & {\lim\limits_{(\phi \colon y \intarrow e) \in \cO^\elem_{y/}}\Map_\cP(\overline x,\phi_! \overline y)} \\
    	{\Map_\cO(x,y)} & {\lim\limits_{(\phi \colon y \intarrow e) \in \cO^\elem_{y/}} \Map_\cO(x,e)}
    	\arrow[from=1-1, to=1-2]
    	\arrow[from=1-1, to=2-1]
    	\arrow[from=1-2, to=2-2]
    	\arrow[from=2-1, to=2-2]
    \end{tikzcd}\]
    is cartesian.
    By \cite[Remark 9.7]{ChuHaugseng2021HomotopycoherentAlgebraSegal}, the map between vertical fibers over an active morphism $f \colon x \actarrow y$ may be identified with the map \cref{eq:simplified-WSF}.
    This shows that the map \cref{eq:simplified-WSF} is an equivalence if $\cP \to \cO$ is an algebrad.
    For the converse, we need to show that the map of fibers
    \[\map^g_\cP(\overline x,\overline y) \to \lim_{(\phi \colon y \intarrow e) \in \cO^\elem_{y/}} \map^{\phi g}_\cP(\overline x, \phi_! \overline y)\]
    is an equivalence for any $g \colon x \to y$ in $\cO$, any $\overline x \in \cP_x$ and any $\overline y \in \cP_y$.
    By factoring $g = f \circ j$, where $f$ is active and $j$ inert, and using cocartesian transport along inerts, we may identify this map with
    \[\Map^f_{\cP}(j_! \overline x, \overline y) \to \lim_{\phi \in \cO^\elem_{y/}} \Map^{\phi f}_{\cP}(j_! \overline x, \phi_! \overline y) \simeq \lim_{\phi \in \cO^\elem_{y/}} \Map^{(\phi f)^\activ}((\phi f j)^\inert_! \overline{x}, \phi_!\overline y).\]
    Since $f$ is active, this map is an equivalence by assumption.
\end{proof}

\begin{proposition}
    \label{prop:equivalence between double and algebraic algebrads}
    Suppose that $\cO$ is an algebraic pattern. Then the  equivalence $(-)^\dbl \colon \AlgPatt_{/\cO} \to \DblPatt_{/\cO^\dbl}$ restricts to an equivalence
    $$\Algd(\O)\to  \Algd(\O^{\dbl}).$$
\end{proposition}

\begin{proof}
    Let $p \colon \cP \to \cO$ in $\AlgPatt_{/\cO}$ be given. By \cite[Lemma 4.2]{BrankoJuran}, the functor $p^\dbl \colon \cP^\dbl \to \cO^\dbl$ is pointwise a left fibration if and only if $\cP \to \cO$ is a cocartesian fibration over $\cO^\inert$ and all morphisms in $\cP^\inert$ are cocartesian.
    Since $\cO^\dbl_0 = \cO^\inert$, we see that the conditions $(2)$ of \cref{defi:of algebrad for factactization system} and \cref{defi:of algebrad for double pattern}  are equivalent.
    
    Now let $f \colon x \actarrow y$ be an active morphism in $\cO$. Using that $(\cO^\dbl_1)^\elem_{f/} \simeq \cO^\elem_{y/}$ by \cref{remark:elementary-slices-On}, we can rewrite the square of the condition $(3)$ of \cref{defi:of algebrad for double pattern} as
    \[\begin{tikzcd}
        	{(\cP^\dbl_1)_f} & {\lim\limits_{(\phi \colon y \intarrow e) \in \cO^\elem_{y/}}(\cP^\dbl_1)_{\phi_* f}} \\
        	{(\cP^\dbl_0)_{x}} & {\lim\limits_{(y \intarrow e) \in \cO^\elem_{y/}}(\cP^\dbl_{0})_{(\phi_*f)_0}}.
        	\arrow[from=1-1, to=1-2]
        	\arrow["{d_1}"', from=1-1, to=2-1]
        	\arrow["{d_1}", from=1-2, to=2-2]
        	\arrow[from=2-1, to=2-2]
    \end{tikzcd}\]
    Note that the top horizontal map is fibered via the target projection over
    \[(\cP^\dbl_0)_{y} \simeq \lim\limits_{(\phi \colon y \intarrow e) \in \cO^\elem_{y/}}(\cP^\dbl_0)_{e}.\]
    Taking fibers of the vertical maps in this square and over $(\cP^\dbl_0)_{y}$, it follows that this square is cartesian precisely if the map
    \[\map^f_\cP(\overline y,\overline x) \to \lim_{\phi \in \cO^\elem_{y/}} \map^{(\phi f)^\activ}_\cP((\phi f)^\inert_!\bar y, \phi_!\bar x)\]
    is an equivalence for every $\overline x \in \cP_x$ and $\overline y \in \cP_y$.
    The result now follows from \cref{lem:simplified-WSF}.
\end{proof}

\begin{convention}
    On account of \cref{cor:equivalence between double and algebraic pattern}, the notions of algebraic patterns and double patterns coincide. Moreover, \cref{prop:equivalence between double and algebraic algebrads} asserts that the notions of algebrads for double patterns and algebraic patterns coincide as well. This justifies referring to double patterns as simply \textit{algebraic patterns} as well. We will do so in the remainder of this article. When it is necessary, we will explicitly state if we are thinking of an algebraic pattern as a factorization system or as a factorization double category.
\end{convention}

\section{The forest and tree categories of an algebraic pattern}\label{sec:forest-and-tree-cats}
In this section, we associate a certain category of \textit{(layered)} \textit{forests} and \textit{trees} to any algebraic pattern $\cO$. We also study the categories of presheaves on these forest and tree categories, and the precise relation between these presheaf categories.
The results we obtain will be used in \cref{sec:algebrads_as_segal_presheaves_on_the_tree_category}, where we show that $\cO$-algebrads can be modelled as certain \emph{complete Segal} presheaves on these forest and tree categories.

From now on, in light of the equivalence between algebraic patterns and double patterns from \cref{cor:equivalence between double and algebraic pattern}, we will not distinguish between $\cO$ when viewed as an algebraic pattern or as a double pattern anymore.
In particular, if $\cO$ is an algebraic pattern, we will simply write $\cO_n$ for $\cO^\dbl_n$.

\begin{definition}
Let $\cO$ be an algebraic pattern. The \textit{$\O$-forest category} is defined as
$$
\textstyle \Forest[\O] \coloneqq (\int_{\Delta^\op} \O)^\op,
$$
the Grothendieck construction of $\O$.
If
${t}\coloneqq t_0 \actarrow t_1 \actarrow \dotsb \actarrow t_n$
is an object of $\O_n$, then we will write $\langle n;{t}\rangle$ for the corresponding object in $\Forest[\O]$.
The objects of $\Forest[\O]$ will be called \emph{forests}.

The \textit{$\O$-tree category} is the full subcategory
$$
\textstyle \Tree[\O] \subset \Forest[\O]
$$
spanned by objects of the shape $\langle n;t \rangle$ such that $t \in \O^\el_n$, i.e.\ so that $t_n$ is an elementary object of $\cO$. These are called \textit{trees}. When viewed as elements of the subcategory $\Tree[\O]$, the trees will be denoted by $[n;t] \coloneqq \langle n;t \rangle$.
If $n = 0$, then $[0;t]$ is called a \textit{root}, and if $n=1$, then $[1;t]$ is called a \emph{corolla}.
\end{definition}

\begin{remark}
    It would be better to call the objects of $\Phi[\cO]$ and $\Omega[\cO]$ \emph{layered} forests and \emph{layered} trees, respectively, as we will see in \cref{ex:forest cat of Fin*} and \cref{remark:trees for Fin* are layered} below.
    However, to avoid cluttering the text and since we don't use non-layered trees throughout this paper, we have decided to drop the adjective \emph{layered} everywhere.
\end{remark}

\begin{remark}
    We note that the forest category $\Forest[\cO]$ only depends on the inert-active factorization system on $\cO$, while $\Tree[\cO]$ also depends on the collection of elementary objects.
\end{remark}

\subsection{Examples}

We now describe the categories $\Forest[\O]$ and $\Tree[\O]$ in a few examples, which also motivate the names ``(layered) forest category'' and ``(layered) tree category''.
Observe that since $\Forest[\O] \to \Delta$ is a cartesian fibration, a map $\langle n;t \rangle \to \langle m;s \rangle$ consists of a map $\phi \colon [n] \to [m]$ in $\Delta$ together with a map
\[\begin{tikzcd}
	{s_{\phi(0)}} & {s_{\phi(1)}} & \cdots & {s_{\phi(n-1)}} & {s_{\phi(n)}} \\
	{t_0} & {t_1} & \cdots & {t_{n-1}} & {t_n}
	\arrow[squiggly, from=1-1, to=1-2]
	\arrow[tail, from=1-1, to=2-1]
	\arrow[squiggly, from=1-2, to=1-3]
	\arrow[tail, from=1-2, to=2-2]
	\arrow[squiggly, from=1-3, to=1-4]
	\arrow[squiggly, from=1-4, to=1-5]
	\arrow[tail, from=1-4, to=2-4]
	\arrow[tail, from=1-5, to=2-5]
	\arrow[squiggly, from=2-1, to=2-2]
	\arrow[squiggly, from=2-2, to=2-3]
	\arrow[squiggly, from=2-3, to=2-4]
	\arrow[squiggly, from=2-4, to=2-5]
\end{tikzcd}\]
in $\O_n$ (beware of the variance!).
We will write $\underline{n}$ for the set $\{1, \cdots, n\}$ with $n$ elements and $\langle n \rangle$ for the pointed set $\{*,1,\ldots,n\}$.

\begin{example}[$\O = \F_*$]\label{ex:forest cat of Fin*}
    Let $\F_*$ be the category of finite pointed sets, with one of the algebraic pattern structures from \cref{examples:examples of patterns}. The category $\F_*^\act$ is equivalent to the category of finite sets, so we may identify the objects of $\Forest[\F_*]$ with sequences $X_0 \to \cdots \to X_n$ of maps between finite sets. Such a sequence of $n$ maps can be visualized as a \emph{forest with $n$ layers}. For example, the forest
    \[\scalebox{0.8}
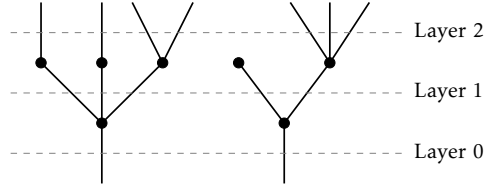
{\begin{tikzpicture}
        \draw[fill] (0,1) circle [radius=0.085];
        \draw[fill] (-1,2) circle [radius=0.085];
        \draw[fill] (0,2) circle [radius=0.085];
        \draw[fill] (1,2) circle [radius=0.085];
        
        \draw[fill] (3,1) circle [radius=0.085];
        \draw[fill] (2.25,2) circle [radius=0.085];
        \draw[fill] (3.75,2) circle [radius=0.085];
        
        \draw[thick] (0,0) -- (0,1);
        \draw[thick] (0,1) -- (-1,2);
        \draw[thick] (0,1) -- (0,2);
        \draw[thick] (0,1) -- (1,2);
        \draw[thick] (1,2) -- (1.5,3);
        \draw[thick] (1,2) -- (0.5,3);
        \draw[thick] (0,2) -- (0,3);
        \draw[thick] (-1,2) -- (-1,3);

        \draw[thick] (3,0) -- (3,1);
        \draw[thick] (3,1) -- (2.25,2);
        \draw[thick] (3,1) -- (3.75,2);
        \draw[thick] (3.75,2) -- (3.10,3);
        \draw[thick] (3.75,2) -- (3.75,3);
        \draw[thick] (3.75,2) -- (4.40,3);
        
        \draw[dashed, gray] (-1.5,0.5) -- (5,0.5) node[right, black, font=\small]{Layer 0};
        \draw[dashed, gray] (-1.5,1.5) -- (5,1.5) node[right, black, font=\small]{Layer 1};
        \draw[dashed, gray] (-1.5,2.5) -- (5,2.5) node[right, black, font=\small]{Layer 2};
    \end{tikzpicture}}\]
    corresponds to a sequence of maps $\underline 7 \to \underline 5 \to \underline 2$, where the edges correspond to the elements of the sets.

    Note that inert maps $\langle k \rangle \to \langle l \rangle$ can be identified with injections $\underline l \to \underline k$. This allows us to identify the maps of forests $\langle n;t \rangle \to \langle m;s \rangle$ over $\phi \colon [m] \to [n]$ with diagrams
    \[\begin{tikzcd}
    	{t_0} & {t_1} & \cdots & {t_{n-1}} & {t_n} \\
    	{s_{\phi(0)}} & {s_{\phi(1)}} & \cdots & {s_{\phi(n-1)}} & {s_{\phi(n)}}
    	\arrow[from=1-1, to=1-2]
    	\arrow[hook, from=1-1, to=2-1]
    	\arrow["\lrcorner", very near start, phantom, from=1-1, to=2-2]
    	\arrow[from=1-2, to=1-3]
    	\arrow[hook, from=1-2, to=2-2]
    	\arrow["\lrcorner", very near start, phantom, from=1-2, to=2-3]
    	\arrow[from=1-3, to=1-4]
    	\arrow[from=1-4, to=1-5]
    	\arrow[hook, from=1-4, to=2-4]
    	\arrow["\lrcorner", very near start, phantom, from=1-4, to=2-5]
    	\arrow[hook, from=1-5, to=2-5]
    	\arrow[from=2-1, to=2-2]
    	\arrow[from=2-2, to=2-3]
    	\arrow[from=2-3, to=2-4]
    	\arrow[from=2-4, to=2-5]
    \end{tikzcd}\]
    of finite sets such that the vertical maps are injective and every square is cartesian.
    It follows that $\Forest[\F_*]$ is precisely the category $\Delta_\mathbb{F}$ defined in \cite{Barwick}, see also \cite[Definition 2.1]{ChuHaugsengea2018TwoModelsHomotopy}.
    The category $\Tree[\F_*^\flat]$ is the full subcategory of $\Forest[\F_*]$ spanned by the trees, which is denoted $\Delta_\mathbb{F}^1$ in \cite{ChuHaugsengea2018TwoModelsHomotopy}.
    If we take the pattern structure $\F_*^\natural$ where the set $\langle 0 \rangle$ is also elementary, then we obtain a slightly larger full subcategory $\Tree[\F_*^\natural] \subset \Forest[\F_*]$ which also contains the forests of the form $\langle n; \langle 0 \rangle = \cdots = \langle 0 \rangle \rangle$.
\end{example}

\begin{remark}\label{remark:trees for Fin* are layered}
    Let us warn the reader that $\Tree[\F_*^\flat] = \Delta^1_\mathbb{F}$ does not agree with the category $\Omega$ of trees introduced by Weiss--Moerdijk \cite{MoerdijkWeiss2007DendroidalSets}, which is used to define dendroidal sets and spaces.
    However, it is shown in \cite{ChuHaugsengea2018TwoModelsHomotopy} that there is a functor $\Tree[\F_*^\flat] \to \Omega$ inducing an equivalence between (complete) Segal presheaves on the categories $\Tree[\F_*^\flat]$ and $\Omega$.
\end{remark}

\begin{example}[$\O = \Span(\F)$]\label{exa:span-forests}
    By taking $G$ to be the trivial group in \cref{examples:examples of patterns}, $\Span(\F)$ has a pattern structure where the active morphisms are the forward maps and the inerts are the backward maps.
    In particular, since $\Span(\F)^\act = \F = \F^\act_*$, the objects of $\Forest[\Span(\F)]$ agree with those of $\Forest[\F_*]$.
    However, the category $\Forest[\Span(\F)]$ has more maps: 
    a map of forests $f \colon \langle n;t \rangle \to \langle m;s\rangle$ over $\phi \colon [m] \to [n]$ is a diagram
    \begin{equation}\label{eq:forestmapsSpanF}
    \begin{tikzcd}
    	{t_0} & {t_1} & \cdots & {t_{n-1}} & {t_n} \\
    	{s_{\phi(0)}} & {s_{\phi(1)}} & \cdots & {s_{\phi(n-1)}} & {s_{\phi(n)}}
    	\arrow[from=1-1, to=1-2]
    	\arrow[from=1-1, to=2-1]
    	\arrow["\lrcorner", very near start, phantom, from=1-1, to=2-2]
    	\arrow[from=1-2, to=1-3]
    	\arrow[from=1-2, to=2-2]
    	\arrow["\lrcorner", very near start, phantom, from=1-2, to=2-3]
    	\arrow[from=1-3, to=1-4]
    	\arrow[from=1-4, to=1-5]
    	\arrow[from=1-4, to=2-4]
    	\arrow["\lrcorner", very near start, phantom, from=1-4, to=2-5]
    	\arrow[from=1-5, to=2-5]
    	\arrow[from=2-1, to=2-2]
    	\arrow[from=2-2, to=2-3]
    	\arrow[from=2-3, to=2-4]
    	\arrow[from=2-4, to=2-5]
    \end{tikzcd}
    \end{equation}
    of finite sets where every square is cartesian. This means that the images of two trees under $f$ are allowed to overlap in the forest $\langle m;s \rangle$, which is not allowed in $\Forest[\F_*]$.
    Observe that if $t_n = \underline 1$ or $t_n = \underline 0$, then every vertical map in \cref{eq:forestmapsSpanF} is injective.
    In particular, the tree categories $\Tree[\Span(\F)^\flat]$ and $\Tree[\F_*^\flat]$ are equivalent, and similarly for $\Tree[\Span(\F)^\natural]$ and $\Tree[\F_*^\natural]$.
\end{example}

\begin{example}[$\O = \Delta^{\op,\flat}$]
    In \cref{examples:examples of patterns}, we described a pattern structure on $\Delta^\op$ where $[1]$ is the only elementary object.
    The tree category for this pattern can be identified with a certain category of ``planar trees with layers'', in analogy with \cref{ex:forest cat of Fin*}.
    The objects of the forest category for this algebraic pattern can be viewed as ordered collections of such planar trees.
\end{example}

\begin{example}[$\O = \Delta^{\op,\natural}$]
    The algebraic pattern $\Delta^{\op,\natural}$ has the same actives and inerts as $\Delta^{\op,\flat}$, but its elementaries are both $[0]$ and $[1]$. Algebrads for this pattern are (virtual) double categories, which suggests a different visual representation of the objects of $\Forest[\Delta^\op]$, namely as pasting diagrams.
    For example, the trees $[1;[2] \leftarrow [1]]$ and $[1;[0] \leftarrow [1]]$ can be pictured as
    
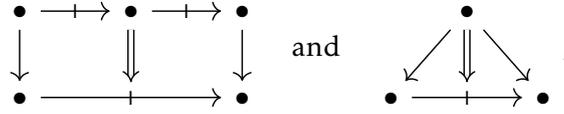
\[\begin{tikzcd}
    	\bullet & \bullet & \bullet \\
    	\bullet && \bullet
    	\arrow["\shortmid"{marking}, from=1-1, to=1-2]
    	\arrow[from=1-1, to=2-1]
    	\arrow["\shortmid"{marking}, from=1-2, to=1-3]
    	\arrow[from=1-3, to=2-3]
    	\arrow[""{name=0, anchor=center, inner sep=0}, "\shortmid"{marking}, from=2-1, to=2-3]
    	\arrow[between={0}{0.8}, Rightarrow, from=1-2, to=0]
    \end{tikzcd}\quad \text{and} \quad
    \begin{tikzcd}[column sep=small]
    	& \bullet & \\
    	\bullet && \bullet
    	\arrow[from=1-2, to=2-1]
    	\arrow[from=1-2, to=2-3]
    	\arrow[""{name=0, anchor=center, inner sep=0}, "\shortmid"{marking}, from=2-1, to=2-3]
    	\arrow[between={0}{0.8}, Rightarrow, from=1-2, to=0]
    \end{tikzcd},\]
    while the forest  $\langle 2;[4]\xleftarrow{d^2}[3]\xleftarrow{d^1} [2]\rangle$ is represented by
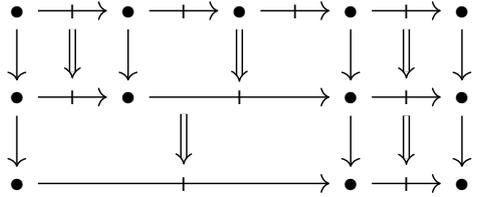
\[\begin{tikzcd}
	\bullet & \bullet & \bullet & \bullet & \bullet \\
	\bullet & \bullet && \bullet & \bullet \\
	\bullet &&& \bullet & \bullet
	\arrow[""{name=0, anchor=center, inner sep=0}, "\shortmid"{marking}, from=1-1, to=1-2]
	\arrow[from=1-1, to=2-1]
	\arrow[""{name=1, anchor=center, inner sep=0}, "\shortmid"{marking}, from=1-2, to=1-3]
	\arrow[from=1-2, to=2-2]
	\arrow["\shortmid"{marking}, from=1-3, to=1-4]
	\arrow[""{name=2, anchor=center, inner sep=0}, "\shortmid"{marking}, from=1-4, to=1-5]
	\arrow[from=1-4, to=2-4]
	\arrow[from=1-5, to=2-5]
	\arrow[""{name=3, anchor=center, inner sep=0}, "\shortmid"{marking}, from=2-1, to=2-2]
	\arrow[from=2-1, to=3-1]
	\arrow[""{name=4, anchor=center, inner sep=0}, "\shortmid"{marking}, from=2-2, to=2-4]
	\arrow[""{name=5, anchor=center, inner sep=0}, "\shortmid"{marking}, from=2-4, to=2-5]
	\arrow[from=2-4, to=3-4]
	\arrow[from=2-5, to=3-5]
	\arrow[""{name=6, anchor=center, inner sep=0}, "\shortmid"{marking}, from=3-1, to=3-4]
	\arrow[""{name=7, anchor=center, inner sep=0}, "\shortmid"{marking}, from=3-4, to=3-5]
	\arrow[between={0.2}{0.8}, Rightarrow, from=0, to=3]
	\arrow[between={0.6}{0.9}, Rightarrow, from=1, to=6]
	\arrow[between={0}{0.8}, Rightarrow, from=1-3, to=4]
	\arrow[between={0.2}{0.8}, Rightarrow, from=2, to=5]
	\arrow[between={0.2}{0.8}, Rightarrow, from=5, to=7]
\end{tikzcd}.\]
    More generally, trees and forests for $\Delta^{\op,\natural}$ correspond to the pasting diagrams appearing in the classical definition of (ordinary) virtual double categories.
\end{example}

\begin{example}[$\O = \Span(\F_G)$]
    Similarly to \cref{exa:span-forests}, the objects of the forest category $\Forest[\Span(\F_G)]$ may be identified with sequences of maps between finite $G$-sets.
    A map of forests $f \colon \langle n;t\rangle \to \langle m;s \rangle$ over $\phi \colon [m] \to [n]$ is then a diagram
    \begin{equation}
    \begin{tikzcd}
    	{t_0} & {t_1} & \cdots & {t_{n-1}} & {t_n} \\
    	{s_{\phi(0)}} & {s_{\phi(1)}} & \cdots & {s_{\phi(n-1)}} & {s_{\phi(n)}}
    	\arrow[from=1-1, to=1-2]
    	\arrow[from=1-1, to=2-1]
    	\arrow["\lrcorner",very near start, phantom, from=1-1, to=2-2]
    	\arrow[from=1-2, to=1-3]
    	\arrow[from=1-2, to=2-2]
    	\arrow["\lrcorner",very near start, phantom, from=1-2, to=2-3]
    	\arrow[from=1-3, to=1-4]
    	\arrow[from=1-4, to=1-5]
    	\arrow[from=1-4, to=2-4]
    	\arrow["\lrcorner",very near start, phantom, from=1-4, to=2-5]
    	\arrow[from=1-5, to=2-5]
    	\arrow[from=2-1, to=2-2]
    	\arrow[from=2-2, to=2-3]
    	\arrow[from=2-3, to=2-4]
    	\arrow[from=2-4, to=2-5]
    \end{tikzcd}
    \end{equation}
    of finite $G$-sets where every square is cartesian.
    In fact, $\Forest[\Span(\F_G)]$ is equivalent to the category $\Fun(BG,\Forest[\Span(\F)])$ of (layered) forests with a $G$-action.
    The category $\Tree[\Span(\F_G)]$ is then the full subcategory spanned by those forests $X_0 \to \cdots \to X_n$ for which $G$ acts transitively on its trees; that is, $G$ acts transitively on $X_n$.
    These categories are layered versions of the categories of $G$-forests and $G$-trees studied by Pereira--Bonventre \cite{Pereira2018EquivariantDendroidalSetsa,BonventrePereira2020EquivariantDendroidalSegal,BonventrePereira2022EquivariantDendroidalSets}.
\end{example}

We will now discuss an example of a different flavour:
it turns out that the tree category $\Tree[\cO]^\op$ of an algebraic pattern $\cO$ admits an algebraic pattern structure which, in a sense, intertwines the algebraic pattern structure of $\Delta^\op$ and $\cO$.
This will allow us to iterate the tree construction.

\begin{definition}\label{def:inerts and actives of the tree category}
Let $f \colon [m;s] \to [n;t]$ be a map in $\Tree[\O]$ with underlying map $\phi \colon [m] \to [n]$ in $\Delta$. Then $f$ is called \textit{inert} if $\phi$ is inert. We call $f$ \textit{active} if $\phi$ is active and for each $i\leq m$, the inert morphism $t_{\phi(i)}\intarrow s_{i}$ is an equivalence, or equivalently, if $t_{n}\intarrow s_{m}$ is an equivalence.
\end{definition}

\begin{example}[The iterated tree construction]\label{ex:iterated trees}
Let $\O$ be an algebraic pattern. We will see in \cref{lemma:factorization system on the tree category} below that the inert and active maps form a factorization system on $\Tree[\cO]$. We therefore obtain a pattern structure on $\Tree[\cO]^\op$ by taking as elementary objects the roots and corollas (i.e.\ trees $[n;t]$ with $n \leq 1$), which we denote by $\Tree[\O]^{\op,\natural}$.

One can then iterate the tree construction. We define $\Tree^n[\O]$ inductively by setting $\Tree^1[\O] \coloneqq \Tree[\O]$, and
\[
\Tree^{n+1}[\O] \coloneqq \Tree\big[\Tree^n[\O]^{\op,\natural}\big].
\]
In particular, $\Tree^{1}[*]^{\op,\natural}$ is the pattern $\Delta^{\op,\natural}$, where $\ast$ is the terminal algebraic pattern.
We recently learned that Cl\'emence Chanavat has an (unpublished) proof of the fact that the pattern $\Tree^n[*]^{\op, \natural}$ can be used to describe the (strict) virtual $n$-uple categories considered by Arkor in \cite{Arkor_slide}.
\end{example}

\subsection{The slice categories of the forest category}\label{sub:the_mapping_spaces_of_the_forest_category}

As preparation for what follows, we need a better understanding of the slices $\Forest[\O]_{/[n;t]}$ of the forest category $\Forest[\O]$. We introduce the following subcategories:

\begin{definition}
    \label{definition: delta circ i}
    Let $0 \leq i \leq n$ be an integer.
    We will write $\Delta^{\circ i}_{/[n]}$ for the full subcategory of $\Delta_{/[n]}$ spanned by the maps $\phi \colon [m] \to [n]$ so that $\phi(m) = i$.  If $\langle n;t \rangle$ is a forest, then we define the full subcategory $\Forest^{\circ i}[\O]_{/\langle n;t \rangle}$ by the pullback 
    \[
    \begin{tikzcd}
        \Forest^{\circ i}[\O]_{/\langle n;t\rangle}\arrow[d] \arrow[r] \ar[dr,"\lrcorner",very near start, phantom] & \Forest[\O]_{/\langle n;t\rangle}\arrow[d]\\
        \Delta^{\circ i}_{/[n]} \arrow[r] & \Delta_{/[n]}.
    \end{tikzcd}
    \]
\end{definition}

\begin{proposition}
    \label{proposition:explicit description of mapping space}
    Let $\cO$ be an algebraic pattern, $\langle n;t \rangle \in \Forest[\cO]$ a forest and $0 \leq i \leq n$ an integer.
    Then there exists an equivalence 
    $$
    \Delta_{/[n]}^{\circ i} \times (\O_{t_{i}/}^\inert)^\op \simeq \Delta_{/[n]}^{\circ i} \times (\O_i^\op)_{/t_{\leq i}} \xrightarrow{\simeq} \Forest^{\circ i}[\O]_{/\langle n;t \rangle}
    $$
    that carries  a pair $(\phi \colon [k]\to [i], g  \colon t_{\leq i}\to s)$ to the composite
$$\langle k;\phi^*s \rangle \to \langle i;s \rangle \to \langle i;t_{\leq i} \rangle \to \langle n;t \rangle.$$
\end{proposition}

To prove the above, we will make use of the following basic results on cartesian fibrations:

\begin{lemma}\label{lemma:slicing fibration}
    Let $p \colon E \to \C$ be a cartesian fibration. Then for every $e \in E$, the induced functor $p_{/e} \colon E_{/e} \to \C_{/p(e)}$ is a cartesian fibration as well. For any morphism $x \xrightarrow{f} y \xrightarrow{g} p(e)$, the associated structure map $(E_{/e})_y \to (E_{/e})_x$ of $p_{/e}$ is given by the functor
    $(E_y)_{/g^*e} \to (E_x)_{/f^*g^*e}$ induced by $f^* \colon E_y \to E_x$.
\end{lemma}

\begin{proof}
    This follows from \cite[Proposition 2.4.3.1]{HTT} and the explicit description of the cartesian morphisms given there.
\end{proof}

\begin{lemma}\label{lemma:trivializing fibration}
    Let $\C$ be a category with a final object $\ast$. If $p \colon E \to \C$ is a cartesian fibration so that the transport maps $E_\ast \to E_c$ are equivalences for every $x \in \C$, then there exists an equivalence $E \simeq \C \times E_\ast$ over $\C$.
\end{lemma}

\begin{proof}
    Let $F \colon \C^\op \to \Cat$ be the functor classifying $p \colon E \to \C$. The conditions on $p$ imply that $F$ is left Kan extended from its restriction to $*$, hence that it is the constant functor with value $F(*) = E_*$.
\end{proof}

\begin{proof}[Proof of \cref{proposition:explicit description of mapping space}]
    On account of \cref{lemma:slicing fibration}, the projection $\Forest[\O]_{/\langle n;t \rangle} \to \Delta_{/[n]}$ is again a cartesian fibration. Let $\phi \colon [m] \to [n]$ be a map so that $\phi(m) = i$. Then $\phi$ factors uniquely through the inclusion $\{0\leq \dotsb \leq i\} \colon [i] \hookrightarrow [n]$ via a map $\psi \colon [m] \to [i]$. Thus $\Delta^{\circ i}_{/[n]}$ has a final object given by the canonical inclusion $\{0 \leq \dotsc \leq i\} \colon [i] \hookrightarrow [n]$. The transport functor along $\psi$ of the cartesian fibration $\Forest[\O]_{/\langle n;t \rangle} \to \Delta_{/[n]}$  fits in a commutative triangle
    \[
        \begin{tikzcd}[column sep =tiny]
            (\O_i)_{t_{\leq i}/} \arrow[rr, "\psi^*"]\arrow[dr] && (\O_m)_{\psi^*t_{\leq i}/}\arrow[dl] \\
            & (\O_0)_{t_i/}
        \end{tikzcd}
    \]
    after applying $(-)^\op$.
    The slanted arrows are equivalences by \cref{rem:slices-factorization-double-category} since $\O$ is a factorization double category, so the result follows from \cref{lemma:trivializing fibration}.
\end{proof}

\begin{corollary}\label{cor:explicit description of mapping space}
    Let $0 \leq i \leq n$ be an integer.
    There is a natural pullback square 
    \[  
    \begin{tikzcd}
        (\O^{\op}_0)_{/t_i} \arrow[r]\arrow[d] & \Forest[\O]_{/\langle n;t\rangle} \arrow[d] \\
        \{\phi\} \arrow[r] & \Delta_{/[n]}
    \end{tikzcd}
    \]
    for $\phi \in \Delta^{\circ i}_{/[n]}$. \qed
\end{corollary}

\begin{remark}
    The equivalence of above can be described on points as follows. For a map $f \colon t_i \to x$, the corresponding morphism $\langle m;s \rangle \to \langle n;t \rangle$ corresponds to $\phi^*t \to s$ in $\O_m$ obtained by completing the diagram below with the dashed factorizations:
    \[
        \begin{tikzcd}
            t_{\phi(0)} \arrow[squiggly, r] \arrow[d,dashed] & t_{\phi(1)}\arrow[d,dashed] \arrow[squiggly, r]& \dotsb \arrow[squiggly, r] & t_{\phi(m-1)}\arrow[squiggly, r] \arrow[d,dashed]& t_{\phi(m)} \arrow[d, "f",tail] \\ 
            s_0 \arrow[r,dashed] & s_1\arrow[r,dashed] & \dotsb \arrow[r, dashed] & s_{m-1}\arrow[r,dashed] & s_m.
        \end{tikzcd}
    \]
\end{remark}

\begin{remark}
    One can also describe the mapping spaces in $\Forest[\cO]$ over a given map $\phi \colon [m] \to [n]$ in $\Delta$. Since $\Forest[\cO] \to \Delta$ is cartesian, the space $\map^\phi_{\Forest[\cO]}(\langle m;s \rangle , \langle n;t \rangle )$ of maps $\langle m;s \rangle  \to \langle n;t \rangle $ over $\phi$ agrees with the space $\map_{\cO_m}(\phi^*t,s)$.
    When $m=1$, this mapping space is given by
    \[\map_{\O_1}(\phi^*t,s) \simeq \map_{\O_0}(t_{\phi(1)},s_1) \times_{(\O^\act_{/s_1})^\simeq} \{s_0 \rightsquigarrow s_1\}.\]
    There is a similar formula for general $m$, which can either be described iteratively (using the Segal condition of $\cO$) or by replacing $\cO^\activ_{/s_1}$ with the category of strings of $m$ active morphisms that end with $s_m$.
\end{remark}

Finally, we are now ready to prove that the inert and active maps from \cref{def:inerts and actives of the tree category} form the factorization system on $\Tree[\cO]$ used in \cref{ex:iterated trees}.

\begin{lemma}
\label{lemma:factorization system on the tree category}
The active and inert morphisms of $\Tree[\O]$ form a factorization system.
\end{lemma}

\begin{proof}
We have to show that the composition functor
$$
\Ar^{\act}(\Tree[\O]) \times_{\Tree[\O]} \Ar^{\inert}(\Tree[\O]) \to \Ar(\Tree[\O])
$$
has contractible fibers on account of \cite[Proposition 3]{FactSystems}. If $f : [m;s] \to [n;t]$ is a map between trees, then one readily verifies that the fiber above $f$ is given by the full subcategory 
$
\mathrm{Fact}(f) \subset (\Tree[\O]_{/[n;t]})_{f/}
$
that corresponds to the triangles
\[
    \begin{tikzcd}[column sep = tiny, row sep =tiny]
        {[m;s]} \arrow[dr,"f"']\arrow[rr, "f'"] && {[k;v]}\arrow[dl, "f''"]\\
        & {[n;t]}
    \end{tikzcd}
\]
so that $f'$ is active and $f''$ is inert. Let $\phi$ be the underlying map of $f$. Then $f$ is contained in $\Tree[\O]^{\circ i}_{/[n;t]}$ with $i := \phi(m)$, and this automatically implies that $f'' \in \Tree[\O]^{\circ i}_{/[n;t]}$ as well since $f'$ is active.   On account of \cref{proposition:explicit description of mapping space}, we may identify $\mathrm{Fact}(f)$ with the full subcategory of $(\Delta_{/[n]})_{\phi/} \times ((\O^{\el, \op})_{/t_i})_{\psi/}$ spanned by the pairs of triangles
\[
    \left(\begin{tikzcd}[column sep = tiny, row sep =tiny]
       {[m]} \arrow[dr,"\phi"']\arrow[rr, "\phi'"] && {[k]},\arrow[dl, "\phi''"]\\
        & {[n]}
    \end{tikzcd}
    \begin{tikzcd}[column sep = tiny, row sep =tiny]
       {s_m} \arrow[dr,"\psi"']\arrow[rr, "\simeq"] && {e}\arrow[dl]\\
        & {t_i}
    \end{tikzcd}
    \right)
\]
so that $\phi'$ is active and $\phi''$ is inert.
Here $\psi$ is induced by $f$. Thus $\mathrm{Fact}(f)$ is a contractible category.
\end{proof}

\subsection{Comparing presheaves on the forest and tree categories}\label{subsection:comparing forests and trees}

Let $i \colon \Tree[\cO] \to \Forest[\cO]$ be the inclusion.
Since $i$ is fully faithful, right Kan extension along $i$ gives an adjunction
\[i^* : \PSh(\Forest[\O]) \rightleftarrows \PSh(\Tree[\O]) : i_*,\]
such that $i_*$ is fully faithful.
The goal of this section is to identify the image of $i_*$.

\begin{proposition}\label{prop:RKan-trees-to-forests}
    Let $\cO$ be an algebraic pattern and $X$ a presheaf on $\Forest[\cO]$. Then $X$ lies in the image of $i_* \colon \PSh(\Tree[\cO]) \to \PSh(\Forest[\cO])$ if and only
    \begin{itemize}
        \item for every object $x$ in $\cO$, the map
        \[X(\langle 0;x \rangle ) \to \lim_{e \in \cO^\elem_{x/}} X(\langle 0;e \rangle)\]
        is an equivalence, and
        \item for every forest $\langle n;t \rangle$ with $n \geq 1$, the square
    \begin{equation}\label{eq:Tree-vs-forest-pullback-square}\begin{tikzcd}
        {X(\langle n;t \rangle )} & {\lim\limits_{(t \intarrow s) \in (\cO_{n})^\elem_{t/}}X(\langle n;s \rangle )} \\
        {X(\langle n-1;t_{\leq n-1} \rangle)} & {\lim\limits_{(t \intarrow s) \in (\cO_{n})^\elem_{t/}}X(\langle n-1;s_{\leq n-1} \rangle )}
        \arrow[from=1-1, to=1-2]
        \arrow["{d_n}"', from=1-1, to=2-1]
        \arrow["{d_n}", from=1-2, to=2-2]
        \arrow[from=2-1, to=2-2]
    \end{tikzcd}\end{equation}
    is cartesian.
    \end{itemize}
\end{proposition}

We will use the following lemma.

\begin{lemma}\label{lem:decomposition-colimit-over-correspondence}
    Let $\cC$ be a category and let $p \colon \cC \to [1]$ be a functor with fibers $\cC_0$ and $\cC_1$ over $0$ and $1$, respectively. Then for any complete category $\mathscr{E}$ and any diagram $F \colon \cC \to \mathscr{E}$, the square
    \[\begin{tikzcd}
        {\lim_{c \in \cC} F(c)} & {\lim_{d \in\cC_1} F(d)} \\
        {\lim_{c \in\cC_0} F(c)} & {\lim_{c \in\cC_0} \lim_{d \in(\cC_1)_{c/}} F(d)}
        \arrow[from=1-1, to=1-2]
        \arrow[from=1-1, to=2-1]
        \arrow[from=1-2, to=2-2]
        \arrow[from=2-1, to=2-2]
    \end{tikzcd}\]
    is cartesian.
\end{lemma}

\begin{proof}
    Let us write $\Gamma(\cC) = \Fun_{/[1]}([1],\cC)$ for the category of sections of $p$.
    By \cite[Lemma 5.1.1]{AyalaFrancis2020FibrationsInftyCategories}, we obtain an equivalence
    \[\cC_0 \cup_{\Gamma(\cC)} \Gamma(\cC) \times [1] \cup_{\Gamma(\cC)} \cC_1 \simeq \cC.\]
    In particular, the limit of $F \colon \cC \to \mathscr{E}$ decomposes as an iterated pullback
    \[\lim_{c \in \cC} F(c) \simeq \lim_{c \in \cC_0} F(c) \times_{\lim_{\gamma \in \Gamma(\cC)} F(\gamma(0))} \lim_{(\gamma,i) \in \Gamma(\cC) \times [1]} F(\gamma(i)) \times_{\lim_{\gamma \in \Gamma(\cC)} F(\gamma(1))} \lim_{d \in \cC_1} F(d).\]
    Since $\Gamma(\cC) \times \{0\} \hookrightarrow \Gamma(\cC) \times [1]$ is initial, we can write this as a single pullback
    \[\begin{tikzcd}
        {\lim_{c \in \cC} F(c)} \ar[dr,"\lrcorner",very near start, phantom] & {\lim_{d \in\cC_1} F(d)} \\
        {\lim_{c \in\cC_0} F(c)} & {\lim_{\gamma \in \Gamma(\cC)} F(\gamma(1))}.
        \arrow[from=1-1, to=1-2]
        \arrow[from=1-1, to=2-1]
        \arrow[from=1-2, to=2-2]
        \arrow[from=2-1, to=2-2]
    \end{tikzcd}\]
    Finally, to see that the bottom right limit agrees with $\lim_{c \in \cC_0} \lim_{d \in (\cC_1)_{c/}} F(d)$, note that limits over $\Gamma(\cC)$ can be computed by first right Kan extending along $\Gamma(\cC) \to \cC_0$ and then taking the limit over $\cC_0$.
    Since $\Gamma(\cC) \to \cC_0$ is a cartesian fibration whose unstraightening is the functor $c \mapsto (\cC_1)_{c/}$, it follows from the dual of \cite[Proposition 4.3.3.10]{HTT} that the right Kan extension of $\Gamma(\cC) \to \Spc; \gamma \mapsto F(\gamma(1))$ along $\Gamma(\cC) \to \cC_0$ is given by $c \mapsto \lim_{\gamma \in (\cC_1)_{c/}} F(\gamma(1))$.
    This concludes the proof.
\end{proof}

Let $f \colon \langle m;s \rangle \to \langle n;t \rangle$ be a map in $\Forest[\cO]$ and $\phi \colon [m] \to [n]$ the underlying map in $\Delta$.
We say that $f$ \emph{reaches $n$} if $\phi(m) = n$. Note that for a triangle
\[\begin{tikzcd}[column sep=tiny]
	{\langle m;s \rangle} && {\langle l;u \rangle} \\
	& {\langle n;t \rangle,}
	\arrow[from=1-1, to=1-3]
	\arrow["f"', from=1-1, to=2-2]
	\arrow["g", from=1-3, to=2-2]
\end{tikzcd}\]
if $f$ reaches $n$, then so does $g$.
In particular, we obtain a functor
\[\Tree[\cO]_{/\langle n;t \rangle} \to [1]; \quad (f \colon \langle m;s \rangle \to \langle n;t \rangle) \mapsto \begin{cases}
1 &\text{if } f \text{ reaches n,} \\
0 &\text{otherwise.}
\end{cases}\]
Let us write $(\Tree[\cO]_{/\langle n;t \rangle})_0$ and $(\Tree[\cO]_{/\langle n;t \rangle})_1$ for the fibers of this functor over $0$ and $1$, respectively.
We will prove \cref{prop:RKan-trees-to-forests} by applying \cref{lem:decomposition-colimit-over-correspondence} to this functor.
This requires the following two lemmas.

\begin{lemma}\label{lem:lemma1-correspondence}
    Let $\cO$ be an algebraic pattern and $\langle n;t \rangle \in \Forest[\cO]$ a forest. Then the inclusion
    \[\cO_{t_n/}^\elem \simeq (\cO_n^\elem)_{t/} \to (\Tree[\cO]_{/\langle n;t \rangle})_1^\op\]
    is initial.
\end{lemma}

\begin{proof}
    By \cref{proposition:explicit description of mapping space}, this inclusion is equivalent to the inclusion
    \[\cO_{t_n/}^\elem \hookrightarrow (\Delta^{\circ n}_{/[n]})^\op \times \cO_{t_n/}^\elem \simeq (\Tree[\cO]_{/\langle n;t \rangle})_1^\op\]
    that takes $f \in \cO_{t_n/}^\elem$ to $(\id_{[n]},f)$.
    Since $\id_{[n]}$ is terminal in $\Delta^{\circ n}_{/[n]}$, this functor is a left adjoint and hence initial.
\end{proof}

\begin{lemma}\label{lem:lemma2-correspondence}
    Let $\cO$ be an algebraic pattern, $\langle n;t \rangle \in \Forest[\cO]$ a forest and $\langle n;s \rangle \to \langle n;t \rangle$ a map of forests whose underlying map in $\Delta$ is $\id_{[n]}$.
    Then the canonical map
    \[\Tree[\cO]_{/\langle n-1;s_{\leq n-1} \rangle} \to ((\Tree[\cO]_{/\langle n;t \rangle})_0)_{/\langle n;s\rangle}\]
    is an equivalence.
\end{lemma}

\begin{proof}
    Note that $((\Tree[\cO]_{/\langle n;t \rangle})_0)_{/\langle n;s \rangle}$ is the full subcategory of $(\Tree[\cO]_{/\langle n;t \rangle})_{/\langle n;s \rangle} \simeq \Tree[\cO]_{/\langle n;s \rangle}$ spanned by those maps $\langle m;u \rangle \to \langle n;s \rangle$ that don't reach $n$.
    This is canonically equivalent to $\Tree[\cO]_{/\langle n-1;s_{\leq n-1} \rangle}$.
\end{proof}

\begin{proof}[Proof of \cref{prop:RKan-trees-to-forests}]
    If $n=0$, then $\cO_{x/}^\elem \to (\Tree[\cO]_{/\langle 0;x \rangle})^\op$ is initial by \cref{lem:lemma1-correspondence}.
    The result therefore follows if we can show that the square \cref{eq:Tree-vs-forest-pullback-square} is cartesian for every forest $\langle n;t \rangle$ with $n \geq 1$ if and only if the map
    \[X(\langle n;t \rangle ) \to \lim_{[m;s] \in (\Tree[\cO]_{/\langle n;t \rangle })^\op} X(\langle m;s\rangle )\]
    is an equivalence for every forest $\langle n;t \rangle $ with $n \geq 1$.
    We will show this by induction on $n$.
    Combining \cref{lem:decomposition-colimit-over-correspondence}, \cref{lem:lemma1-correspondence} and \cref{lem:lemma2-correspondence}, it follows that this map is an equivalence precisely if the square
    \[\begin{tikzcd}
        {X(\langle n;t \rangle)} & {\lim\limits_{(t \intarrow s) \in (\cO_{n})^\elem_{t/}}X(\langle n;s \rangle)} \\
        {\lim\limits_{[m;u] \in (\Tree[\cO]_{/\langle n-1;t_{\leq n-1} \rangle})^\op} X(\langle m;u \rangle)} & {\lim\limits_{(t \intarrow s) \in (\cO_{n})^\elem_{t/}}\left(\lim\limits_{[m;u] \in (\Tree[\cO]_{/\langle n-1;s_{\leq n-1} \rangle})^\op} X(\langle m;u \rangle) \right)}
        \arrow[from=1-1, to=1-2]
        \arrow["{d_n}"', from=1-1, to=2-1]
        \arrow["{d_n}", from=1-2, to=2-2]
        \arrow[from=2-1, to=2-2]
    \end{tikzcd}\]
    is cartesian.
    By the induction hypothesis, the bottom map is equivalent to $X(\langle n-1;t_{\leq n-1} \rangle) \to \lim_{t \intarrow s \in (\cO_n)^\elem_{t/}} X(\langle n-1;s_{n-1} \rangle)$ and hence this square is equivalent to the square \cref{eq:Tree-vs-forest-pullback-square}, concluding the proof.
\end{proof}

The conditions of \cref{prop:RKan-trees-to-forests} are reminiscent of the definition of an algebrad given in \cref{subsec:double-algebrads}.
However, in the definition of an algebrad one only needs to consider cartesian squares for $n=1$.
We will now show that, under a mild condition on the presheaf $X$, one only needs to consider the squares in \cref{prop:RKan-trees-to-forests} where $n=1$.

\begin{proposition}\label{prop:RKan-trees-to-forests-simplified}
    Let $\cO$ be an algebraic pattern and $X$ a presheaf on $\Forest[\cO]$ such that for any forest $\langle n;t \rangle$, the canonical map
    \[X(\langle n;t \rangle) \to X(\langle 1;t_{0,1} \rangle) \times_{X(\langle 0;t_1 \rangle)} \cdots \times_{X(\langle 0;t_{n-1} \rangle} X(\langle 1;t_{n-1,n} \rangle)\]
    is an equivalence.
    Then $X$ lies in the image of $i_*$ if and only if for every object $x$ in $\cO$, the map $X(\langle 0;x \rangle) \to \lim_{e \in \cO^\elem_{x/}} X(\langle 0;e \rangle)$ is an equivalence and for any forest $\langle 1;t \rangle$ of length 1, the square
    \[\begin{tikzcd}[ampersand replacement=\&]
    	{X(\langle 1;t \rangle)} \& {{\lim\limits_{(t \intarrow s) \in(\cO_1)^\elem_{t/}}} X(\langle 1;s \rangle)} \\
    	{X(\langle 0;t_0 \rangle)} \& {{\lim\limits_{(t \intarrow s) \in(\cO_1)^\elem_{t/}}} X(\langle 0;s_0 \rangle)}
    	\arrow[from=1-1, to=1-2]
    	\arrow["{d_1}"', from=1-1, to=2-1]
    	\arrow["{d_1}", from=1-2, to=2-2]
    	\arrow[from=2-1, to=2-2]
    \end{tikzcd}\]
    is cartesian.
\end{proposition}

\begin{proof}
    A proof similar to \cref{lem:algebrad-n=1-suffices} shows that this implies that for any forest $\langle n;t \rangle$ with $n \geq 1$, the square \cref{eq:Tree-vs-forest-pullback-square} from \cref{prop:RKan-trees-to-forests} is cartesian.
\end{proof}

In case $\cO$ is sound in the sense of \cite[Definition 3.3.4]{BarkanHaugsengSteinebrunner}, one can simplify the conditions further.

\begin{proposition}\label{prop:RKan-trees-to-forest-sound}
    Suppose $\cO$ is sound and let $X$ be a presheaf on $\Forest[\cO]$. Then $X$ lies in the image of $i_*$ if and only if for every forest $\langle n;t \rangle$, the map
    \[X(\langle n;t \rangle) \to \lim\limits_{(t \intarrow s) \in (\cO_{n})^\elem_{t/}} X(\langle n;s \rangle)\]
    is an equivalence.
    If $X$ satisfies the condition from \cref{prop:RKan-trees-to-forests-simplified}, then this only needs to hold for forests of length $\leq 1$.
\end{proposition}

\begin{proof}
    We will show that these conditions are equivalent to those of \cref{prop:RKan-trees-to-forests}.
    The case $n=0$ gives the first condition of \cref{prop:RKan-trees-to-forests}.
    We now show by induction on $k \geq 1$ that
    \[X(\langle n;t \rangle) \to \lim\limits_{(t \intarrow s) \in (\cO_{n})^\elem_{t/}} X(\langle n;s \rangle)\]
    is an equivalence for forests of length $n \leq k$ precisely if the square \cref{eq:Tree-vs-forest-pullback-square} from \cref{prop:RKan-trees-to-forests} is cartesian for all forests of length $n \leq k$.
    Observe that by the induction hypothesis, the bottom map of the square \cref{eq:Tree-vs-forest-pullback-square} can be rewritten as
    \[X(\langle n-1;t_{\leq n-1} \rangle) \to \operatorname{lim}_{s \in (\cO_n)^\elem_{t/}} \operatorname{lim}_{u \in (\cO_{n-1})^\elem_{s_{\leq n-1}/}} X(\langle n-1;u \rangle).\]
    By \cite[Observation 3.3.6]{BarkanHaugsengSteinebrunner} and the equivalences $(\cO_n)^\elem_{t/} \simeq \cO^\elem_{t_n/}$ and $(\cO_{n-1})^\elem_{s_{\leq n-1}/} \simeq \cO^\elem_{s_{n-1}/}$, soundness of $\cO$ implies that this map agrees with
    \[X(\langle n-1;t_{\leq n-1} \rangle) \to \lim_{s \in (\cO_{n-1})^\elem_{t_{\leq n-1}/}} X(\langle n-1;s \rangle),\]
    which is an equivalence by the induction hypothesis.
    In particular, the square \cref{eq:Tree-vs-forest-pullback-square} from \cref{prop:RKan-trees-to-forests} is cartesian precisely if the top map is an equivalence, which finishes the induction.
    It is clear that if $X$ satisfies the conditions from \cref{prop:RKan-trees-to-forests-simplified}, then it suffices to consider forests of length $\leq 1$.
\end{proof}

\section{Algebrads as Segal presheaves on the tree category}\label{sec:algebrads_as_segal_presheaves_on_the_tree_category}

In this section, we will present yet another perspective on algebrads, namely as certain ``complete Segal'' presheaves on the tree category $\Tree[\cO]$.
In the context of \textit{operator categories}, a similar equivalence was exhibited by Barwick \cite{Barwick}.
For algebraic patterns, a similar equivalence was studied by Kern \cite{Kern} as well. 
We fix the following terminology:

\begin{definition}
    Presheaves on $\Tree[\O]$ are called \textit{$\Tree[\O]$-spaces}.
\end{definition}

\subsection{Segal presheaves}
We introduce Segal conditions for $\Tree[\O]$-spaces by making use of the simplicial nature of $\Tree[\O]$. We will see at the end of this section that this terminology matches up with the previous terminology of Segal objects (\cref{def:segal objects}) applied to the pattern $\Tree[\O]^{\op, \natural}$ of \cref{ex:iterated trees}.

\begin{notation}\label{notation:forest as tree space}
    The projection $\Tree[\cO] \to \Delta$ will be denoted by $p$, and the presheaf on $\Tree[\cO]$ represented by a tree $[n;t]$ will be denoted by $[n;t]$.
    More generally, let $\langle n; t\rangle$ be a forest. Then we will consider the $\Tree[\O]$-space 
    $$
    [n;t] \coloneqq i^*\langle n ; t\rangle
    $$
    that is obtained by restricting the presheaf represented by $\langle n; t \rangle$ along the inclusion $i \colon \Tree[\O] \to \Forest[\O]$. Note that this leads to no ambiguity in notation when $t_n \in \O^\el$. We note that there is a natural projection $[n;t] \to p^*[n]$ obtained by applying $i^*$ to $\langle n;t \rangle \to q^*[n]$, where $q$ denotes the projection $\Forest[\cO] \to \Delta$.
\end{notation}

\begin{construction}\label{cons:loose braket}
    Let $Y$ be a simplicial space. Then we consider the functor 
    $$
    (-)\boxtimes_Y(-) \colon \PSh(\Delta)_{/Y} \times \PSh(\Tree[\O])_{/p^*Y} \to \PSh(\Tree[\O])_{/p^*Y}^{\times 2} \xrightarrow{\times_{p^*Y}} \PSh(\Tree[\O]).
    $$
    Since $\PSh(\Tree[\O])$ is locally cartesian closed, this preserves colimits in both variables.
    In practice, we will use this construction when $Y = [n]$, so that we obtain a colimit preserving functor 
    $$
    (-) \boxtimes_{[n]} [n;t] \colon \PSh(\Delta)_{/[n]} \to \PSh(\Tree[\O])
    $$
    for every forest $\langle n;t \rangle$.
\end{construction}

We will make extensive use of the following \textit{forest decomposition formulas}:

\begin{lemma}\label{prop:forest decomposition}
    Let $\langle n;t\rangle$ be a forest, and suppose that $X$ is a simplicial space over $[n]$.  Then we have the following formulas:
    \begin{itemize}
        \item if $n = 0$, then the canonical map 
    $$
    \colim_{e \in (\O^{\el,\op})_{/t}}X \boxtimes_{[0]} [0;e] \to X \boxtimes_{[0]} [0;t]
    $$
    is an equivalence, and
        \item in the case that $n \geq 1$, the commutative square 
    \[
    \begin{tikzcd}
        \colim_{s \in (\O^{\el,\op}_n)_{/t}} X_{\leq n-1} \boxtimes_{[n-1]} [n-1;s_{\leq n-1}] \arrow[r]\arrow[d] & \colim_{s \in (\O^{\el,\op}_n)_{/t}} X \boxtimes_{[n]} [n;s] \arrow[d] \\
        X_{\leq n-1} \boxtimes_{[n-1]} [n-1;t_{\leq n-1}]  \arrow[r] &  X \boxtimes_{[n]} [n;t]
    \end{tikzcd}
    \]
   is a pushout square, where $X_{\leq n-1} \coloneqq d_n^*X$.
    \end{itemize}
\end{lemma}
\begin{proof}
    Recall that the functor $X \boxtimes_{[n]}(-)$ preserves colimits and that $X_{\leq n-1} \boxtimes_{[n-1]} [n-1;-] = X \boxtimes_{[n]} [n-1;-]$ by pullback pasting. Consequently, if the formulas hold for the absolute case that $X = [n]$ when there are no box-products, then they will hold for general $X$. To check this absolute case, we can apply $\Map_{\PSh(\Tree[\O])}(-,Y)$, where $Y$ is an arbitrary $\Tree[\O]$-space, to the diagrams in both bullet points. In this case, we recover the  formulas for $i_*Y$ of \cref{prop:RKan-trees-to-forests}.
\end{proof}

\begin{remark}\label{rem:forest decomposition sound}
    If $\O$ is sound, then it follows from \cref{prop:RKan-trees-to-forest-sound} that the canonical map 
    $$
    \colim_{s\in(\O^{\el,\op}_n)_{/t}} X \boxtimes_{[n]} [n;s] \to X \boxtimes_{[n]} [n;t]
    $$
    is an equivalence for every forest $\left<n;t\right>$ and simplicial space $X$ over $[n]$.
\end{remark}

\begin{lemma}\label{lemma:box with forest is forest}
    Let $\langle n;t\rangle$ be a forest. For every map $\phi \colon [m] \to [n]$, there is  a natural equivalence 
    $$
    [m] \boxtimes_{[n]} [n;t] \simeq [m; \phi^*t].
    $$
    If $X \to [n]$ is a map between simplicial spaces, then we have a colimit expression 
    $$\colim_{(\phi \colon [k]\to [n])\in (\Delta_{/[n]})_{/X}}[k;\phi^*{t}] \xrightarrow{\simeq} X\boxtimes_{[n]} [n;t].$$
\end{lemma}
\begin{proof}
    Let $[k;s]$ be a tree. Then there is a natural pullback square
    \[
        \begin{tikzcd}
            \map_{\PSh(\Tree[\O])}([k;s], [m] \boxtimes_{[n]} [n;t]) \arrow[r]\arrow[d] & \map_{\Forest[\O]}(\left<k;s\right>, \langle n;t \rangle) \arrow[d]  \\
            \map_{\Delta}([k], [m]) \arrow[r,"\phi \circ(-)"] & \map_{\Delta}([k], [n]),
        \end{tikzcd}
    \]
    thus $[m] \boxtimes_{[n]} [n;t] \to [n;t]$ must be represented by the cartesian transport of $\langle n;t\rangle $ along $\phi$ with respect to the projection $\Forest[\O] \to \Delta$. The second assertion now follows immediately.
\end{proof}

\begin{definition}
    \label{definition: segal tree space}
    We recall that we have the following special morphisms between simplicial spaces:

    \begin{enumerate}[label=(\roman*)]
    \item for $n \geq 0$, a \textit{spine inclusion} $\Sp[n] \coloneqq [1] \cup_{[0]} \dotsb \cup_{[0]} [1] \to [n]$, 
    \item  the \textit{generating completeness extension} $J\to [0]$, where $J$ is the simplicial set defined by the pushout square 
    \[
        \begin{tikzcd}[column sep = large]
            {[1]}\sqcup [1] \arrow[r, "\{0 \leq 2\} \sqcup \{1\leq 3\}"]\arrow[d] & {[3]} \arrow[d] \\
            {[0]} \sqcup [0] \arrow[r] & J.
        \end{tikzcd}
    \]
\end{enumerate}
Thanks to \cref{cons:loose braket}, these induce the following classes of maps between $\Tree[\O]$-spaces:
\begin{enumerate}
    \item for every tree $[n;t]$, a \textit{spine inclusion}  $$\Sp[n;t] \coloneqq \Sp[n] \boxtimes_{[n]} [n;t] \to [n;t],$$
    \item for every root $[0;e]$, a \textit{generating completeness extension} $$[J;e] \coloneqq J \boxtimes_{[0]}[0;e]\to [0;e].$$
\end{enumerate}
\end{definition}

\begin{definition}\label{def:segal presheaves}
A map $f \colon X\to Y$ between $\Tree[\O]$-spaces is called a \textit{Segal fibration} if it has the unique right lifting property with respect to the spine inclusions. Additionally, $f$ is called a \textit{complete Segal fibration} if it is a Segal fibration and it has the unique right lifting property with respect to the generating completeness extensions. 

An $\Tree[\O]$-space $X$ is called a \textit{(complete) Segal $\Tree[\O]$-space} if the unique map $X \to \ast$ is a (complete) Segal fibration. We will write $$
\CSeg(\Tree[\O]) \subset \Seg(\Tree[\O]) \subset \PSh(\Tree[\O])
$$
for the full subcategories spanned by the complete Segal $\Tree[\O]$-spaces and Segal $\Tree[\O]$-spaces, respectively.

A map $A\to B$ between $\Tree[\O]$-spaces is called a \textit{(complete) Segal extension} if it has the unique left lifting property with respect to the (complete) Segal fibrations.
\end{definition}

On account of \cite[Example 5.5.5.6 \& Proposition 5.5.5.7]{HTT}, the Segal extensions form the left part of a factorization system on $\PSh(\Tree[\cO])$, whose right part consists of the Segal fibrations. Moreover, the Segal extensions are the smallest \emph{saturated} class containing the spine inclusions. Analogous statements hold for complete Segal extensions and complete Segal fibrations.

\begin{remark}\label{remark:saturated-vs-strongly-saturated}
    By \cite[Proposition 5.5.4.15]{HTT}, the category $\CSeg(\Tree[\cO])$ is a reflective subcategory of $\PSh(\Tree[\cO])$ and thus admits a reflection functor $L \colon \PSh(\Tree[\cO]) \to \CSeg(\Tree[\cO])$.
    The functor $L$ carries all complete Segal extensions to equivalences.
    On the other hand, if $i \colon A \to B$ is a map of $\Tree[\cO]$-spaces such that $Li$ is an equivalence, then it is not necessarily true that $i$ is a complete Segal extension.
    However, if $B$ is a complete Segal object, then this \emph{does} hold by the following standard argument:
    First, factor $A \to B$ as a complete Segal extension $A \to B'$ followed by a complete Segal fibration $B' \to B$. By 2-out-of-3, the map $LB' \to LB$ is an equivalence. Since $B'$ is a complete Segal object, it follows that $B' \to B$ is an equivalence and hence that  $A \to B$ is a complete Segal extension.
\end{remark}

Recall that for a sequence of active morphisms $t = (t_0 \actarrow \cdots \actarrow t_n)$ and $0 \leq i < j \leq n$, we write $t_{i,j}$ for the subsequence $(t_i \actarrow \cdots \actarrow t_j)$.

\begin{remark}\label{remark:unpacking-complete-segal-condition}
    By \cref{lemma:box with forest is forest}, we have an equivalence
    \[\Sp[n;t] \simeq [1;t_{0,1}] \cup_{[0;t_1]} \cdots \cup_{[0;t_{n-1}]} [1;t_{n-1,n}].\]
    It therefore follows that a $\Tree[\cO]$-space $X$ is Segal if and only if its right Kan extension $i_*X$ has the property that for any tree $[n;t]$, the canonical map
    \[i_*X(\langle n;t \rangle) \to i_*X(\langle 1;t_{0,1} \rangle) \times_{i_*X(\langle 0;t_{1} \rangle)} \cdots \times_{i_*X(\langle 0;t_{n-1} \rangle)} i_*X(\langle 1;t_{n-1,n} \rangle)\]
    is an equivalence.
    It similarly follows that $X$ is complete if and only if
    \[X([0;e]) \to X([3;e=\cdots=e]) \times_{X([1;e=e]) \times X([1;e=e])} (X([0;e]) \times X([0;e]))\]
    is an equivalence for every $e$ in $\cO^\elem$.
\end{remark}

\begin{remark}
    Suppose that $p \colon X\to Y$ is a (complete) Segal fibration. If $Y$ is a  (complete) Segal $\Tree[\O]$-space, then so is $X$.
\end{remark}

\begin{proposition}\label{prop:simplicial segal extension}
    Let $\left<n;t\right>$ be a forest. If $X \to Y$ is a (complete) Segal extension between simplicial spaces over $[n]$, then 
    the induced map
    $$X \boxtimes_{[n]} [n;t] \to Y \boxtimes_{[n]} [n;t]$$
    between $\Tree[\O]$-spaces is a (complete) Segal extension as well.
\end{proposition}
\begin{proof}
    First suppose $X \to Y$ is a Segal extension.
    As $(-) \boxtimes_{[n]} [n;t]$ preserves colimits, it suffices to handle the case that $X \to Y$ is given by $\Sp[m] \to [m]$. If $\phi \colon [m] \to [n]$ denotes the underlying map, then the induced map in question is identified with the canonical map
    $\Sp[m] \boxtimes_{[m]} [m; \phi^*t] \to [m; \phi^*t]$
    on account of \cref{lemma:box with forest is forest}.
    All in all, we may thus reduce to the case that $X \to Y$ is given by $\Sp[n] \to [n]$, so that we have to show that the map
    \[
        \Sp[n] \boxtimes_{[n]} [n;t] \to [n;t]
    \]
    is a Segal extension for every forest $\left<n;t\right>$.
    In light of \cref{prop:forest decomposition}, we see that an induction argument on $n$ allows us to further reduce to the case that $[n;t]$ is a tree. Then the map is a Segal extension by definition.
    For the claim about complete Segal extensions, it suffices to consider the case that $X \to Y$ is of the form $J \to [0]$.
    This follows by an analogous argument.
\end{proof}

\begin{proposition}\label{prop:boxtimes with exponentiable map}
    Let $X \to Y$ be an exponentiable map between Segal spaces. If $A \to B$ is a Segal extension between $\Tree[\O]$-spaces over $Y$,  then the induced map $$X \boxtimes_Y A \to X \boxtimes_Y B$$
    between $\Tree[\O]$-spaces is a Segal extension as well.
\end{proposition}
\begin{proof}
    We may reduce to the case that $A\to B$ is given by the spine inclusion $\Sp[n;t] \to [n;t]$. The map of the proposition can be rewritten as the map obtained by applying $(-)\boxtimes_{[n]} [n;t]$ to the map $i : X \times_Y \Sp[n] \to X \times_Y [n]$ induced by the underlying map $[n]\to Y$. The assumption implies that $i$ is a Segal extension between simplicial spaces, so the desired conclusion follows from \cref{prop:simplicial segal extension}.
\end{proof}

As promised at the beginning of this subsection, we include the following observation for thoroughness:

\begin{proposition}\label{prop:segal tree spaces are segal objects}
    Let $X$ be a $\Tree[\O]$-space. Then $X$ is a Segal $\Tree[\O]$-space in the sense of \cref{def:segal presheaves} if and only if it is a Segal space for the algebraic pattern $\Tree[\O]^{\op, \natural}$ in the sense of \cref{def:segal objects}.
\end{proposition}
\begin{proof}
    Let $[n;t]$ be a tree.
    Then one may consider the canonical span
    of functors
    $$
     \Tree[\O]^{\el}_{/[n;t]} \xleftarrow{\simeq} \Tree[\O]^{\el}_{/\Sp[n;t]} \xrightarrow{i} \Tree[\O]_{/\Sp[n;t]}.
    $$
    This induces a natural comparison map
    $
    \colim_{[m;s] \in \Tree[\O]^{\el}_{/[n;t]}} [m;s] \to \Sp[n;t]
    $
    in $\PSh(\Tree[\O])$.
    The proposition follows if we show that this map is an equivalence, so we will show that $i$ is final.
  
    If $n \geq 2$, then we may write $i$ as an iterated pushout
    \[\begin{tikzcd}[row sep=small]
    	{\Tree[\O]^{\el}_{/[1;t_{01}]} \cup_{\Tree[\O]^{\el}_{/[0;t_1]}} \dotsb \cup_{\Tree[\O]^{\el}_{/[0;t_{n-1}]}}\Tree[\O]^{\el}_{/[1;t_{n-1,n}]}} \\
    	{\Tree[\O]_{/[1;t_{01}]} \cup_{\Tree[\O]_{/[0;t_1]}} \dotsb \cup_{\Tree[\O]_{/[0;t_{n-1}]}}\Tree[\O]_{/[1;t_{n-1,n}]}.}
    	\arrow[from=1-1, to=2-1]
    \end{tikzcd}\]
    Thus we can reduce to the cases $n = 0,1$.
    Note that $\Tree[\cO]_{/[m;s]} \simeq \Tree[\cO]_{/\langle m;s \rangle}$ since we defined $[m;s]$ as the restriction $i^*\langle m;s \rangle$ for any forest $\langle m;s \rangle$.
    The case $n=0$ now follows directly from \cref{lem:lemma1-correspondence}. 
    
    Now suppose that $n=1$. Let $\Tree[\cO]^\bot_{/[1;t]}$ be the full subcategory of $\Tree[\cO]_{/[1;t]}^\elem \subset \Tree[\cO]_{/[1;t]}$ spanned by the maps $[m;s] \to [1;t]$ whose underlying map $\phi \colon [m] \to [1]$ preserves $0$, so $\phi$ is either the the identity or the inclusion of $0$.
    Observe that by the same proof as \cref{lemma:factorization system on the tree category}, any map $[m;s] \to [1;t]$ uniquely factors as a composite $[m;s] \to [i;s'] \to [1;t]$ such that $[m] \to [i]$ preserves the top element and $s'_i \intarrow s_m$ is invertible, and the second map lies in $\Tree[\cO]^\bot_{/[1;t]}$.
    This shows that the inclusions $\Tree[\cO]^\bot_{/[1;t]} \hookrightarrow \Tree[\cO]^\elem_{/[1;t]}$ and $\Tree[\cO]^\bot_{/[1;t]} \hookrightarrow \Tree[\cO]_{/[1;t]}$ both admit a left adjoint, hence they are final.
    We conclude that the inclusion $\Tree[\cO]^\elem_{/[1;t]} \hookrightarrow \Tree[\cO]_{/[1;t]}$ is final.
\end{proof}

\subsection{Functoriality}\label{subsection:functoriality of algebrad cats} Every map $f \colon \O \to \P$ between algebraic patterns induces a functor $\Tree[f] \colon \Tree[\O]\to\Tree[\P]$ between their tree categories. In turn, we obtain an adjunction $$
    \Tree[f]_! : \PSh(\Tree[\O]) \rightleftarrows \PSh(\Tree[\P]) : \Tree[f]^*.
$$
One might ask whether this derives to an adjunction between the categories of complete Segal presheaves. We recall the following notion from \cite[\S 4]{ChuHaugseng2021HomotopycoherentAlgebraSegal}:

\begin{definition}[Chu--Haugseng]
    A map  $f \colon \O \to \P$ between algebraic patterns is called \textit{strong Segal}  if for every object $t \in \O$, the induced map 
    $
    \O^\el_{t/} \to \P^\el_{f(t)/}
    $
    is initial.
\end{definition}

\begin{proposition}
    If $f \colon \O \to \P$ is strong Segal, then the canonical map $$\Tree[f]_!(X \boxtimes_{[n]} [n;t]) \to X \boxtimes_{[n]} [n; f(t)]$$ in $\PSh(\Tree[\P])$ is an equivalence for every $\O$-forest $\left<n;t\right>$ and every simplicial space $X$ over $[n]$.
\end{proposition}
\begin{proof}
    In light of \cref{lemma:box with forest is forest}, it suffices to show this for $X = [n]$. We may now proceed inductively using the forest decomposition formula of \cref{prop:forest decomposition}, and use that the map $(\O^{\el,\op}_n)_{/t} \to (\P_n^{\el, \op})_{/f(t)}$ is final for every $t \in \O_n$ and $n\geq 0$.
\end{proof}

In particular, $\Tree[f]_!$ preserves the spine inclusions and generating completeness extensions.
We conclude the following.

\begin{corollary}\label{rem:base change decorated Segal spaces}
    If $f \colon \O \to \P$ is a strong Segal map between algebraic patterns, then the adjunction $\Tree[f]_! \dashv \Tree[f]^*$
    derives to an adjunction 
    $$
    f_! : \CSeg(\Tree[\O]) \rightleftarrows \CSeg(\Tree[\P]) :  f^*.
    $$
    In particular, $\Tree[f]^*$ preserves complete Segal objects.\qed
\end{corollary}

\begin{remark}
    We can now rephrase the completeness condition of \cref{def:segal presheaves} as follows. Suppose that $e \in \O$ is an elementary. Then this corresponds to a map $e \colon \ast \to \O$ between algebraic patterns where $\ast$ is the terminal algebraic pattern. It is automatic that this is a strong Segal morphism, so that $e$ induces a functor
    $\Tree[e]^* \colon \Seg(\Tree[\O]) \to \Seg(\Delta)$. A Segal $\Tree[\O]$-space $X$ is complete if and only if for every elementary $e \in \O$, the Segal space $\Tree[e]^*X$ is complete in the sense of Rezk \cite{Rezk}.
\end{remark}

Using the initiality of $\cO^\elem_{t/} \to \cP^\elem_{f(t)/}$, one also obtains the following analogue of \cref{rem:base change decorated Segal spaces} for algebrads.
We write $\Cocart^\inert(\cO)$ for the category of functors $\cC \to \cO$ that are cocartesian over $\cO^\inert$ (cf.\ \cref{remark:algebrads-all-elementary}).

\begin{lemma}\label{cor:base change factorization algebrads}
    If $f \colon \cO \to \cP$ is a strong Segal morphism between algebraic patterns, then the pullback functor $f^* \colon \Cocart^\inert(\cP) \to \Cocart^\inert(\cO)$ preserves algebrads. In particular, it restricts to a functor $f^* \colon \Algd(\cP) \to \Algd(\cO)$. \qed
\end{lemma}

\subsection{The equivalence with algebrads} 
We will now show \cref{thmD} from the introduction, namely that algebrads coincide with complete Segal $\Tree[\O]$-spaces.

\begin{lemma}\label{lem:fiberwise lfibs}
    Suppose that we have a map 
    \[
        \begin{tikzcd}[column sep = tiny, row sep = small]
            E\arrow[rr, "f"]\arrow[dr] && E'\arrow[dl] \\
            & \C
        \end{tikzcd}
    \]
    between cocartesian fibrations on $\C$. Then $f$ is a left fibration if and only if $f$ is a fiberwise left fibration, i.e.\ $f_x\colon E_x \to E'_x$ is a left fibration for every $x \in \C$.
\end{lemma}
\begin{proof}
    This is an easy exercise, or, alternatively, the statement is obtained by combining Propositions 2.4.2.11, 2.4.2.8 and 2.4.2.4 of \cite{HTT}.
\end{proof}

\begin{theorem}\label{prop:comparison algebrads vs segal spaces}
    There exists an equivalence
    $$
        \Algd(\O) \simeq \CSeg(\Tree[\O]).
    $$
    that is natural in strong Segal maps between algebraic patterns.
\end{theorem}

\begin{proof}
    We will first prove the equivalence $\Algd(\cO) \simeq \CSeg(\Tree[\cO])$, and afterwards show that it is natural in strong Segal maps.
    The straightening equivalence induces an equivalence on slice categories
    $\Fun(\Delta^\op, \Cat)_{/\O} \to \Cocart(\Delta^\op)_{/\Forest[\O]^\op}$.
    Let us write $\C \subset \Fun(\Delta^\op, \Cat)_{/\O}$ for the full subcategory spanned by the maps $X \to \O$ between simplicial categories such that $X_n \to \O_n$ is a left fibration for every $n$.
    On account of \cref{lem:fiberwise lfibs}, the restriction of the straightening equivalence to $\C$ factors through the inclusion $\LFib(\Forest[\O]^\op) \to \Cocart(\Delta^\op)_{/\Forest[\O]^\op}$ via an equivalence
    \[\C \to \LFib(\Forest[\O]^\op) \simeq \PSh(\Forest[\O]).\]
    Unwinding the definitions, we see that the essential image of the algebrads $\Algd(\O) \subset \C$ is given by the full subcategory 
    $$
    \CSeg(\Forest[\O]) \subset  \PSh(\Forest[\O])
    $$ that is spanned by those presheaves $X$ such that
    \begin{enumerate}[(1)]
        \item\label{item1:algd-tree-equiv} for any forest $\left<n;t\right>$, the map
        \[X(\left<n;t\right>) \to X(\left<1;t_{0,1}\right>) \times_{X(\left<0;t_1\right>)} \cdots \times_{X(\left<0;t_{n-1}\right>} X(\left<1;t_{n-1,n}\right>)\]
        is an equivalence,
        \item\label{item2:algd-tree-equiv} for any $x$ in $\cO$, the map
        \[X(\left<0;x\right>) \to X(\left<3;x=\cdots=x\right>) \times_{X(\left<1;x=x\right>) \times X(\left<1;x=x\right>)} (X(\left<0;x\right>) \times X(\left<0;x\right>))\]
        is an equivalence,
        \item\label{item3:algd-tree-equiv} for every $x$ in $\cO$, the map
        \[X(\left<0;x\right>) \to \lim_{e \in \cO^\elem_{x/}}X(\left<0;e\right>)\]
        is an equivalence, and
        \item\label{item4:algd-tree-equiv} for every forest $\left<1;t\right>$ of length 1, the map
        \[X(\left<1;t\right>) \to X(\left<0;t_0\right>) \times_{\lim_{s \in (\cO_1)^\elem_{t/}} X(\left<0;s_0\right>)} \lim_{s \in (\cO_1)^\elem_{t/}} X(\left<1;s\right>)\]
        is an equivalence.
    \end{enumerate}
    Items \cref{item1:algd-tree-equiv} and \cref{item2:algd-tree-equiv} correspond precisely to the condition that $X$ comes from a complete double category.
    Conditions \cref{item3:algd-tree-equiv} and \cref{item4:algd-tree-equiv} correspond to the last two conditions of \cref{defi:of algebrad for double pattern}.
    
    We will now show that the fully faithful functor $i_* \colon \PSh(\Tree[\cO]) \to \PSh(\Forest[\cO])$ identifies $\CSeg(\Tree[\cO])$ with $\CSeg(\Forest[\cO])$.
    It follows from \cref{prop:RKan-trees-to-forests-simplified} that any object in $\CSeg(\Forest[\cO])$ is of the form $i_*X$ for some $\Tree[\cO]$-space $X$, and it follows from \cref{remark:unpacking-complete-segal-condition} that $X$ lies in $\CSeg(\Tree[\cO])$.
    It therefore remains to show that for any complete Segal $\Tree[\cO]$-space $X$, its right Kan extension $i_*X$ lies in $\CSeg(\Forest[\cO])$.
    Observe that items \cref{item3:algd-tree-equiv} and \cref{item4:algd-tree-equiv} hold by \cref{prop:RKan-trees-to-forests}.
    By \cref{prop:simplicial segal extension}, $X$ is local with respect to $\Sp[n;t] \to [n;t]$ for any forest $[n;t]$, hence condition \cref{item1:algd-tree-equiv} holds.
    Finally, it follows from \cref{prop:forest decomposition} that $X$ is local with respect to $J \boxtimes_{[0]} [0;x] \to [0;x]$ for any $x$ in $\cO$, hence condition \cref{item2:algd-tree-equiv} holds for $i_*X$.
    We conclude that $\CSeg(\Tree[\cO]) \simeq \CSeg(\Forest[\cO]) \simeq \Algd(\cO)$.
    
    To prove that the equivalence is natural, write $\AlgPatt^\mathrm{Seg}$ for the wide subcategory of $\AlgPatt$ spanned by the strong Segal morphisms.
    It follows from \cref{cor:base change factorization algebrads} that we have a functor $\Algd(-) \colon \AlgPatt^\mathrm{Seg,op} \to \Cat$.
    We wish to show that this is equivalent to the functor
    \[\Psi \colon \AlgPatt^\mathrm{Seg,op} \to \Cat; \quad \cO \mapsto \CSeg(\Tree[\cO])\]
    that exists by \cref{rem:base change decorated Segal spaces} as a subfunctor of $\PSh(-)$.
    Note that $\Psi$ sends $f \colon \cP \to \cO$ to the functor $\CSeg(\Tree[\cO]) \to \CSeg(\Tree[\cP])$ that restricts along $\Tree[f]$.
    A proof similar to \cref{rem:base change decorated Segal spaces}, combined with \cref{prop:RKan-trees-to-forests}, shows that this agrees with restriction along $\Forest[f]$ under the equivalence $\CSeg(\Forest[\cO]) \simeq \CSeg(\Tree[\cO])$.
    Observe that the cartesian fibration corresponding to the functor $\PSh(\Forest[-])$ is given by the pullback
    \[\begin{tikzcd}
    	{\mathscr{E}} & \LFib \\
    	{\AlgPatt^\Seg} & \Cat,
    	\arrow[from=1-1, to=1-2]
    	\arrow["p"', from=1-1, to=2-1]
    	\arrow["\lrcorner", very near start, phantom, draw=none, from=1-1, to=2-2]
    	\arrow["{\mathrm{target}}", from=1-2, to=2-2]
    	\arrow["{\Forest[-]^\op}"', from=2-1, to=2-2]
    \end{tikzcd}\]
    so the cartesian fibration corresponding to the functor $\Psi$ is the subfibration of $p$ whose fiber over an algebraic pattern $\cO$ is $\CSeg(\Forest[\cO]) \subset \PSh(\Forest[\cO]) \simeq \LFib(\Forest[\cO]^\op)$.
    Let us write $\overline{\Psi} \colon \mathscr{A} \to \AlgPatt^\Seg$ for this cartesian fibration.
    Since $\Forest[\cO]^\op \to \Delta^\op$ is constructed as the cocartesian fibration corresponding to the double category $\cO^\dbl$, we may identify $p$ with the pullback
    \[\begin{tikzcd}
    	\mathscr{E} & {\Ar^\LFib(\Fun(\Delta^\op,\Cat)) \times_{\Fun(\Delta^\op,\Cat)} \DblCat} & {\Ar^\LFib(\Cocart(\Delta^\op))} \\
    	{\AlgPatt^\Seg} & \DblCat & {\Cocart(\Delta^\op)}
    	\arrow[from=1-1, to=1-2]
    	\arrow[from=1-1, to=2-1,"p"']
    	\arrow["\lrcorner"{very near start}, phantom, from=1-1, to=2-2]
    	\arrow[from=1-2, to=1-3, hook]
    	\arrow[from=1-2, to=2-2,"{\mathrm{target}}"]
    	\arrow["\lrcorner"{very near start}, phantom, from=1-2, to=2-3]
    	\arrow[from=1-3, to=2-3,"{\mathrm{target}}"]
    	\arrow[from=2-1, to=2-2,"{(-)^\dbl}"]
    	\arrow[hook, from=2-2, to=2-3,"\smallint_{\Delta^\op}(-)"]
    \end{tikzcd}\]
    where $\Ar^\LFib(-)$ denotes the full subcategory of the arrow category spanned by componentwise or fiberwise left fibrations.
    Using the equivalence $\AlgPatt \simeq \DblPatt$ from \cref{cor:equivalence between double and algebraic pattern} and unwinding the equivalence $\Algd(\cO) \simeq \CSeg(\Forest[\cO])$ proved above, it follows that the cartesian fibration $\overline{\Psi} \colon \mathscr{A} \to \AlgPatt^\Seg$ is equivalent to the functor $\Algd \to \AlgPatt^\Seg$ defined as follows:
    \begin{itemize}
        \item $\Algd$ is the subcategory of $\Ar(\AlgPatt)$
        \begin{itemize}
            \item whose objects are algebrads $\cP \to \cO$ such that $\cP$ has the algebraic pattern structure from \cref{def:category-of-algebrads},
            \item whose morphisms are squares
            \[\begin{tikzcd}
            	{\cP'} & \cP \\
            	{\cO'} & \cO
            	\arrow[from=1-1, to=1-2]
            	\arrow[from=1-1, to=2-1]
            	\arrow[from=1-2, to=2-2]
            	\arrow[from=2-1, to=2-2]
            \end{tikzcd}\]
            such that $\cO' \to \cO$ is a strong Segal morphism, and
        \end{itemize}
        \item the functor $\Algd \to \AlgPatt^\Seg$ is the target projection.
    \end{itemize}
    This cartesian fibration precisely encodes the functor $\Algd(-) \colon \AlgPatt^\mathrm{Seg} \to \Cat$.
\end{proof}

\subsection{Unraveling the equivalence}\label{ssec:unraveling} We will now give an explicit description of the functor underlying the equivalence $\Algd(\cO) \simeq \CSeg(\Tree[\cO]^\op)$ from \cref{prop:comparison algebrads vs segal spaces}.

Let $\cO$ be an algebraic pattern, $\cP \to \cO$ an algebrad and $X \colon \Tree[\cO]^\op \to \Spc$ its corresponding complete Segal $\Tree[\cO]$-space under the equivalence $\Algd(\cO) \simeq \CSeg(\Tree[\cO])$.
Then $X$ is obtained from $\cP$ by first considering its associated double category $\Sq_{L,R}(\cP) \to \cO^\dbl$, then unstraightening this over $\Delta^\op$ to obtain a left fibration $\smallint_{\Delta^\op} \Sq_{L,R}(\cP) \to \Forest[\cO]^\op$, subsequently straightening this to obtain a presheaf on $\Forest[\cO]$ and finally restricting along $i \colon \Tree[\cO] \to \Forest[\cO]$.
Let $[n;t]$ be a tree and write $\overline{t} \colon [n] \to \cO^\activ$ for the corresponding functor.
It follows by construction that the value of $X$ at a tree $[n;t]$ is given by
\[X([n;t]) \simeq \Map_{/\cO}(\overline{t}, \cP).\]
Write $\Cocart^\inert(\cO)$ for the subcategory of $\Cat_{/\cO}$ spanned by functors that admit cocartesian lifts of inerts, and whose morphisms preserve cocartesian lifts of inerts.
Let us write $[t]^\inert = [n] \times_{\cO} \Ar^\inert(\cO)$.
Here $\Ar^\inert(\cO) \subset \Ar(\cO)$ is the full subcategory spanned by the inert morphisms, and the pullback is taken along $\overline{t} \colon [n] \to \cO$ and $\ev_0 \colon \Ar^\inert(\cO) \to \cO$.
By \cite[Corollary 2.1.5]{BarkanHaugsengSteinebrunner}, the target projection $[t]^\inert \to \cO$ is the free cocartesian fibration over $\cO^\inert$ on $\overline{t}$, hence there is an equivalence
\[X([n;t]) \simeq \Map_{/\cO}(\overline{t},\cP) \simeq \Map_{\Cocart^\inert(\cO)}([t]^\inert, \cP).\]
We obtain the following explicit description of the equivalence $\Algd(\cO) \simeq \CSeg(\Tree[\cO])$ as the ``nerve'' with respect to the functor $\Tree[\cO] \to \Cocart^\inert(\cO)$ that sends $[n;t]$ to $[t]^\inert$.

\begin{proposition}\label{prop:equivalence-cseg-algad-is-nerve}
    Let $\cO$ be an algebraic pattern.
    Then the category $\Tree[\cO]$ is equivalent to the full subcategory of $\Cocart^\inert(\cO)$ spanned by objects of the form $[t]^\inert$ for $[n;t]$ a tree,
    and the equivalence $\Algd(\cO) \xrightarrow{\simeq} \CSeg(\Tree[\cO])$ is given by the functor
    \[\cP \mapsto \left([n;t] \mapsto \Map_{\Cocart^\inert(\cO)}([t]^\inert, \cP)\right).\]
\end{proposition}

\begin{proof}
    First consider the case where every object of $\cO$ is elementary, so $\Tree[\cO] = \Forest[\cO]$ and $\Algd(\cO) = \Cocart^\inert(\cO)$.
    In this case it follows from \cref{exa:all-elementary-then-robust} and \cref{cor:fibrant forest} that all trees $[n;t]$ lie in $\CSeg(\Tree[\cO])$. (Alternatively, this is straightforward to show by hand.) 
    Write $\Psi$ for the equivalence $\Cocart^\inert(\cO) = \Algd(\cO) \xrightarrow{\simeq} \CSeg(\Tree[\cO]) $ from \cref{prop:comparison algebrads vs segal spaces}.
    By the discussion above, we have natural equivalences
    \[
        \Hom_{\Cocart^\inert(\cO)}(\Psi^{-1}[n;t], \cP) \simeq \Psi(\cP)([n;t]) \simeq \Hom_{\Cocart^\inert(\cO)}([t]^\inert, \cP).
    \]
    This proves the proposition when every object of $\cO$ is elementary.
    For a general algebraic pattern $\cO$, write $\cO^\sharp$ for the algebraic pattern structure with the same inert-active factorization system, but where every object is elementary.
    The result then follows for $\cO$ since \cref{prop:comparison algebrads vs segal spaces} identifies $\CSeg(\Tree[\cO])$ with the full subcategory of $\CSeg(\Tree[\cO^\sharp])$ spanned by those objects that are right Kan extended along $\Tree[\cO] \hookrightarrow \Tree[\cO^\sharp] = \Forest[\cO]$.
\end{proof}

\begin{remark}
    The category $\Algd(\cO)$ is a reflective localization of $\Cocart^\inert(\cO)$; write $L$ for this localization.
    Then we obtain a functor $\Tree[\cO] \to \Algd(\cO)$ given by $[n;t] \mapsto L([t]^\inert)$, and the equivalence $\Algd(\cO) \simeq \CSeg(\Tree[\cO])$ is given by the nerve with respect to this functor.
\end{remark}

\begin{remark}\label{rem:boxproduct-algebrad-side}
    Let $\cP$ be an $\cO$-algebrad and $X$ its corresponding complete Segal $\Tree[\cO]$-space.
    Suppose that $t = t_0 \actarrow \dotsb \actarrow t_n$ is a string of active arrows in $\cO$ corresponding to the functor $\overline{t} \colon [n] \to \cO^\activ$ and let $\phi \colon [m] \to [n]$ be a map in $\Delta$.
    The equivalence $X([n;t]) \simeq \Hom_{/\cO}(\overline{t}, \cP)$ generalizes to an equivalence
    \[\Map([m] \boxtimes_{[n]} [n;t], X) \simeq X([m;\phi^*(t)]) \simeq \Map_{/\cO}(\overline{\phi^*(t)}, \cP),\]
    where $\overline{\phi^*(t)}$ denotes the composite $[m] \xrightarrow{\phi} [n] \xrightarrow{\overline{t}} \cO^\activ$ corresponding to the string $\phi^*(t) = t_{\phi(0)} \actarrow t_{\phi(1)} \actarrow \dotsb \actarrow t_{\phi(m)}$.
\end{remark}

\section{Exponentiable maps between algebrads}\label{sec:exponentiable-objects}

This section is dedicated to the study of 
the exponentiable objects and morphisms in the category of $\cO$-algebrads. In particular, we establish \cref{thmA}.

\begin{definition}
Let $\cC$ be a category that admits finite limits.
A morphism $f \colon x \to y$ in $\cC$ is called \textit{exponentiable} if $f^*\colon \cC_{/y} \to \cC_{/x}$ is a left adjoint.
An object $x$ is \textit{exponentiable} if the unique morphism $x \to \ast$ to the terminal object is exponentiable.
\end{definition}

Our goal is to show the following theorem:

\begin{theorem}\label{thm:main-thm-expo-CSeg}
    Let $\cO$ be an algebraic pattern. A map $f \colon X \to Y$ in $\CSeg(\Tree[\cO])$ is exponentiable if the following condition is satisfied:
    \begin{enumerate}[(CC)]
        \item\label{conditionCC} For every tree $[2;t] \to Y$ of length $2$, the following map is an equivalence:
        \[\colim_{[k] \in \Delta^\op} \map_{/Y}([1+k+1;t_0 \actarrow \underbrace{t_1 = \cdots = t_1}_{k\text{-times}} \actarrow t_2],X) \to \map_{/Y}([1;t_0 \actarrow t_2] ,X).\]
    \end{enumerate}
\end{theorem}

\begin{remark}
    The proof of \cref{thm:main-thm-expo-CSeg} will show that the same result also holds when replacing the category $\CSeg(\Tree[\cO])$ with the larger category $\Seg(\Tree[\cO])$ that also includes non-complete Segal $\Tree[\O]$-spaces.
\end{remark}

\subsection{Exponentiability in the category of algebrads}

Before embarking on the proof of \cref{thm:main-thm-expo-CSeg}, we will use the equivalence $\CSeg(\Tree[\cO]) \simeq \Algd(\cO)$ from \cref{prop:comparison algebrads vs segal spaces} to describe the exponentiable morphisms in $\Algd(\cO)$.
In particular, this proves \cref{thmA}.

\begin{lemma}\label{lemma:dictionary cc conditions}
    Let $\cO$ be an algebraic pattern. Suppose that $\cP \to\cQ$ is a map of $\cO$-algebrads, equivalently given by a map $X \to Y$ of complete Segal $\Tree[\O]$-spaces. Then the following assertions are equivalent:
    \begin{enumerate}
        \item Condition \cref{conditionCC} of \cref{thm:main-thm-expo-CSeg} holds for $X \to Y$.
        \item For any composable pair of active morphisms $h \colon x \actarrow y$ and $g \colon y \actarrow e$ with $e$ elementary in $\cQ$, and any lift $f \colon \bar{x} \to \bar{e}$ of $g \circ h$, the category 
    \[
    \mathrm{Fact}(f \mid g\circ h) \coloneqq \{x \overset{h}{\actarrow} y \overset{g}{\actarrow} e\} \times_{(\Q_{/{e}})_{{gh}/}} (\P_{/\bar{e}})_{f/}.
    \]
    is weakly contractible.
    \end{enumerate}
\end{lemma}
\begin{proof}
    Let a map $[2;t] \to Y$ from a tree be given. Viewing $t$ as a map $[2] \to \cO^\activ$, the map $[2;t] \to Y$ corresponds to a map $[2] \to \cQ^\activ$ over $\cO^\activ$; in other words, a pair of active morphisms $h \colon x \actarrow y$ and $g \colon y \actarrow e$ in $\cQ$ over $t$. (Note that $e$ is elementary by definition.)
It follows from \cref{rem:boxproduct-algebrad-side} that we may identify the map
\[\colim_{[k] \in \Delta^\op} \Hom_{/Y}([1 + k + 1] \boxtimes_{[2]} [2;t],X) \to \Hom_{/Y}([1] \boxtimes_{[2]} [2;t],X)\]
with the map
\[\colim_{[k] \in \Delta^\op} \Hom_{/\cQ}([1 + k + 1], \cP) \to \Hom_{/\cQ}([1],\cP). \]
An argument similar to \cite[Lemma 2.2.8]{AyalaFrancis2020FibrationsInftyCategories}, using universality of colimits in spaces, shows that the fiber of this map over $f \colon \bar{x} \actarrow \bar{e}$ is given by the geometric realization of the category
$\mathrm{Fact}(f \mid g\circ h)$.
We conclude that $\cP \to \cQ$ satisfies condition (2) of this lemma if and only if $X \to Y$ satisfies the condition from \cref{thm:main-thm-expo-CSeg}.
\end{proof}

We obtain the following as an immediate corollary of \cref{thm:main-thm-expo-CSeg} and \cref{lemma:dictionary cc conditions}:

\begin{theorem}\label{thm:main-thm-expo-factorization-Algd}
    Let $\cO$ be an algebraic pattern and $\cP \to \cQ$ a map of $\cO$-algebrads.
    Then $\cP \to \cQ$ is exponentiable if for any composable pair of active morphisms $h \colon x \actarrow y$ and $g \colon y \actarrow e$ with $e$ elementary in $\cQ$, and any lift $f \colon \bar{x} \to \bar{e}$ of $g \circ h$, the category 
    \[
    \mathrm{Fact}(f \mid g\circ h) \coloneqq \{x \overset{h}{\actarrow} y \overset{g}{\actarrow} e\} \times_{(\Q_{/{e}})_{{gh}/}} (\P_{/\bar{e}})_{f/}.
    \]
    is weakly contractible.\qed
\end{theorem}

\begin{remark}
    In the case that $\O = \ast$ is the terminal pattern, this recovers the Conduch\'e criterion for exponentiable functors between categories that was exhibited by Lurie \cite[Proposition B.3.14]{HA} and Ayala--Francis--Rozen\-blyum \cite[Lemma 5.16(2)]{AyalaFrancisRozenblyum}.
\end{remark}

\begin{remark}
    The condition from \cref{thm:main-thm-expo-factorization-Algd} can be rephrased as follows: for every functor $t \colon [2] \to \Q^\act$ such that $t(2)$ is elementary, the base-change 
    $
    [2] \times_\Q \P \to [2]
    $
    is exponentiable in $\Cat$.
\end{remark}

\begin{remark}\label{remark:expo-over-all-actives}
    It follows that whenever $\cP^\activ \to \cQ^\activ$ is exponentiable in $\Cat$, then $\cP \to \cQ$ is exponentiable in $\Algd(\cO)$.
    However, the condition from \cref{thm:main-thm-expo-factorization-Algd} is generally weaker than the condition that $\cP^\activ \to \cQ^\activ$ is exponentiable.
    An explicit counterexample is described in \cref{subsection:non-triv exp map between vdc}.
\end{remark}

\subsection{A first reduction step}
We start the proof of \cref{thm:main-thm-expo-CSeg} by making a few reduction steps.

\begin{lemma}
    \label{lemma:all is fine with completeness extension}
    Suppose that $[0;e]$ is a root. If 
    \[
        \begin{tikzcd}
            X' \arrow[r, "v"]\arrow[d] & X \arrow[d] \\
            {[J;e]} \arrow[r] & {[0;e]}.
        \end{tikzcd}
    \]
    is a pullback square in $\PSh(\Tree[\cO])$, then $v$ is a complete Segal extension.
\end{lemma}
\begin{proof} We have to show that the map $v : J \boxtimes_{[0]} X \to X$ is a complete Segal extension for every $X \to [0;e]$. Since both sides are cocontinuous in $X \in \PSh(\Tree[\O])_{/[0;e]}$, we may assume that $X = [n;t]$. In light of \cref{proposition:explicit description of mapping space} and \cref{lemma:box with forest is forest}, the map $X \to [0;e]$ is then of the form $[n] \boxtimes_{[0]} [0;e'] \to [0;e]$, with $e'$ an elementary. By pasting pullback squares, $v$ may now be identified with the map $$(J \times [n]) \boxtimes_{[0]} [0;e'] \to [n] \boxtimes_{[0]} [0;e']$$
induced by the projection $J \times [n] \to [n]$. But the latter map is a complete Segal extension, so that the result follows from \cref{prop:simplicial segal extension}.
\end{proof}

\begin{notation}
    There is a canonical inclusion 
    $\Gamma[n] \coloneqq [n-1] \cup_{[0]} [1] \to [n]$
    of simplicial spaces for every $n\geq 0$. Via \cref{cons:loose braket}, this induces an inclusion 
    $$
    \Gamma[n;t] \coloneqq \Gamma[n] \boxtimes_{[n]} [n;t] \to [n;t]
    $$
    for every tree $[n;t]$.
\end{notation}

\begin{remark}\label{rem:on gamma}
    Via an easy induction argument, it is readily verified that the Segal extensions are generated by the inclusions $\Gamma[n;t] \to [n;t]$ where $[n;t]$ ranges over all trees.
\end{remark}

\begin{lemma}\label{lemma:conduche reduction}
    Let $f \colon X \to Y$ be a map between complete Segal $\Tree[\O]$-spaces. Then $f$ is exponentiable if for every commutative diagram
    \[
        \begin{tikzcd}
            X'' \arrow[r, "v"]\arrow[d] & X' \arrow[r]\arrow[d] & X\arrow[d] \\
            \Gamma[n;t] \arrow[r] & {[n;t]}\arrow[r] & Y
        \end{tikzcd}
    \]
    of pullback squares, with $[n;t]$ a tree,  the map $v$ is a complete Segal extension.
\end{lemma}
\begin{proof}
    As presheaf categories are locally cartesian closed, there is an adjunction $$f^* : \PSh(\Tree[\O])_{/Y} \rightleftarrows \PSh(\Tree[\O])_{/X} : f_*.$$  The base change functor $\CSeg(\Tree[\O])_{/Y} \to \CSeg(\Tree[\O])_{/X}$ induced by $f$ is restricted from this left adjoint, and we will also denote it by $f^*$.
    Suppose that $f$ meets the  condition considered in the statement. Then this hypothesis and \cref{rem:on gamma}, combined with \cref{lemma:all is fine with completeness extension}, imply that the functor $f^*$ on presheaf categories preserves all complete Segal extensions over $Y$.
    Thus the functor $f^*\colon\CSeg(\Tree[\O])_{/Y} \to \CSeg(\Tree[\O])_{/X}$ is cocontinuous.
\end{proof}

\subsection{An explicit Segal replacement} In light of \cref{lemma:conduche reduction}, we may reduce to the case that $Y$ is given by a tree $[n;t]$. Our next goal is to exhibit an explicit Segal replacement for the pullback $X \times_{[n;t]} \Gamma[n;t] \to [n;t]$. We will make use of the following terminology:

\begin{definition}
    \label{defi:shape for conduche fibration}
    A map $\phi \colon [m] \to [n]$ is called \textit{$n$-concave} if there exists a (necessarily unique) integer $0 \leq l < m$ such that $\phi(l) < n - 1$ and $\phi(l+1) = n$.  Otherwise, $\phi$ is called \textit{$n$-convex}. We will write $\Lambda_{/[n]} \subset \Delta_{/[n]}$ for the full subcategory spanned by the $n$-convex maps. 
    
    Similarly, we will say that a map $f \colon [m;s]\to [n;t]$ in $\Tree[\O]$ is \textit{$n$-convex} or \textit{$n$-concave} whenever the underlying simplicial map $\phi \colon [m] \to [n]$ is. We will write $\Lambda[\O]_{/[n;t]} \coloneqq \Tree[\O]_{/[n;t]} \times_{\Delta_{/[n]}} \Lambda_{/[n]}$ for the full subcategory of $\Tree[\O]_{/[n;t]}$ spanned by the $n$-convex maps.
\end{definition}

\begin{construction}\label{cons:Q}
The inclusion $j \colon \Lambda[\O]_{/[n;t]} \to \Tree[\O]_{/[n;t]}$ 
induces an adjunction $$j_! : \PSh(\Lambda[\O]_{/[n;t]}) \rightleftarrows \PSh(\Tree[\O])_{/[n;t]} : j^*,$$
where the left adjoint is computed by left Kan extensions.
We define the \textit{replacement} endofunctor $Q$ by setting
$$Q \coloneqq j_!j^* \colon \PSh(\Tree[\O])_{/[n;t]}\to \PSh(\Tree[\O])_{/[n;t]}.$$
The counit gives rise to a natural map 
$$
\alpha_X \colon QX \to X 
$$
for $X \in \PSh(\Tree[\O])_{/[n;t]}$. Since the inclusion $\Tree[\O]_{/\Gamma[n;{t}]} \subset \Tree[\O]_{/[n;t]}$  factors through $\Lambda[\O]_{/[n;t]}$,  the induced map 
   $
   QX \times_{[n;t]} \Gamma[n;t] \to X \times_{[n;t]} \Gamma[n;t]
   $
   is an equivalence for every $X \in \PSh(\Tree[\O])_{/[n;t]}$. All in all, we obtain a natural factorization 
   $$
   X \times_{[n;t]} \Gamma[n;t] \to QX \to X.
   $$
\end{construction}

We will show that $Q$ can be computed using simplicial resolutions of $n$-concave maps by $n$-convex ones.

\begin{construction}\label{con:functor T}
    Let $\Xi_{/[n]} \subset \Delta_{/[n]}$ denote the full subcategory spanned by the maps that do not hit $n-1$.
    Suppose that $k \geq 0$ is an integer. Let $\phi \colon [m]\to [n]$ be a map in $\Xi_{/[n]}$. If $\phi$ is $n$-concave, then there must exist a unique integer $0\leq l < m$ so that 
    $\phi(l) < n-1$ and $\phi(l+1) = n$. We will then consider the $n$-convex map
      $T_{k}\phi \colon [m+k+1]\to [n]$
    described by
     \[
        (T_k\phi)(i) = \begin{cases}
            \phi(i) & \text{if $i \leq l$,}\\
            n-1 & \text{if $l + 1 \leq i \leq l + k + 1$},\\
            n & \textit{if $l+k +2\leq i$}.
        \end{cases}
     \]
    In other words, $T_k\phi$ coincides with $\phi$ on the inclusion $\{0, \dotsc, l, l+k+2, \dots, m+k+1\} \colon [m] \to [m+k+1]$, and it is constant to $n-1$ on the inclusion $\{l+1, \dotsc, l + k + 1\} \colon [k] \to [m+k+1]$.
    If $\phi \colon [m] \to [n]$ is $n$-convex, then we set 
    $$
    T_k\phi \coloneqq \phi.
    $$
    One readily verifies that this definition assembles to a functor
    $$T\colon \Delta \times \Xi_{/[n]}\to \Delta_{/[n]} ; \quad ([k], \phi) \mapsto T_k\phi.$$ We note that there is a canonical map $$\eta_{k,\phi} \colon \phi \to T_k\phi$$
    that is natural in $\phi \in \Xi_{/[n]}$ and $[k] \in \Delta$.
\end{construction}

\begin{proposition}\label{prop:computation Q} 
    Let $[n;s] \to [n;t]$ be a map of trees above $\id_{[n]} \colon [n] \to [n]$. 
    Suppose that $X \in \PSh(\Tree[\O])_{/[n;t]}$ and $\phi \colon [m] \to [n] \in \Xi_{/[n]}$,  then the natural maps in the span 
    \[
        \begin{tikzcd}[column sep = {10em, between origins}, row sep = small]
            & \colim_{[k] \in \Delta^\op}\map_{/[n;t]}(T_k[m]\boxtimes_{[n]}[n;s],QX) \arrow[dl] \arrow[dr] & \\ 
            \map_{/[n;t]}([m]\boxtimes_{[n]}[n;s], QX) && \colim_{[k] \in \Delta^\op}\map_{/[n;t]}(T_k[m]\boxtimes_{[n]}[n;s],X)
        \end{tikzcd}
    \]
    are equivalences. Here, the left leg is induced by $\eta$, and the right leg is induced by $\alpha$.
\end{proposition}
\begin{proof}
    We will first handle the right leg. To this end, it will suffice to verify that the comparison map 
    $$
        \map_{/[n;t]}([m';s'],QX) \to \map_{/[n;t]}([m';s'],X)
    $$
    is an equivalence for every map of \textit{forests} $\left<m';s'\right> \to \left<n;t\right>$ for which the underlying map $[m'] \to [n]$ is $n$-convex; see \cref{lemma:box with forest is forest}. This holds by construction if $\left<m';s'\right>$ is a tree. The general case follows by applying the forest decomposition formula of \cref{prop:forest decomposition}.
    
    We will now handle the left leg.
    If $\phi$ is $n$-convex, then $\eta_{\bullet,\phi}$ is the identity componentwise. Since $\Delta$ is weakly contractible, this then implies that the left leg is an equivalence. 
    
    Suppose now that $\phi$ is $n$-concave. Then the natural transformation
    $$
    \eta'_{\bullet,\phi} \coloneqq \eta_{\bullet,\phi} \boxtimes_{[n]} [n;s] \colon \Delta \times [1] \to \PSh(\Tree[\O])_{/[n;t]}
    $$
    induced by $\eta_{\bullet,\phi}$ factors through the Yoneda embedding $\Tree[\O]_{/[n;t]} \to \PSh(\Tree[\O])_{/[n;t]}$.
    The target of $\eta'_{\bullet,\phi}$ has image in $\Lambda[\O]_{/[n;t]}$. Consequently, it induces a map 
    $$
    T_{\bullet}f \colon \Delta \to (\Lambda[\O]_{/[n;t]})_{f/}
    $$
    where $f$ denotes the map $[m; \phi^*s] = [m] \boxtimes_{[n]}[n;s] \to [n;t]$. One may then identify the left leg of the span from the statement with the map 
    $$
    \colim_{[k] \in \Delta^\op} X(T_k[m] \boxtimes_{[n]} [n;s]) \to \colim_{[m';s'] \to [n;t] \in ((\Lambda[\O]_{/[n;t]})^\op)_{/f}} X([m';s'])
    $$
    induced by restriction along $(T_\bullet f)^\op$. We will conclude the proof by demonstrating that $T_\bullet f$ is initial.
    
    The equivalence of \cref{proposition:explicit description of mapping space} induces an identification
    $$
    (\Lambda[\O]_{/[n;t]})_{f/} \simeq (\Lambda_{/[n]})_{\phi/} \times ((\O^{\el, \op}_n)_{/t})_{\psi/},
    $$
    where $\psi \colon t \to s$ is the morphism in $\O^\el_n$ corresponding to the fixed map $[n;s] \to [n;t]$.
    Under this equivalence, the functor $T_\bullet f$ is given by the functor 
    $$
    (T_\bullet \phi, \{\psi\}) \colon \Delta \to (\Lambda_{/[n]})_{\phi/} \times ((\O^{\el, \op}_n)_{/t})_{\psi/} ; \quad [k] \mapsto ((\phi \to T_k\phi), (\psi = \psi)).
    $$
    This functor is initial since it admits a right adjoint.
\end{proof}

\begin{proposition}\label{prop:Q is a replacement}
    The map $X\times_{[n;t]}\Gamma[n;{t}]\to QX$ is a Segal extension for all maps $p \colon X \to [n;t]$, and $QX \to [n;t]$ is a (complete) Segal fibration if $p$ is a (complete) Segal fibration.
\end{proposition}

\begin{proof}
    We will first show that the comparison map $X \times_{[n;t]} \Gamma[n;t] \to QX$ is a  Segal extension for all $p \colon X \to [n;t]$.
    As the functors that appear on both sides preserve colimits in $f \colon X \to [n;t]$, it suffices to handle the case that $X$ is a tree $[m;s]$. 
    Suppose that the underlying map $\phi\colon[m]\to[n]$ of $f$ is $n$-convex, then 
    the comparison map can be identified with the map
    $$([m]\times_{[n]}\Gamma[n])\boxtimes_{[m]}[m;s]\to  [m;s]$$
    induced by the inclusion $[m] \times_{[n]} \Gamma[n] \to [m]$.
    If $\phi$ skips $n-1$, then this inclusion is the identity. Thus we may assume that $\phi$ hits $n-1$. If $\phi$ skips $n$, then the inclusion map is the identity as well. Otherwise, both $n-1$ and $n$ are contained in the image of $\phi$, and one readily checks  that $[m] \times_{[n]} \Gamma[n] \to [m]$ is a Segal extension so that the desired result follows from \cref{prop:simplicial segal extension}.
    
   If $\phi$ is $n$-concave, then we proceed as follows. 
    Let $l$ be the integer such that $\phi(l) < n-1$ and $\phi(l+1) = n$. Then one readily verifies that the map 
    $[m] \times_{[n]} \Gamma[n] \to [m]$ is given by the canonical inclusion 
    $V \coloneqq [l] \sqcup [m-l-1] \to [m]$. The comparison map is then given by a map
    $$
    V \boxtimes_{[m]} [m;s] \to Q[m;s],
    $$
    and we claim that this is an equivalence.
    As both sides are contained in the image of $j_! \colon \PSh(\Lambda[\O]_{/[n;t]}) \to \PSh(\Tree[\O])_{/[n;t]}$, we have to check that the map 
    \begin{equation}\label{eq: Q is a replacement reduction}
     \Hom_{/[n;t]}([m';s'], V\boxtimes_{[m]} [m;s]) \to \Hom_{/[n;t]}([m';s'], Q[m;s]) \simeq \Hom_{/[n;t]}([m';s'], [m;s])
    \end{equation}
    is an equivalence for every $n$-convex map $[m';s'] \to [n;t]$. One may readily verify that the map displayed in \cref{eq: Q is a replacement reduction} comes from the canonical map $V \boxtimes_{[m]} [m;s] \to [m;s]$, and thus fits in the pullback square
    \[
        \begin{tikzcd}
            \Hom_{/[n;t]}([m';s'], V \boxtimes_{[m]} [m;s]) \arrow[r]\arrow[d] & \Hom_{/[n;t]}([m';s'], [m;s]) \arrow[d] \\
            \Hom_{/[n]}([m'], V) \arrow[r] & \Hom_{/[n]}([m'], [m]).
        \end{tikzcd}
    \]
    The bottom arrow is an equivalence as we assumed that the underlying map $[m'] \to [n]$ is $n$-convex.
    
    For the final assertion, we must demonstrate that $QX \to [n;t]$ is local with respect to the maps $i \colon \Sp[m;s'] \to [m;s']$ over $[n;t]$. Since $X \to [n;t]$ is a Segal fibration, \cref{cons:Q} implies that it will be sufficient to check the case where the underlying map $\phi \colon [m] \to [n]$ is $n$-concave. In particular, $\phi$ hits $n$, so that by \cref{cor:explicit description of mapping space}, the map $[m;s'] \to [n;t]$ factors as $[m;s'] \to [n;s] \to [n;t]$ where the first arrow is cartesian over $\phi$, and the second arrow lies over the identity on $[n]$. Then we must show that the restriction
    $$
    i^* \colon \map_{/[n;t]}([m] \boxtimes_{[n]} [n;s], QX)\to\map_{/[n;t]}(\Sp[m] \boxtimes_{[n]} [n;s], QX)
    $$
    is an equivalence. 
    
    As $\phi$ is $n$-concave, there exists a unique integer $l$ such that $\phi(l)<n-1$ and $\phi(l+1)=n$. The restriction $i^*$ is then given by the gap map (cf.\ \cref{conventions}) in the canonical commutative square
    \[
        \begin{tikzcd}
            \map_{/[n;t]}([m] \boxtimes_{[n]} [n;s], QX)\arrow[d] \arrow[r] & \map_{/[n;t]}([1] \boxtimes_{[n]} [n;s], QX)\arrow[d] \\
            \Hom_{/[n;t]}((\Sp[l] \sqcup \Sp[m - l -1]) \boxtimes_{[n]} [n;s], QX)\arrow[r] & \Hom_{/[n;t]}(([0] \sqcup [0]) \boxtimes_{[n]} [n;s], QX),
        \end{tikzcd}
    \]
    where the top map is induced by the inclusion $\{l, l+1\} \colon [1] \to [m]$. 
    In turn, we may identify this square with the commutative square
    \[
        \begin{tikzcd}
            \colim_{[k]\in\Delta^{\op}}\map_{/[n;t]}(T_k[m] \boxtimes_{[n]} [n;s], X)\arrow[d] \arrow[r] & \colim_{[k]\in\Delta^{\op}}\map_{/[n;t]}(T_k[1] \boxtimes_{[n]} [n;s], X)\arrow[d] \arrow[d] \\
            \Hom_{/[n;t]}((\Sp[l] \sqcup \Sp[m - l -1]) \boxtimes_{[n]} [n;s], X)\arrow[r] & \Hom_{/[n;t]}(([0] \sqcup [0]) \boxtimes_{[n]} [n;s], X)
        \end{tikzcd}
    \]
    using the computation of  \cref{prop:computation Q}.
    By \cref{prop:simplicial segal extension}, it will be sufficient to verify that the map of simplicial sets
     $$\textstyle\Sp[l]\bigcup_{[0]}[1+k+1]\bigcup_{[0]}\Sp[m-l-1]\to [m+k+1]$$
     is a Segal extension. But this is clear.
     
    Finally, we need to show that $QX \to [n;t]$ is local with respect to $[J;e] \to [0;e]$ for any elementary $e$ if $X \to [n;t]$ is.
    This follows since any map $[0;e] \to [n;t]$ lies in $\Lambda[\cO]_{/[n;t]}$.
\end{proof}

\subsection{The conclusion of the proof} We will be ready to finish the proof of \cref{thm:main-thm-expo-CSeg} after the following last observation:

\begin{lemma}\label{lemma:exponentiable Segal fibrations}
    Let $p \colon X\to Y$ be a (complete) Segal fibration between $\Tree[\O]$-spaces. Then the following are equivalent:
    \begin{enumerate}[(1)]
    \item For every diagram of pullback squares
\[\begin{tikzcd}
	{X''} & {X'} & X \arrow[d,"p"]\\
	{\Gamma[n;{t}]} & {[n;t]} & Y,
	\arrow["v", from=1-1, to=1-2]
	\arrow[from=1-1, to=2-1]
	\arrow[from=1-2, to=1-3]
	\arrow[from=1-2, to=2-2]
	\arrow[from=1-3, to=2-3]
	\arrow[from=2-1, to=2-2]
	\arrow[from=2-2, to=2-3]
\end{tikzcd}\]
    $v$ is a (complete) Segal extension.
    \item For every pullback square 
    \[\begin{tikzcd}
	 {X'} & X \arrow[d,"p"]\\
	 {[n;t]} & Y,
	\arrow[from=1-1, to=1-2]
	\arrow[from=1-1, to=2-1]
	\arrow[from=1-2, to=2-2]
	\arrow[from=2-1, to=2-2]
\end{tikzcd}\]
    the canonical map $QX'\to X'$ is an equivalence.

    \item For every map from a tree $[2;t] \to Y$, the induced map
    $$
    \colim_{[k] \in \Delta^\op}\map_{/Y}(T_k[1] \boxtimes_{[2]} [2;t], X) \to \map_{/Y}([1] \boxtimes_{[2]} [2;t], X)
    $$
    is an equivalence.

\end{enumerate}
\end{lemma}
\begin{proof}
    In the situation of (1), we have a factorization 
    $
    X'' \xrightarrow{v'} QX' \xrightarrow{v''} X'
    $
    of $v$. 
    On account of \cref{prop:Q is a replacement}, $v'$ is a Segal extension and $QX'\to[n;t]$ is a (complete) Segal fibration. By left cancellation, this implies that $v''$ is a (complete) Segal fibration as well.  
    Thus $v$ is a (complete) Segal extension if and only if $v''$ is an equivalence. This shows that (1) and (2) are equivalent.
    
    Suppose that (2) holds. 
    Let $f \colon [2;t] \to Y$ be a map. Then for $X' \coloneqq X \times_Y [2;t]$, the map $QX' \to X'$ is an equivalence. Thus it follows from \cref{prop:computation Q} that the map 
    $$
    \colim_{[k] \in \Delta^\op}\map_{/[2;t]}(T_k[1] \boxtimes_{[2]} [2;t], X') \to \map_{/[2;t]}([1] \boxtimes_{[2]} [2;t], X')
    $$
    is an equivalence. But this can be identified with the map from (3).
    
    For the final step, we suppose that (3) holds. By reasoning as in the proof of \cref{prop:Q is a replacement}, assertion (2) holds if and only if the map
$$\map_{/[n;t]}([m]\boxtimes_{[n]}[n;s],QX')\to \map_{/[n;t]}([m]\boxtimes_{[n]} [n;s],X')$$
is an equivalence for every $n$-concave map $\phi \colon [m] \to [n]$ and map $s\to t \in \O^\el_n$ giving rise to a map $[n;s] \to [n;t]$ of trees.
In light of \cref{prop:computation Q}, this is computed by the map 
$$\colim_{[k]\in\Delta^\op}\map_{/[n;t]}(T_k[m]\boxtimes_{[n]}[n;s],X')\to \map_{/[n;t]}([m] \boxtimes_{[n]} [n;s],X').$$
Let $0 \leq l < m$ be the integer so that $\phi(l) < n-1$ and $\phi(l) = n$. This determines a map $\{l,l+1\} \colon [1] \to [m]$. We have the following natural commutative square 
\[
    \begin{tikzcd}
        {[l] \cup_{[0]} [1] \cup_{[0]} [m-l-1]}\arrow[r]\arrow[d] & {[l] \cup_{[0]} T_k[1] \cup_{[0]} [m-l-1]} \arrow[d] \\
        {[m]} \arrow[r] & {T_k[m]}
    \end{tikzcd}
\]
of simplicial spaces over $[n]$, so that the vertical arrows are Segal extensions. Consequently, \cref{prop:simplicial segal extension} 
implies that it suffices to check that the map
\begin{equation}\label{eq:lemma exp Segal reduction}
\colim_{[k]\in\Delta^\op}\map_{/[n;t]}(T_k[1] \boxtimes_{[n]}[n;s],X')\to \map_{/[n;t]}([1] \boxtimes_{[n]}[n;s],X')
\end{equation}
is an equivalence instead. 
There is a commutative square 
\[
    \begin{tikzcd}
        {[1]} \arrow[r,"d_1"] \arrow[d,"{\{l, l+1\}}"'] & {[2]} \arrow[d, "j"] \\ 
        {[m]} \arrow[r,"\phi"] & {[n]},
    \end{tikzcd}
\]
where $j$ is the map that hits $\phi(l)$, $n-1$ and $n$. The map \cref{eq:lemma exp Segal reduction} can then be identified with the map 
$$\colim_{[k]\in\Delta^\op}\map_{/Y}(T_k[1] \boxtimes_{[2]} [2;j^*s],X)\to \map_{/Y}([1] \boxtimes_{[2]} [2;j^*s],X).$$
This is an equivalence by the assumption of (3).
\end{proof}

\begin{proof}[Proof of \cref{thm:main-thm-expo-CSeg}]
    This now follows from combining \cref{lemma:all is fine with completeness extension}, \cref{lemma:conduche reduction}, and \cref{lemma:exponentiable Segal fibrations}.
\end{proof}

\begin{remark}\label{rem:thmB easy for atomic}
    Suppose that $\O$ has the property that all its trees $[n;t]$ are complete Segal $\Tree[\O]$-spaces. Then the converse of \cref{thm:main-thm-expo-CSeg} is easily shown.
     Namely, suppose that $f : X \to Y$ is exponentiable in $\CSeg(\Tree[\O])$. Let $[n;t] \to Y$ be a map from a tree.
     By the assumption on $\O$, we have a colimit expression $$\colim_{[m;s] \in \Tree[\O]_{/\Gamma[n;t]}} [m;s] \xrightarrow{\simeq} [n;t]$$ \emph{in the category} $\CSeg(\Tree[\O])_{/Y}$. In particular, the left adjoint $L \colon \PSh(\Tree[\cO])_{/Y} \to \CSeg(\Tree[\cO])_{/Y}$ takes the map $X \times_Y \Gamma[n;t] \simeq \colim_{[m;s] \in \Tree[\cO]_{/\Gamma[n;t]}} X \times_Y [m;s] \to X \times_Y [n;t]$ to an equivalence.
     By \cref{remark:saturated-vs-strongly-saturated}, this map is a complete Segal extension. One may then invoke \cref{lemma:exponentiable Segal fibrations}.
     
    Unfortunately, the condition on the trees $[n;t]$ to be complete Segal is quite restrictive. For instance, this fails if $\O$ is $\F_*^\flat$ or $\Span(\F_G)^\flat$. We will remedy this situation in \cref{section:robust patterns}, where we introduce so-called \textit{robust} algebraic patterns for which we can show the converse of \cref{thm:main-thm-expo-CSeg} in a more elaborate way. There is a special subclass of robust patterns, called the \textit{atomically} robust patterns, for which the trees $[n;t]$ are always complete Segal (see \cref{cor:fibrant forest}).
\end{remark}

\subsection{Examples}\label{subsection:examples of CC}
We will now spell out the conditions of \cref{thmA} for a range of examples of algebraic patterns $\O$.

\begin{example}\label{example:CC-criterion-VDC}
    Let $\cO = \Delta^{\op,\natural}$ be the pattern from \cref{examples:examples of patterns} describing virtual double categories.
    By \cref{thm:main-thm-expo-factorization-Algd}, it follows that a map $\cP \to \cQ$ in $\VirtDblCat = \Algd(\Delta^{\op,\natural})$ is exponentiable if
    for every functor $t \colon [2] \to \cQ^\activ$ such that $t(2)$ lies over $[1]$ or $[0]$, the base-change $\cP^\activ \times_{\cQ^\activ} [2] \to [2]$ is exponentiable in $\Cat$.
    Note that if $t(2) = [0]$, then all of $t$ must necessarily land in the fiber $\cQ_0$ over $[0]$.
    In particular, this condition can be rephrased as follows: $\cP \to \cQ$ is exponentiable if
    \begin{enumerate}[(1)]
        \item the functor $\cP_0 \to \cQ_0$ in $\Cat$ is exponentiable, and
        \item for any $t \colon [2] \to \cQ^\activ$ such that $t(2)$ lies over $[1]$, the base-change $\cP^\activ \times_{\cQ^\activ} [2] \to [2]$ is exponentiable in $\Cat$.
    \end{enumerate}
    In \cref{exa:CC-necessary-VDC}, we will see that the converse also holds: if $\cP \to \cQ$ is exponentiable, then it satisfies these two conditions.
    Note that these conditions are strictly weaker than $\cP^\activ \to \cQ^\activ$ being exponentiable: in \cref{subsection:non-triv exp map between vdc} we give an explicit example of an exponentiable morphism $\cP \to \cQ$ in $\VirtDblCat$ for which $\cP^\activ \to \cQ^\activ$ is not exponentiable.
\end{example}

\begin{example}\label{example:CC-criterion-gen-operads}
    Recall from \cref{examples:examples of patterns} that generalized operads are algebrads for the pattern $\F_*^\natural$. It follows as in \cref{example:CC-criterion-VDC} that a map $\cP \to \cQ$ of generalized operads is exponentiable in $\Algd(\F_*^\natural)$ if
    \begin{enumerate}[(1)]
        \item the functor $\cP_{\langle 0 \rangle} \to \cQ_{\langle 0 \rangle}$ is exponentiable in $\Cat$, and
        \item for any $t \colon [2] \to \cQ^\activ$ such that $t(2)$ lies over $\langle 1 \rangle$, the base-change $\cP^\activ \times_{\cQ^\activ} [2] \to [2]$ is exponentiable in $\Cat$.
    \end{enumerate}
    In \cref{exa:CC-necessary-genoperad}, we will see that the converse also holds.
\end{example}

In contrast to the case of virtual double categories (cf.\ \cref{rem:VDC-expo-vs-actives-conduche}), the following lemma shows that for various flavors of operads, the conditions from \cref{thm:main-thm-expo-CSeg} and \cref{thm:main-thm-expo-factorization-Algd} are equivalent to $\cP^\activ \to \cQ^\activ$ being exponentiable in $\Cat$.

\begin{proposition}\label{prop:main-thm-trees-vs-forests}
    Suppose that $\O$ is sound and that the elementary slices $\O^\el_{x/}$  accept an initial functor from a finite set for every $x\in \O$. If $f \colon X \to Y$ in $\CSeg(\Tree[\O])$ is a map that meets the condition \cref{conditionCC} of \cref{thm:main-thm-expo-CSeg}, then the comparison map
    \[\colim_{[k] \in \Delta^\op} \map_{/Y}( T_k[1]\boxtimes_{[2]}[2;t] ,X) \to \map_{/Y}([1;t_0 \actarrow t_2] ,X).\]
    is an equivalence for any forest $\langle 2;t \rangle$ of length 2 and any map $[2;t] \to Y$.
\end{proposition}
\begin{proof}
    Using the forest decomposition formula of \cref{rem:forest decomposition sound} for sound patterns, one shows that the comparison map for a forest $\left<2;t\right>$ is computed by the map
    \[
   \colim_{\Delta^\op} \lim_{s\in(\O^\el_2)_{t/}}\Hom_{/Y}(T_k[1]\boxtimes_{[2]}[2;s],X) \rightarrow \lim_{s\in(\O^\el_2)_{t/}}\Hom_{/Y}([1]\boxtimes_{[2]}[2;s],X).
    \]
   We may interchange the limit and colimit by our assumption on the elementary slices of $\O$ and the fact that $\Delta^\op$ is sifted. Then we can use the assumption that $f$ meets the conditions of \cref{thm:main-thm-expo-CSeg}.
\end{proof}

The following can now be easily deduced by arguing as in the proof of \cref{lemma:dictionary cc conditions}.

\begin{corollary}\label{cor:main-thm-trees-vs-forests}
    Suppose that $\O$ is sound and that the elementary slices $\O^\el_{x/}$  accept an initial functor from a finite set for every $x\in \O$. Let $f \colon \cP \to \cQ$ be a map of $\O$-algebrads in the sense of \cref{defi:of algebrad for factactization system}. Then the condition of \cref{thm:main-thm-expo-factorization-Algd} for $f$ is equivalent to the condition that $f^\act \colon \cP^\act \to \cQ^\act$ is exponentiable in $\Cat$.\qed
\end{corollary}

\begin{example}\label{ex:CC-criterion-operads}
    Let $\cO = \F_*^\flat$ be the pattern describing operads.
    Suppose that $f \colon \cP \to \cQ$ is a map between operads.
    It follows from \cref{cor:main-thm-trees-vs-forests} and \cref{thm:main-thm-expo-factorization-Algd} that $f$ is exponentiable in the category $\mathrm{Op} \simeq \Algd(\F_*^\flat)$ of operads if the following two equivalent conditions hold:
    \begin{enumerate}
        \item for every functor $t \colon [2] \to \Q^{\act}$ so that $t(2)$ lies over $\left<1\right>$, the base-change $\P^\act \times_{\Q^\act} [2] \to [2]$ is exponentiable in $\Cat$, and,
        \item the underlying functor $\P^\act\to \Q^\act$ is exponentiable in $\Cat$.
    \end{enumerate}
    (The equivalence of these conditions was also proved by Hinich \cite[Lemma 2.8.2]{Hinich2020YonedaLemmaEnriched}.)
    In \cref{exa:CC-necessary-operad}, we will conversely see that any exponentiable map of operads satisfies these two conditions.
    By \cref{example:CC-criterion-gen-operads}, it follows that such a map of operads is also exponentiable in the bigger category $\Algd(\F_*^\natural)$ of generalized operads
\end{example}

\begin{example}\label{example:CC-criterion-G-operads}
    Let $G$ be a finite group and consider the pattern $\Span(\F_G)$ from \cref{examples:examples of patterns} describing $G$-operads.
    By \cite[Example 3.3.26]{BarkanHaugsengSteinebrunner}, the conditions of \cref{cor:main-thm-trees-vs-forests} are satisfied, so we see that a map $f \colon \cP \to \cQ$ is exponentiable in the category $\Algd(\Span(\F_G))$ of $G$-operads if the following two equivalent conditions hold:
    \begin{enumerate}
        \item for every functor $t \colon [2] \to \cP^{\act}$ so that $t(2)$ lies over a $G$-orbit (i.e.\ a transitive $G$-set), the base-change $\cP^\act \times_{\cQ^\act} [2] \to [2]$ is exponentiable in $\Cat$, and,
        \item the underlying functor $\cP^\act \to \cQ^\act$ is exponentiable in $\Cat$.
    \end{enumerate}
    
    In \cite[Definition 2.1.7]{NardinShah}, $G$-operads were defined as algebrads for a different algebraic pattern $\underline{\F}_{G,*}$ (see also \cite[Observation 5.2.12]{BarkanHaugsengSteinebrunner}).
    By \cite[Observation 5.2.9]{BarkanHaugsengSteinebrunner}, the pattern $\underline{\F}_{G,*}$ satisfies the conditions of \cref{cor:main-thm-trees-vs-forests}, so we see that a map $\cP \to \cQ$ in $\Algd(\underline{\F}_{G,*})$ is exponentiable if $\cP^\activ \to \cQ^\activ$ is exponentiable in $\Cat$.
    This recovers Corollary 3.1.5 of \cite{NardinShah} for $G$-operads.
    We will see in \cref{exa:CC-necessary-G-operad} that the converse also holds.
\end{example}

\section{Underlying graphs of algebrads}\label{sec:Underlying-graph}

Let $f \colon \O \to \P$ be a strong Segal map between algebraic patterns. As explained in \cref{subsection:functoriality of algebrad cats}, there is an induced adjunction 
$$
f_!: \Algd(\O) \rightleftarrows \Algd(\P) : f^*.
$$
The following is an easy observation:

\begin{proposition}
    The right adjoint functor $f^*$ preserves the maps in $\Algd(\P)$ that satisfy condition \cref{conditionCC} from \cref{thm:main-thm-expo-CSeg}. \qed
\end{proposition}

Although $f^*$ might preserve the exponentiable objects that satisfy \cref{CondCrit}, it is not true that $f^*$ commutes with taking \textit{exponential} objects in general. That is, the canonical comparison map $f^*[X,Y] \to [f^*X,f^*Y]$ is not invertible in general for $\P$-algebrads $X$ and $Y$, where $X$ is exponentiable (see \cref{conventions} for the notation). We will now discuss an important example of a morphism $f$ where this \textit{does} hold.

\begin{construction}
    There is a canonical map of algebraic patterns $\O^{\el} \to \O$. Here, $\O^\el$ is considered as an algebraic pattern where all morphisms are inert and all objects are elementary. We note that $\Algd(\O^{\el}) \simeq \Fun(\O^{\el}, \Cat)$. Consequently, \cref{rem:base change decorated Segal spaces} gives rise to an induced adjunction
    $$
    \iota : \Fun(\O^{\el}, \Cat) \rightleftarrows  \Algd(\O) : \Gamma.
    $$
    If $X$ is an algebrad on $\O$, then $\Gamma X$ is called the \textit{underlying graph} of $X$. 
\end{construction}

\begin{definition}
    A tree $[n;t]$ in $\Tree[\O]$ is called \textit{unary} if there exists a cartesian morphism $[n;t] \to [0;t_0]$; i.e.\ if $t_0 \actarrow \cdots \actarrow t_n$ is a sequence of equivalences.
\end{definition}

We will establish the following description of $\iota$:

\begin{proposition}\label{prop:unary algebrads}
    The functor $\iota$ is fully faithful, and an algebrad $X$ is in the essential image of $\iota$ if and only if $X([n;t]) = \emptyset$ for every non-unary tree $[n;t]$.
\end{proposition}

To this end, we will need some preparations. We will view $\O$ as a double category in what follows.

\begin{remark}
    The final object of $\Delta$ is given by $[0]$, hence there is a unique functor $\Delta \times \O_0^\op \to \Forest[\O]$ over $\Delta$ that preserves cartesian edges, and that induces an equivalence on the fiber above $[0]$. Its restriction $$i \colon \Delta \times \O_0^{\el,\op} \to \Tree[\O] ; \quad ([n],e) \mapsto [n;\overline{e}]$$ is precisely the functor that is induced by the map $\O^{\el}_0 \to \O$ between algebraic patterns. By definition, the functor $\iota$ is derived from the left Kan extension functor
    $$i_! \colon \PSh(\Delta \times \O_0^{\el,\op}) \to \PSh(\Tree[\O]).$$
    We note that $i_!$ carries $([n], e)$ to the tree $[n; \bar{e}] = [n] \boxtimes_{[0]} [0;e]$.
\end{remark}

\begin{lemma}\label{lem: about graphs i ff}
    The functor $i\colon \Delta \times \O_0^{\el,\op} \to \Tree[\O]$ is fully faithful.
\end{lemma}
\begin{proof}
    It suffices to check that the functor $\Delta \times \O_0^\op \to \Forest[\O]$ is fully faithful. This may be checked fiberwise, so we have to verify that the functor $\O_0 \to \O_n$ is fully faithful for every $n$. Using the Segal condition, we may reduce to the case that $n=1$. As a double category, $\O$ is in the image of the $\Sq_{L,R}$-construction. Thus the fully faithfulness in level $n=1$ ultimately follows from the fact that the diagonal functor $\C \to \Fun([1], \C)$ is fully faithful for every category $\C$, since $[1] \times [1] \cup_{\{0,1\}\times [1]} \{0,1\} = [1]$.
\end{proof}

\begin{lemma}\label{lem: about value on non-unaries}
    Let $X$ be a presheaf on $\Delta \times \O^{\el,\op}_0$. Then $(i_!X)([n;t]) = \emptyset$ for every non-unary tree $[n;t]$.
\end{lemma}
\begin{proof}
     The commutative square
     \[
        \begin{tikzcd}
            \emptyset \arrow[r] \arrow[d] & \Tree[\O]_{[n;t]/}\arrow[d] \\
            \Delta \times \O^{\el, \op} \arrow[r] & \Tree[\O]
        \end{tikzcd}
     \]
     is a pullback square. Namely, if $[n;t] \to [m;\overline{e}]$ were a map, then we could consider the composite $[n;t] \to [m;\overline{e}] \to [0;e]$. But then the description of \cref{proposition:explicit description of mapping space} would imply that $[n;t]$ was unary.
\end{proof}

\begin{proof}[Proof of \cref{prop:unary algebrads}]
    Recall that $\iota$ is derived from $i_!$, but in fact, we claim that $\iota$ is also restricted from $i_!$. Since $i_!$ is fully faithful by \cref{lem: about graphs i ff}, this will then prove that $\iota$ is fully faithful. The claim follows by observing that $i_!$ carries each complete Segal $(\Delta \times \O^{\el,\op})$-space $X$ to a complete Segal $\Tree[\O]$-space. Namely, it is clear that $i_!X$ is local with respect to the completeness extensions by the fully faithfulness of $i$.
    To show that $i_!X$ is local with respect to the Segal extensions, we note that $\Map_{\PSh(\Tree[\O])}(\Sp[n;t], i_!X) = \Map_{\PSh(\Tree[\O])}([n;t], i_!X) = \emptyset$ if $[n;t]$ is non-unary by \cref{lem: about value on non-unaries}. If $[n;t]$ is unary, then $i_!X$ is local with respect to its associated spine inclusion by the fully faithfulness of $i_!$. 
    
    For the second assertion, assume that $X$ is an $\O$-algebrad. Then it follows from the above and \cref{lem: about value on non-unaries} that $X([n;t]) = \emptyset$ for all non-unary trees $[n;t]$ if $X$ is in the image of $\iota$. Conversely, suppose that $X$ satisfies $X([n;t]) = \emptyset$ for every non-unary tree $[n;t]$. Then one readily deduces that the counit $\iota\Gamma X \to X$ must be an equivalence.
\end{proof}

It follows directly from \cref{prop:unary algebrads} that the unary trees are the only trees that are in the essential image of $\iota$. Therefore, we may introduce the following extension of the terminology:

\begin{definition}
The algebrads in the essential image of $\iota$ are called \textit{unary algebrads}.
\end{definition}

\begin{corollary}\label{cor:unary only accepts unary}
    If $X$ is a unary algebrad, then $\Map_{\Algd(\O)}(Y,X) = \emptyset$ for every non-unary algebrad $Y$.
\end{corollary}
\begin{proof}
    Since $Y$ is non-unary, \cref{prop:unary algebrads} implies that there exists a map $[n;t] \to Y$ where $[n;t]$ is non-unary. We then obtain a restriction map $\Map_{\Algd(\O)}(Y,X) \to X([n;t])$, so that \cref{prop:unary algebrads} implies the desired result.
\end{proof}

\begin{corollary}\label{cor:unaries form ideal}
    If $X$ is a unary algebrad, then $X \times Y$ is unary as well for every algebrad $Y$.
\end{corollary}
\begin{proof}
    If $[n;t]$ is a non-unary tree, then $(X \times Y)([n;t]) = X([n;t]) \times Y([n;t]) = \emptyset$ by \cref{prop:unary algebrads}. Thus the desired conclusion follows from another application of \cref{prop:unary algebrads}.
\end{proof}

\begin{remark}\label{rem:trivial algebrad}
    The functor $\iota \colon \Fun(\O^\el, \Cat) \to \Algd(\O)$ factors as the composite
    $$
    \Fun(\O^\el, \Cat) \to \Algd(\O)_{/\iota(\ast)} \to \Algd(\O),
    $$
    where $\ast \colon \O^\el \to \Cat$ denotes the terminal functor. The second functor is the projection. The resulting first functor is fully faithful on account of \cref{prop:unary algebrads}. Moreover, \cref{cor:unary only accepts unary} also implies that it is essentially surjective. Under this identification $\Fun(\O^\el, \Cat) \simeq \Algd(\O)_{/\iota(\ast)}$, the adjunction $\iota \dashv \Gamma$ is thus induced by the unary algebrad $\iota(\ast)$ via push-forward and base-change.
\end{remark}

We can now easily show that the underlying graph functor preserves exponential objects, establishing \cref{thmC}:

\begin{theorem}\label{thm:graph preserve exponential objects}
    The graph functor $\Gamma$ preserves exponential objects. Precisely, for each exponentiable algebrad $X$, and every algebrad $Y$, the canonical comparison map 
    $
    \Gamma[X,Y] \to [\Gamma X, \Gamma Y]
    $
    is an equivalence in $\Fun(\O^{\el}, \Cat)$.
\end{theorem}
\begin{proof}
Suppose that $Y$ is an $\O$-algebrad, and that $X$ is an exponentiable $\O$-algebrad. The map $\Gamma[X,Y] \to [\Gamma X,\Gamma Y]$ represents the map between presheaves 
    $$
    \map_{\Algd(\O)}(\iota(-) \times X, Y) \to \map_{\Algd(\O)}(\iota((-) \times \Gamma X), Y)
    $$
    that comes from the natural comparison map $\gamma_A : \iota(A \times \Gamma X) \to \iota A \times X$ defined for each algebrad $A$. 
    Thus it suffices to show that $\gamma_A$ is an equivalence. This follows from \cref{cor:unaries form ideal}. 
\end{proof}

\begin{example}
    We consider the pattern $\O = \F_*^\flat$ whose category of algebrads is given by the category $\mathrm{Op} = \Algd(\F_*^\flat)$ of operads. In this case, we obtain an adjunction 
    $$
    \iota : \Cat \rightleftarrows \mathrm{Op} : \Gamma.
    $$
    The underlying graph of an operad $\P$ is precisely the underlying category $\P_{
    \langle 1 \rangle}$ of the operad. The operad $\iota(\ast)$ can be identified with the \textit{trivial operad} of \cite[Example 2.1.1.20]{HA}.
    Suppose that $\P$ is an exponentiable operad. Then for every other operad $
    \Q$, the underlying category of the exponential object $[\P,\Q]$ is given by $\Fun(\P_{\langle 1 \rangle}, \Q_{\langle 1 \rangle}),$ cf.\ \cite[Proposition 2.2.6.4]{HA}.
    
    There is also an equivariant analogue of this example that is obtained when considering the pattern $\O = \Span(\F_G)^\flat$ of \cref{examples:examples of patterns} for $G$-operads, cf.\ \cite[Proposition 3.1.9]{NardinShah}. In this case, we obtain an adjunction 
    $$
    \iota : \Fun(\Orb_G^\op, \Cat) \rightleftarrows  \Algd(\Span(\F_G)^\flat) : \Gamma,
    $$
    where $\Orb_G \subset \F_G$ is the full subcategory spanned by the $G$-orbits. The functor $\Gamma$ assigns the underlying $G$-category to a $G$-operad.
\end{example}

\begin{example}\label{exa:underlying-graph-vdc}
    We consider the pattern $\O=\Delta^{\op,\natural}$ whose category of algebrads $\Algd(\Delta^{\op,\natural}) \simeq \VirtDblCat$ is the category of virtual double categories (see \cref{examples:examples of patterns}). In this case, $$\mathbb{G} \coloneqq \O^\el = \{[1] \rightrightarrows [0]\}$$ is the category of two parallel arrows. Thus we obtain an adjunction 
    $$
       \iota : \Fun(\mathbb{G},\Cat) \rightleftarrows \VirtDblCat : \Gamma.
    $$
    For this specific example, one may view  \cref{prop:unary algebrads} as a generalization of \cite[Proposition 3.3]{Arkor}, and \cref{thm:graph preserve exponential objects} as a generalization of \cite[Lemma 3.8]{Arkor} to the $\infty$-categorical setting.
    
    Let $X, Y \in \Fun(\mathbb{G}, \Cat)$. One can verify that the exponential object $Z \coloneqq [X,Y]$ in $\Fun(\mathbb{G}, \Cat)$ is characterized by the pullback diagram
    \[
        \begin{tikzcd}
            Z_1 \arrow[r]\arrow[d]& \Fun(X_1, Y_1) \arrow[d] \\
            Z_0^{\times 2} \arrow[r] & \Fun(X_1, Y_0)^{\times 2},
        \end{tikzcd}
    \] 
    and where $Z_0$ is defined to be $\Fun(X_0, Y_0)$, and the bottom functor in the square is the obvious one, cf.\ \cite[Proposition 3.6]{Arkor}. If we combine this with the fact that $\Gamma$ preserves exponential objects, we obtain a concrete description of the cells of exponential objects of virtual double categories.
\end{example}

\section{Robust algebraic patterns}\label{section:robust patterns}
In \cref{section:necessity}, we will show the converse of \cref{thm:main-thm-expo-CSeg} (and hence also \cref{thm:main-thm-expo-factorization-Algd}) for so-called \emph{robust} algebraic patterns. This preliminary section is dedicated to introducing this class of patterns. We will discuss a range of examples of robust patterns in \cref{subsec:examples-robust}. The reader may safely skip to \cref{section:robust patterns} and consult the results of this section when needed.

\subsection{The $\pi_0$ functor}
We will start with a general construction that will be necessary to state the definition of the robust patterns. 

Let $x$ be an object of an algebraic pattern $\O$. Then we will write 
$$\pi_0(x):=\pi_0|\O^{\el}_{x/}|$$
for the set of connected components of its elementary slice.
Any inert morphism $f:x\intarrow y$ in $\O$ induces a function 
$\pi_0(y)\to \pi_0(x)$ by precomposition. Moreover, if $\O$ is \emph{sound} in the sense of \cite[Definition 3.3.2]{BarkanHaugsengSteinebrunner}, then  any active map $g:x\actarrow y$ induces a morphism $\pi_0(x)\to \pi_0(y)$. Indeed, following  \cite{BarkanHaugsengSteinebrunner}, we define the category $\cO^\elem(g)$ as the pullback
\[\begin{tikzcd}
	{\cO^\elem(g)} & {\Ar(\cO^\inert_{x/})} \\
	{\cO^\elem_{y/} \times \cO^\elem_{x/}} & {\cO^\inert_{x/} \times \cO^\inert_{x/}}.
	\arrow[from=1-1, to=1-2]
	\arrow[from=1-1, to=2-1]
	\arrow["\lrcorner", phantom, very near start, draw=none, from=1-1, to=2-2]
	\arrow[from=1-2, to=2-2]
	\arrow[from=2-1, to=2-2]
\end{tikzcd}\]
Recall that $\cO$ is defined to be {sound} if for any active morphism $g \colon x \actarrow y$, the projection $\cO^\elem(g) \to \cO^\elem_{x/}$ is initial.
This in particular implies that $\pi_0|\cO^\elem(g)| \to \pi_0|\O^{\el}_{x/}| = \pi_0(x)$ is an equivalence when $g$ is active, hence we obtain the desired map $\pi_0(x)\to \pi_0(y)$. We will now show that we can extend this assignment to an \textit{oplax} functor $\pi_0:\O\to \Span(\Set)$ where $\Span(\Set)$ is the $2$-category of spans of sets.

\begin{construction}
Let $\Ar^\inert(\cO) \subset \Ar(\cO)$ be the full subcategory spanned by inert morphisms. Then $\ev_0 \colon \Ar^\inert(\cO) \to \cO$ is the cartesian fibration corresponding to the functor $\cO^\inert_{-/} \colon \cO^\op \to \Cat$.
We write $\Ar^\inert(\cO)^\vee \to \cO^\op$ for the \textit{cocartesian} fibration corresponding to this functor.
    We will consider the full subcategory $\Ar^\inert(\cO)^{\vee,\elem}$ spanned by the inert arrows $x \intarrow e$ whose target is an elementary.
   
\end{construction}

\begin{remark}
    We note that $\Ar^\inert(\O)^{\vee, \el} \to \O^\op$ is cocartesian when restricted to the inert maps.
\end{remark}

\begin{remark}\label{rem:sections of the expo fib}
    Using the description of the cartesian lifts of $\Ar^\inert(\O)\to \O$, we can compute the fibers of 
    $
    \Fun([1], \Ar^\inert(\O)^{\vee,\el}) \to \Fun([1], \O^\op)
    $
    above a morphism $f : x\to y$ as follows. It is given by the category of commutative diagrams in $\O$ of shape 
    \[\begin{tikzcd}
	x \arrow[r,"f"]\arrow[d,tail] & y\arrow[d,tail] \\
	{x'} \arrow[d,tail]\arrow[r,squiggly]& e. \\
	{e'}
\end{tikzcd}\]
    This recovers the category $\O^\el(f)$ defined before if $f$ is active. We will use the same notation for this category if $f$ is not necessarily active.
\end{remark}

\begin{lemma}
    \label{lema:expo result}
    Let $\O$ be a sound pattern. Then the functor $\Ar^\inert(\cO)^{\vee,\elem}\to \O^\op$ is exponentiable.
\end{lemma}
\begin{proof}
Let $f:x\to y$ and $g:y\to z$ be two maps in $\O$, and $\alpha$ be a morphism in $\Ar^\inert(\cO)^{\vee,\elem}$ over $g\circ f$. By \cref{rem:sections of the expo fib}, the morphism $\alpha$ corresponds to a diagram 
\[\begin{tikzcd}[sep=scriptsize]
	x & y & z \\
	{x'} && e \\
	e'
	\arrow["f", from=1-1, to=1-2]
	\arrow[tail, from=1-1, to=2-1]
	\arrow["g", from=1-2, to=1-3]
	\arrow["i", tail, from=1-3, to=2-3]
	\arrow[squiggly, from=2-1, to=2-3]
	\arrow["j"', tail, from=2-1, to=3-1]
\end{tikzcd}\]
with $e$ and $e'$ elementary. To apply the Conduch\'e criterion (see \cite[Proposition B.3.2]{HA} or \cite[Lemma 2.2.8]{AyalaFrancis2020FibrationsInftyCategories}), we have to show that the category of diagrams in $\O$ of shape
\[\begin{tikzcd}[sep=scriptsize]
	x & y & z \\
	{x'} & {y'} & e \\
	{x''} & {e''} \\
	e'
	\arrow["f", from=1-1, to=1-2]
	\arrow[tail, from=1-1, to=2-1]
	\arrow["g", from=1-2, to=1-3]
	\arrow[tail, from=1-2, to=2-2]
	\arrow["i", tail, from=1-3, to=2-3]
	\arrow["{f'}", squiggly, from=2-1, to=2-2]
	\arrow[tail, from=2-1, to=3-1]
	\arrow["j"', curve={height=12pt}, tail, from=2-1, to=4-1]
	\arrow[squiggly, from=2-2, to=2-3]
	\arrow[tail, from=2-2, to=3-2]
	\arrow[squiggly, from=3-1, to=3-2]
	\arrow[tail, from=3-1, to=4-1]
\end{tikzcd}\]
is weakly contractible. However, this category is equivalent to the fiber of $\O^{\el}(f')\to \O^{\el}_{e'/}$ at $j$ and is hence weakly contractible by \cite[Lemma 3.3.9]{BarkanHaugsengSteinebrunner}.
\end{proof}

\begin{construction}
\label{construction:oplax functor between corr and span set}
    Let $\mathrm{Prof}$ be the (flagged) 2-category of categories and profunctors/\allowbreak correspondences/\allowbreak two-sided fibrations.\footnote{The underlying (flagged) 1-category is the \textit{opposite} of the category $\mathrm{Corr}$ considered by Ayala--Francis \cite{AyalaFrancis2020FibrationsInftyCategories}.} A \textit{two-sided fibration} from a category $\C$ to a category $\D$ is a span $(p,q) : E \to \C \times \D$ so that its fibers are spaces, $p$ is cocartesian with cocartesian morphisms lying above equivalences in $\C$, and $q$ is cartesian with cartesian morphisms lying above equivalences in $\D$. These were called \textit{bifibrations} in \cite[\S 2.4.7]{HTT}.
    We will now construct a unital oplax functor $\mathrm{Prof}\to \Span(\Set)$ that carries a two-sided fibration $E \to \C \times \D$ to the span $\pi_0|E| \to \pi_0|\C| \times \pi_0|\D|$.
    
    The (flagged) 2-category $\mathrm{Prof}$ arises as the horizontal fragment (see e.g.\ \cite[Section 3.3]{Ruit2025Companions}) of the double category $\mathbb{C}\mathrm{at}$ of categories. Similarly, the 2-category $\Span(\Set)$ is the horizontal fragment of the double category $\SSpan(\Set)$ of spans.
     There is a canonical functor $\mathbb{S}\mathrm{pan}(\Spc) \to \mathbb{C}\mathrm{at}$ between double categories that selects the full subdouble category spanned by the spaces.\footnote{Alternatively, the canonical inclusion from $\mathbb{S}\mathrm{pan}(\Spc)$ to $\mathbb{C}\mathrm{at}$ can be obtained by viewing $\mathbb{C}\mathrm{at}$ as a subdouble category of the Morita construction applied to $\mathbb{S}\mathrm{pan}(\Spc)$; see \cite[Proposition 7.1]{Blom2024StraighteningEveryFunctor}.} Consider the composite inclusion $$\mathbb{S}\mathrm{pan}(\Set) \to \mathbb{S}\mathrm{pan}(\Spc) \to \mathbb{C}\mathrm{at}.$$
    This functor recovers the inclusion $\Set \to \Cat$ on vertical categories, and interprets each span of sets as a two-sided fibration. We claim that this functor has an oplax left adjoint;\ i.e.\ that it has a left adjoint in the ambient 2-category $\Fun^\mathrm{oplax/int-strong}(\Delta^\op, \Cat)$ of functors and oplax natural transformations that are strong when restricted to inerts. By e.g.\ the dual of \cite[Corollary 6.9]{Ruit2025Companions}  and the Segal condition, it suffices to check that the horizontal functors in the square 
    \[
        \begin{tikzcd}
            \Fun(\bullet \leftarrow \bullet \rightarrow \bullet, \Set)\arrow[d] \arrow[r]\arrow[d] & \mathrm{TsFib}\arrow[d] \\
            \Set \times \Set \arrow[r] & \Cat \times \Cat
        \end{tikzcd}
    \]
    admit left adjoints, and that the associated mate is an equivalence. This is readily verified. Here $\mathrm{TsFib} \subset \Fun(\bullet \leftarrow \bullet \rightarrow \bullet, \Cat)$ is the full subcategory spanned by the two-sided fibrations.  The resulting left adjoint functor $\pi_0|-|: \mathbb{C}\mathrm{at} \to \mathbb{S}\mathrm{pan}(\Set)$ is also unital. To see this, we have to verify that the mate of the square 
    \[
        \begin{tikzcd}
            \Set \arrow[d] \arrow[r] & \Cat \arrow[d]\\
            \Fun(\bullet \leftarrow \bullet \rightarrow \bullet, \Set) \arrow[r]& \mathrm{TsFib}
        \end{tikzcd}
    \]
    is also invertible. But this follows from the fact that $|\Fun([1], \C)| \simeq |\C|$. We can now take horizontal fragments to obtain the desired oplax functor $\mathrm{Prof} \to \Span(\Set)$.
\end{construction}

\begin{construction}
    By \cite[Theorem 0.8]{AyalaFrancis2020FibrationsInftyCategories} or \cite[Corollary 6.2]{Blom2024StraighteningEveryFunctor}, and \cref{lema:expo result}, the functor 
    $\Ar^\inert(\cO)^{\vee,\elem} \to \cO^\op$ is classified by a functor $E:\O\to \mathrm{Prof}$ to the (flagged)  2-category of categories and profunctors.
    Composing with the (unital) oplax functor 
     $\mathrm{Prof}\to \Span(\Set)$ of \cref{construction:oplax functor between corr and span set},  we obtain a unital oplax functor $\pi_0:\O\to \Span(\Set)$ that sends a map $f \colon x \to y$ to the span
    \[\begin{tikzcd}[sep={2.5em, between origins}]
    	& {\pi_0|\cO^\elem(f)|} & \\
    	{\pi_0(x)} && {\pi_0(y)},
    	\arrow[from=1-2, to=2-1]
    	\arrow[from=1-2, to=2-3]
    \end{tikzcd}\]
    see \cref{rem:sections of the expo fib}.
    If $f$ is inert, then the right leg is an equivalence since $\Ar^\inert(\cO)^{\vee,\elem} \to \cO^\op$ is a cocartesian fibration over the inert morphisms, and if $f$ is active then the left leg is an equivalence by soundness.
    We conclude that $\pi_0$ is of the desired form.
\end{construction}

\begin{remark}\label{remark:oplax-comparison-pi0}
    Let $f \colon x \to y$ and $g \colon y \to z$ be morphisms in $\cO$.
    Since $E \colon \cO \to \mathrm{Prof}$ is a strong functor, as opposed to (op)lax, the composition formula for two-sided fibrations (see \cite{AyalaFrancis2020FibrationsInftyCategories}) implies that the canonical map
    $|\cO^\elem(f) \times_{\cO^\elem_{y/}} \cO^\elem(g)| \to |\cO^\elem(gf)|$
    is an equivalence.
    In particular, we obtain a diagram
    \[\begin{tikzcd}[sep={3em,between origins}]
    	&& {\pi_0|\cO^\elem(gf)|} && \\
    	& {\pi_0|\cO^\elem(f)|} && {\pi_0|\cO^\elem(g)|} \\
    	{\pi_0(x)} && {\pi_0(y)} && {\pi_0(z)}
    	\arrow[from=1-3, to=2-2]
    	\arrow[from=1-3, to=2-4]
    	\arrow[from=2-2, to=3-1]
    	\arrow[from=2-2, to=3-3]
    	\arrow[from=2-4, to=3-3]
    	\arrow[from=2-4, to=3-5]
    \end{tikzcd}\]
    where the square is obtained by applying $\pi_0|-|$ to the pullback defining $\cO^\elem(f) \times_{\cO^\elem_y/} \cO^\elem(g)$
    and then identifying $\pi_0|\cO^\elem(f) \times_{\cO^\elem_{y/}} \cO^\elem(g)|$ with $\cO^\elem(gf)$ as above.
    The gap map of this square precisely defines the oplax comparison 2-morphism $\pi_0(gf) \Rightarrow \pi_0(g) \circ \pi_0(f)$ of $\pi_0 \colon \cO \to \Span(\Set)$.
\end{remark}

\begin{lemma}
    Suppose that any square of the form
    \begin{equation}\label{eq:inertactivsquare}
    \begin{tikzcd}
        	w & z \\
        	x & y
        	\arrow[squiggly, from=1-1, to=1-2]
        	\arrow[tail, from=2-1, to=1-1]
        	\arrow[tail, from=2-2, to=1-2]
        	\arrow[squiggly, from=2-1, to=2-2]
    \end{tikzcd}
    \end{equation}
    is sent to a cartesian square of sets by the functor $\pi_0$.
    Then the unital oplax functor $\pi_0 \colon \cO \to \Span(\Set)$ sending $x$ to $\pi_0(x)$ is a strong functor.
\end{lemma}

\begin{proof}
    By \cref{remark:oplax-comparison-pi0}, it suffices to show that the square
    \begin{equation}\label{eq:oplax-cart-square}
    \begin{tikzcd}
    	{\pi_0|\cO^\elem(f) \times_{\cO^\elem_{y/}} \cO^\elem(g)|} & {\pi_0|\cO^\elem(g)|} \\
    	{\pi_0|\cO^\elem(f)|} & {\pi_0|\cO^\elem_{y/}|}
    	\arrow[from=1-1, to=1-2]
    	\arrow[from=1-1, to=2-1]
    	\arrow[from=1-2, to=2-2]
    	\arrow[from=2-1, to=2-2]
    \end{tikzcd}
    \end{equation}
    is cartesian for any $f \colon x \to y$ and $g \colon y \to z$ in $\cO$.
    In light of the inert-active factorization system on $\cO$, we may assume without loss of generality that $f$ is active or inert, and similarly for $g$.
    If $f$ is inert, then $\cO^\elem(f) \to \cO^\elem_{y/}$ admits a fully faithful right adjoint by \cite[Lemma 3.2.6]{AyalaFrancis2020FibrationsInftyCategories}, hence so does $\cO^\elem(f) \times_{\cO^\elem_{y/}} \cO^\elem(g) \to \cO^\elem(g)$ (see e.g.\ \cite[Proposition 2.6]{HaineRamziSteinPushouts}).
    This means that the horizontal maps in the square \cref{eq:oplax-cart-square} are bijections, hence it is cartesian.
    If instead $g$ is active, then $\cO^\elem(g) \to \cO^\elem_{y/}$ is initial.
    Since $\cO^\elem(f) \to \cO^\elem_{y/}$ is a cartesian fibration, the map $\cO^\elem(g) \times_{\cO^\elem_{y/}} \cO^\elem(f) \to \cO^\elem(f)$ is also initial.
    We therefore see that the vertical maps in the square \cref{eq:oplax-cart-square} are bijections, hence it is cartesian.
    We are therefore left with the case that $f$ is active and $g$ is inert.
    In this case, by factoring the composite $gf$ into an inert $x \intarrow w$ followed by an active $w \actarrow z$ and considering the previous cases, it follows that the square \cref{eq:oplax-cart-square} agrees with the square obtained by applying $\pi_0$ to \cref{eq:inertactivsquare}.
\end{proof}

\begin{remark}\label{rem:naturality of pi0}
    Let $f: \O \to \P$ be a map between algebraic patterns that induces a bijection $\pi_0(x) \to \pi_0f(x)$ for every $x \in \O$. Then we claim that we have a commutative diagram
    \[
        \begin{tikzcd}
            \O \arrow[r, "\pi_0^\O"]\arrow[d,"f"'] & \Span(\Set) \\
            \P \arrow[ur, "\pi_0^\P"'] & 
        \end{tikzcd}
    \]
    of oplax functors. Namely, we have a commutative diagram  
    \[
        \begin{tikzcd}
            \Ar^\inert(\O)^{\vee,\elem} \arrow[r]\arrow[d] & \Ar^\inert(\P)^{\vee,\elem} \arrow[d] \\
            \O^\op \arrow[r] & \P^\op.
        \end{tikzcd}
    \]
    The induced map $\Ar^\inert(\O)^{\vee,\elem} \to \Ar^\inert(\P)^{\vee,\elem} \times_{\P^\op} \O^\op$ is classified by a map $[1]_v \times \O_h \to \CCAT$; see \cite[Proposition 3.8.5]{RuitThesis} for a proof following Ayala--Francis \cite{AyalaFrancis2020FibrationsInftyCategories}, or \cite[Theorem 8.1]{Blom2024StraighteningEveryFunctor}. Thus we obtain an oplax functor $[1]_v \times \O_h \to \SSpan(\Set)$ from $\pi_0^\O$ to $\pi_0^\P \circ f$. By assumption, this factors through $[1]_v \times \O_h \to \O_h$. 
\end{remark}

\subsection{Robust patterns} We introduce the following terminology. 

\begin{definition}\label{def:new robust}
    An algebraic pattern $\O$ is called \textit{robust} if the following conditions hold:
    \begin{enumerate}[(1)]
        \item $\cO$ is sound in the sense of \cite[Definition 3.3.4]{BarkanHaugsengSteinebrunner}.
        \item $\cO$ is saturated in the sense of \cite[Definition 14.15]{ChuHaugseng2021HomotopycoherentAlgebraSegal},\ i.e.\ for every $x \in \O$, the presheaf $\Hom_{\cO}(x,-)$ is a Segal $\O$-space (see \cref{def:segal objects}).
        \item A map $x \to y$ in $\cO$ is inert if and only if for any inert $y \intarrow e$ with $e$ elementary, the composite $x \to y \intarrow e$ is inert.
        \item \label{item} The (unital) oplax functor  $\pi_0 : \O \to \Span(\Set)$ satisfies the following conditions:
            \begin{enumerate}
        \item For any $x$, $\pi_0(x)$ is a finite set, i.e.\ each elementary slice $\O^{\el}_{x/}$ has a finite number of connected components.
        \item The functor $\pi_0$ is strong, or equivalently, it sends any square of shape  \cref{eq:inertactivsquare} to a cartesian square of sets.
        \item If $x \in \O$ lies over $\underline{n}$, there exist $\pi_0$-cocartesian inert maps $x \intarrow x^i$ over the inerts $\underline{n} \hookleftarrow \{i\} \xrightarrow{=} \{i\}$ 
        that induce an equivalence
        $$
        \O^\act_{/x} \to \prod_{i=1}^n \O^\act_{/x^i}.
        $$
        \item
        For any commutative square 
         \[
                \begin{tikzcd}
                    a \arrow[d,tail]\arrow[r,squiggly]& x \arrow[d, tail] \\ b \arrow[r,squiggly] & x^i,
                \end{tikzcd}
        \]
        in $\O$ so that the right vertical arrow is a $\pi_0$-cocartesian lift of the inert $\underline{n} \hookleftarrow \{i\} \xrightarrow{=} \{i\}$, the left vertical arrow is $\pi_0$-cocartesian as well.
    \end{enumerate}
    \end{enumerate}
\end{definition}

\begin{remark}
    The idea of this definition is that it allows one to talk about \emph{partial composites} of active morphisms in the following sense: Suppose $y \actarrow x$ is active, let $j \in \pi_0(y)$ and let $y \intarrow y^j$ be as in condition (4c). Given an active $z \actarrow y^j$, let $\overline{z} \actarrow y$ be the map that, under the equivalence $\cO^\activ_{/y} \simeq \prod_{i=1}^n \cO^\activ_{/y^i}$, corresponds to $y^i \xrightarrow{=} y^i$ if $i \neq j$ and to $z \actarrow x^j$ when $i=j$.
    One can then think of the composite $\overline{z} \actarrow y \actarrow x$ as a partial composite of $z \actarrow y^i$ with $y \actarrow x$.
    Condition (4d) will guarantee that one can form multiple such partial composites after each other for a collection of elements of $\pi_0(y)$.
    These ideas will be made precise in \cref{subsec:grafting,subsec:multigrafting}.
\end{remark}

\begin{definition}\label{def:atomically robust}
    An algebraic pattern $\O$ is called 
    \textit{atomically robust} if it satisfies conditions (1)-(3) of \cref{def:new robust} and $\pi_0(x)\simeq\underline{1}$ for every $x\in \O$, i.e.\ if each elementary slice $\O^{\el}_{x/}$ is connected.
\end{definition}

\begin{remark}
    Every atomically robust pattern is robust, as the functor $\pi_0:\O\to \Span(\Set)$ factors through $\underline{1} : \ast \to \Span(\F)$ and then automatically meets condition \cref{item} of \cref{def:new robust}.
\end{remark}

A range of examples of robust patterns will be showcased in \cref{subsec:examples-robust}.
We now state some results about robust patterns that will be useful later.

\begin{notation}
    Given $t\in \O_n$, we will write $\pi_0(t)$ for the set of connected components of $(\O^{\el}_n)_{t/}$. Note that $\pi_0(t) \simeq \pi_0(t_n)$ by \cref{rem:slices-factorization-double-category}.
\end{notation}

We recall the algebraic pattern structure $\Span(\F)^\flat$ on $\Span(\F)$ from \cref{examples:examples of patterns} where the backward maps are the inert morphisms, the forward maps are the active morphisms and the elementaries are given by $\{\underline 1\}$. One readily reads off from the construction that the functor $\pi_0 : \O \to \Span(\F)$ preserves inerts, actives and elementaries. Consequently, it may be viewed as a map of patterns $\O \to \Span(\F)^\flat$.
In turn, this gives rise to a functor between factorization double categories via \cref{prop:factorization system and double factorysation caregory}. Thus we have functors $\O_n \to \Span(\F)_n$ between the degree $n$ parts of these double categories. 

\begin{proposition}\label{lem:pi0 levelwise cocart}
    Let $\O$ be a robust pattern.
    The functor $\O_n \to \Span(\F)_n$ induced by $\pi_0$ admits cocartesian lifts of maps with elementary codomains.
\end{proposition}
\begin{proof}
    We first handle the case that $n = 0$. Let $x \in \O$, and suppose that we have a map $i:\pi_0(x) \to \underline{1}$ in $\Span(\F)_0 = \F^\op$. Let $x \intarrow x^i$ be the associated $\pi_0$-cocartesian lift. We may then consider the commutative square 
    \[
        \begin{tikzcd}
            \Hom_{\O_0}(x^i, y) \arrow[r]\arrow[d] & \Hom_{\O_0}(x, y) \times_{\Hom_{\F^\op}(\pi_0(x), \pi_0(y))}\Hom_{\F^\op}(\underline{1}, \pi_0(y)) \arrow[d]  \\
            \Hom_{\O}(x^i, y) \arrow[r] & \Hom_{\O}(x, y) \times_{\Hom_{\Span(\F)}(\pi_0(x), \pi_0(y))}\Hom_{\Span(\F)}(\underline{1}, \pi_0(y))
        \end{tikzcd}
    \]
    for each $y \in \O$.
    The bottom arrow is an equivalence by assumption, and the vertical arrows are monomorphisms. The top arrow is also an equivalence by cancellation of inert morphisms.
    
    For general $n$, suppose that $t \in \O_n$ and that we have a map $\pi_0(t) \to X$ in $\Span(\F)_n$ with $X_n = \underline{1}$. We have an equivalence $(\O_n)_{t/} \simeq (\O_0)_{t_n/}$ by \cref{rem:slices-factorization-double-category}. Let $t \to t^i$ be the map so that $t_n \to t_n^i$ is the $\pi_0$-cocartesian lift of $\pi_0(t_n) \to X_n = \underline{1}$. It then follows from the above 
    that the map $(\O_n)_{t^i/} \to (\O_n)_{t/} \times_{(\Span(\F)_n)_{\pi_0(t)/}} (\Span(\F)_n)_{X/}$ is an equivalence.
\end{proof}

\begin{proposition}\label{prop:robust elementary slices decomp}
    Let $\O$ be a robust pattern. If $x\in \O$ lies over $\underline{n}$, then the $\pi_0$-cocartesian inert maps $x\intarrow x^i$ over $\underline{n} \hookleftarrow \{i\}\xrightarrow{=}\{i\}$ induce an equivalence
    $$\coprod_{i=1}^n \O^\el_{x^i/} \to \O^\el_{x/}.$$
\end{proposition}
\begin{proof}
    We will proceed fiberwise. For $y \in \O^\el$, the canonical square
    \[
        \begin{tikzcd}
            \coprod_{i=1}^n\Hom_{\O_0}(x^i,y) \arrow[r]\arrow[d] & \Hom_{\O_0}(x,y) \arrow[d] \\
            \coprod_{i=1}^n \Hom_{\F^\op}(\underline{1}, \pi_0(y)) \arrow[r] & \Hom_{\F^\op}(\underline{n}, \pi_0(y))
        \end{tikzcd}
    \]
    is cartesian by \cref{lem:pi0 levelwise cocart}. The bottom arrow is an equivalence as $\pi_0(y) \simeq \underline 1$, hence the desired conclusion follows.
\end{proof}

\begin{remark}\label{rem:decomposition by atoms formula}
Let $\O$ be a robust pattern and  $\left<n;t\right>$ be a forest. Note that the canonical map
    $
    \textstyle \coprod_{i\in \pi_0(t)} (\O_n^{\el})_{t^i/} \to (\O_n^{\el})_{t/}
    $ is an equivalence by \cref{rem:slices-factorization-double-category} and \cref{prop:robust elementary slices decomp},
    where $t\to t^i$ is the cocartesian lift of the map $\pi_0(t_n) \hookleftarrow \{i\} \xrightarrow{=} \{i\}$.
   We may then use the decomposition formula of \cref{rem:forest decomposition sound} to conclude that the canonical map 
    $$
    \coprod_{i\in \pi_0(t)} X \boxtimes_{[n]} [n;t^i] \to X \boxtimes_{[n]} [n;t]
    $$
    is an equivalence in $\PSh(\Tree[\O])$ for every simplicial space $X$ over $[n]$.
\end{remark}

\begin{proposition}\label{prop:robust-On-is-saturated}
    Let $\O$ be a robust pattern.
    Then the tuple $(\cO_n,\cO_n^\elem)$ is saturated; that is, for any $t$ and $s$ in $\cO_n$, the map
        \[\Map_{\cO_n}(s,t) \to \operatorname{lim}_{u \in (\cO^\elem_n)_{t/}} \Map_{\cO_n}(s,u)\]
        is an equivalence.
\end{proposition}
\begin{proof}
    For $n = 0$, this follows directly from assumptions (2) and (3). Suppose that $n=1$. Since $\cO_1$ is a subcategory of $\Ar(\cO)$, the map of the proposition is a map between pullbacks:
    \[\begin{tikzcd}[ampersand replacement=\&,column sep=tiny]
    	\& {\lim\limits_{(t \intarrow u) \in (\cO^\elem_{1})_{t/}}\Map_{\cO_1}(s,u)} \&\& {\lim\limits_{(t \intarrow u) \in (\cO^\elem_{1})_{t/}} \Map_{\cO_0}(s_0,u_0)} \\
    	{\Map_{\cO_1}(s,t)} \&\& {\Map_{\cO_0}(s_0,t_0)} \\
    	\& {\lim\limits_{(t \intarrow u) \in (\cO^\elem_{1})_{t/}} \Map_{\cO_0}(s_1,u_1)} \&\& {\lim\limits_{(t \intarrow u) \in (\cO^\elem_{1})_{t/}} \Map_{\cO}(s_0,u_1)}. \\
    	{\Map_{\cO_0}(s_1,t_1)} \&\& {\Map_{\cO}(s_0,t_1)}
    	\arrow[from=1-2, to=1-4]
    	\arrow[from=1-2, to=3-2]
    	\arrow["\lrcorner"{xshift=-20pt,yshift=8pt}, phantom, very near start, from=1-2, to=3-4]
    	\arrow[from=1-4, to=3-4]
    	\arrow[from=2-1, to=1-2]
    	\arrow[from=2-1, to=4-1]
    	\arrow["\lrcorner"{xshift=-15pt,yshift=3pt}, phantom, very near start, from=2-1, to=4-3]
    	\arrow[from=2-3, to=1-4,"g_3"]
    	\arrow[from=3-2, to=3-4]
    	\arrow[from=4-1, to=3-2,"g_1"]
    	\arrow[from=4-1, to=4-3]
    	\arrow[from=4-3, to=3-4,"g_2"]
    	\arrow[from=2-3, to=4-3, crossing over]
    	\arrow[from=2-1, to=2-3, crossing over]
    \end{tikzcd}\]
    We see that $g_1$ is an equivalence since $(\cO_0,\cO_0^\elem)$ is saturated and $(\cO^\elem_{1})_{t/} 
    \simeq \cO^\elem_{t_1/}$, while $g_2$ is an equivalence since $\cO$ is saturated.
    To see that $g_3$ is an equivalence, note that since $\cO_0$ is saturated, we can rewrite it as the map
    \[\Map_{\cO_0}(s_0,t_0) \to \lim_{u \in (\cO^\elem_{1})_{t/}} \operatorname{lim}_{e \in (\cO^\elem)_{u_0/}} \Map_{\cO_0}(s_0,e).\]
    By \cite[Observation 3.3.6]{BarkanHaugsengSteinebrunner}, the soundness of $\cO$ and the saturatedness of $\cO_0$, this map is an equivalence.
    
    To handle the general case, we can use that $\O_\bullet$ is a double category and reason as in \cref{prop:RKan-trees-to-forest-sound}.
\end{proof}

\subsection{Verifying the robustness condition in practice}\label{subsec:checking-robustness} We will now describe a few methods for checking whether a pattern $\cO$ is robust.

\begin{proposition}\label{prop:robust soundly extendable}
    Suppose that $\O$ is a soundly extendable pattern. Then conditions (2) and (3) of \cref{def:new robust} can be replaced by the single condition: 
    \begin{enumerate}
        \item[(2\&3')] $(\O^\inert, \O^\el)$ is saturated.
    \end{enumerate}
    Moreover, condition (4c) can be replaced by:
    \begin{enumerate}
    \item[(4c')] If $x \in \O$ lies over $\underline{n}$, then there exist cocartesian inert maps $x \intarrow x^i$ over the inerts $\underline{n} \hookleftarrow \{i\} \xrightarrow{=} \{i\}$ for $1\leq i \leq n$.
    \end{enumerate}
\end{proposition}
\begin{proof}
    For the first assertion, we need to show that $\cO$ is saturated if and only if $\cO^\inert$ is.
    Since $\cO$ is extendable, left Kan extension along $\cO^\inert \hookrightarrow \cO$ preserves Segal objects by \cite[Proposition 8.8]{ChuHaugseng2021HomotopycoherentAlgebraSegal}.
    Applying this to representables, we see that $\cO$ is saturated if $\cO^\inert$ is.
    For the converse, note that by \cite[Remark 4.1.4]{BarkanHaugsengSteinebrunner}, $\cO_{x/}^\inert \to \cO$ is a fibrous pattern if the representable $\Map_\cO(x,-)$ satisfies the Segal condition.
    Being a fibrous pattern in particular implies the Segal condition (cf.\ \cite[Remark 4.1.8]{BarkanHaugsengSteinebrunner}), so since $\Map_{\cO^\inert}(x,-)$ is the straightening of $\cO_{x/}^\inert \to \cO$, we conclude that it satisfies the Segal condition.
    
    For the second assertion, suppose that $x \in \O$.
    We note that there is a commutative square 
    \[
        \begin{tikzcd}
            \O^{\act}_{/x} \arrow[r]\arrow[d] & \lim_{e \in \O^{\el}_{x/}} \O^{\act}_{/e}\arrow[d] \\
            \prod_{i \in \pi_0(x)} \O^{\act}_{/x^i} \arrow[r] & \prod_{i \in \pi_0(x)}\lim_{e' \in \O^{\el}_{x^i/}} \O^{\act}_{/e'},
        \end{tikzcd}
    \]
    so that the right vertical functor is an equivalence.
    As $\O$ is extendable, the top and bottom functors are equivalences. So the left vertical functor is an equivalence as well.
\end{proof}

\begin{corollary}\label{cor:atomically robust soundly extendable}
    Suppose that $\O$ is a soundly extendable pattern. Then $\O$ is atomically robust if and only if $(\cO^\inert,\cO^\elem)$ is saturated and its elementary slices are connected. \qed
\end{corollary}

Under an extra assumption, atomically robust patterns are stable under products:

\begin{proposition}\label{prop:product-robust-patterns}
    Let $\cO$ and $\cP$ be robust algebraic patterns with weakly contractible elementary slices. Then $\cO \times \cP$ is atomically robust.
\end{proposition}

\begin{proof}
    Observe that $(\cO \times \cP)^\elem_{(x,y)/} \simeq \cO^\elem_{x/} \times \cP^\elem_{y/}$ is weakly contractible for any $(x,y) \in \cO \times \cP$ and that $\cO \times \cP$ is sound by \cite[Lemma 3.3.13]{BarkanHaugsengSteinebrunner}.
    To see that $\cO \times \cP$ and $(\cO \times \cP)^\inert$ are saturated, observe that
    \[\lim_{(e,e') \in (\cO \times \cP)^\elem_{(x,y)/}} \Map((z,w),(e,e')) \simeq \lim_{e \in \cO^\elem_{x/}} \Map(z,e) \times \lim_{e' \in \cP^\elem_{y/}} \Map(w,e')\]
    since $\cO^\elem_{x/}$ and $\cP^\elem_{y/}$ are weakly contractible.
\end{proof}

Finally, when $\cO$ is robust, any $\cO$-algebrad is robust as well.

\begin{proposition}\label{prop:algebrads-are-robust}
    Let $\cO$ be a robust algebraic pattern. Then any algebrad $p \colon \cP \to \cO$ is also robust, when given the pattern structure from \cref{def:category-of-algebrads}.
\end{proposition}

\begin{proof}
    We first note that $p$ induces an equivalence on elementary slices since $\cP^\inert \to \cO^\inert$ is a left fibration, and thus $\P^\inert_{x/} \simeq \O^{\inert}_{p(x)/}$.
    By \cite[Lemma 4.1.15]{BarkanHaugsengSteinebrunner}, $\cP$ is again sound.
    To see that $\cP$ is saturated, we note that the saturatedness of $\cO$ implies that the bottom map in the pullback square \cref{defi:algd fact:item3} of \cref{defi:of algebrad for factactization system} is an equivalence.
    Since $\cP^\elem_{y/} \simeq \cO^\elem_{p(y)/}$, this implies saturatedness of $\cP$.
    Finally, saturatedness of $\cP^\inert$ follows by the same argument since $\cP^\inert \to \cO^\inert$ is an algebrad.
    This shows that (1)-(3) hold for $\cP$ as well.
    
    To verify condition (4), we note that conditions (4a) and (4b) directly follow from \cref{rem:naturality of pi0} (note that the remark simplifies in this case, as the displayed square is already a pullback square).
    Suppose that $x \in \P$. For every $i \in \pi_0(x) = \pi_0p(x)$, we take a $p$-cocartesian lift $x\to x^i$ of $p(x) \intarrow p(x)^i$. By \cite[Proposition 4.1.7]{BarkanHaugsengSteinebrunner}, we have a pullback square 
    \[
        \begin{tikzcd}
            \P \times_\O \O^\act_{/p(x)} \arrow[r]\arrow[d] & \prod_{i \in \pi_0(x)} \P \times_\O \O^\act_{/p(x)^i} \arrow[d]  \\
            \O^\act_{/p(x)} \arrow[r] & \prod_{i \in \pi_0(x)} \O^\act_{/p(x)^i},
        \end{tikzcd}
    \]
    where the top functor is described in \cite[Observation 4.1.1]{BarkanHaugsengSteinebrunner}.
    The bottom functor is an equivalence by assumption, so that the top functor is an equivalence as well. Similarly as in the proof of \cite[Lemma 4.1.15]{BarkanHaugsengSteinebrunner}, the top functor recovers the functor of (4c) when passing to the slices over $(x, p(x) = p(x))$ and the $(x^i, p(x)^i = p(x)^i)$'s.
\end{proof}

\subsection{Examples of robust patterns}\label{subsec:examples-robust}

We now discuss some examples and non-examples of robust patterns.

\begin{example}\label{example:fin*-not-robust}
    Let $\F_*^\flat$ be the pattern from \cref{examples:examples of patterns} whose algebrads are operads.
    Note that $\F_*^{\flat,\inert} = (\F_\mathrm{inj})^\op$, the opposite of the category of finite sets and injections.
    The two-element set $\underline{2}$ is \textbf{not} the coproduct $\underline{1} \sqcup \underline{1}$ in $\F_\mathrm{inj}$, since the fold map $\underline{2} \to \underline{1}$ is not injective.
    This means that $\langle 2 \rangle$ is not the product of $\langle 1 \rangle$ with itself in $\F_*^{\flat,\inert}$, so $\F_*^{\flat,\inert}$ is not saturated and hence $\F_*^\flat$ is not robust.
    One similarly sees that the algebraic pattern $\Delta^{\op,\flat}$ whose algebrads are non-symmetric operads is not robust.
\end{example}

\begin{example}\label{example:Span-robust}
    Let $G$ be a finite group. Then the algebraic pattern $\Span(\F_G)^\flat$ is robust. It is soundly extendable by \cite[Example 3.3.26]{BarkanHaugsengSteinebrunner}, and we see that $\Span(\F_G)^\inert = \F_G^\op$ is saturated since any object in $\F_G$ is the coproduct of its orbits, i.e.\ its subsets on which $G$ acts transitively. Thus condition (2\&3') of \cref{prop:robust soundly extendable} is met. It remains to check condition (4a), (4b) and (4d) of \cref{def:new robust}, and condition (4c') of \cref{prop:robust soundly extendable}. 
    
    If $X$ is a finite $G$-set, then we can consider the (finite) set $O(X)$ of orbits of $X$. This defines a functor $O : \F_G \to \F$. One readily verifies that this commutes with pullbacks and that $O(G/H)\simeq \underline{1}$ for each subgroup $H \leq G$. Hence, we get a functor $\Span(\F_G)^\flat \to \Span(\F)^\flat$ of patterns. By \cref{rem:naturality of pi0}, this is precisely the functor $\pi_0$. Thus conditions (4a) and (4b) of \cref{def:new robust} are met. Suppose that $X$ is a finite $G$-set.
    Then we can consider its decomposition $X= \coprod_{i=1}^n X^i$ into orbits.
    Then $\pi_0(X) \simeq \underline{n}$ and one readily checks that the inert maps $X \hookleftarrow X^i \xrightarrow{=} X^i$ in $\Span(\F_G)$ are the desired $\pi_0$-cocartesian lifts of (4c').
    Finally, condition (4d) of \cref{def:new robust} follows since pullbacks of injections are again injections in $\F_G$. 
\end{example}

\begin{example}\label{example:doublecat-robust}
    Consider the algebraic pattern $\Delta^{\op,\natural}$ for virtual double categories from \cref{examples:examples of patterns}.
    In contrast to \cref{example:fin*-not-robust}, this algebraic pattern \emph{is} atomically robust: it is soundly extendable by \cite[Example 3.3.18]{BarkanHaugsengSteinebrunner}, and it is saturated since the representables $[n] \in \Fun(\Delta^\op,\Spc)$ are Segal spaces.
    The elementary slice $(\Delta^{\op,\natural})^\elem_{[n]/}$ is an iterated span
    \[\begin{tikzcd}[sep={2.5em, between origins}]
    	& \bullet && \cdots && \bullet & \\
    	0 && 1 & & {n-1} && n
    	\arrow[from=1-2, to=2-1]
    	\arrow[from=1-2, to=2-3]
    	\arrow[from=1-4, to=2-3]
    	\arrow[from=1-4, to=2-5]
    	\arrow[from=1-6, to=2-5]
    	\arrow[from=1-6, to=2-7]
    \end{tikzcd}\]
    hence weakly contractible.
\end{example}

\begin{warning}
    One can also endow $\F_*$ with the pattern structure $\F_*^\natural$ where the elementaries are $\langle 0 \rangle$ and $\langle 1 \rangle$.
    The algebrads for this pattern are \emph{generalized operads} in the sense of \cite[Definition 2.3.2.1]{HA}.
    In light of the previous example, one might be inclined to believe that this algebraic pattern is atomically robust.
    However, this is \textbf{not} the case, for the same reason as \cref{example:fin*-not-robust}.
\end{warning}

\begin{example}\label{exa:gen-operad-robust}
    While $\F_*^\natural$ is not robust, we can replace it with an (atomically) robust pattern as follows.
    Namely, we will see in \cref{exa:CC-necessary-genoperad} that the inclusion $\F_* \hookrightarrow \Span(\F)$ induces an equivalence
    \[\Algd(\F^\natural_*) \simeq \Algd(\Span(\F)^\natural).\]
    The pattern $\Span(\F)^\natural$ is atomically robust: its elementary slices are clearly contractible, it is soundly extendable by \cite[Proposition 3.3.23]{BarkanHaugsengSteinebrunner} and $\Span(\F)^\inert \simeq \F^\op$ is easily seen to be saturated.
\end{example}

\begin{example}\label{example:theta-robust}
    The category $\Theta_n$ has a pattern structure whose elementary objects are the free $k$-cells $C_k$ with $k \leq n$ (see e.g.\ \cite[Example 3.2.2]{BarkanHaugsengSteinebrunner}).
    The Segal spaces for this pattern are (flagged) $n$-categories and its algebrads can be thought of as virtual versions of $(1,n)$-double categories.
    This pattern is saturated by \cite[Examples 14.21]{ChuHaugseng2021HomotopycoherentAlgebraSegal}, while it is soundly extendable by \cite[Example 3.3.18]{BarkanHaugsengSteinebrunner}.
    Finally, the elementary slices of $\Theta_n$ are weakly contractible by \cite[Corollary 2.9]{HaugsengTheta}.
    By \cref{prop:robust soundly extendable}, this algebraic pattern structure on $\Theta_n$ is atomically robust.
\end{example}

\begin{example}\label{example:Deltan-robust}
    The $n$-fold product $\Delta^{\times n,\op,\natural}$ of $\Delta^{\op,\natural}$ with itself can be used to describe (virtual versions of) $n$-uple categories.
    By \cref{example:doublecat-robust} and \cref{prop:product-robust-patterns}, the algebraic pattern $\Delta^{\times n, \op,\natural}$ is atomically robust.
\end{example}

\begin{example}\label{exa:all-elementary-then-robust}
    Let $\cO$ be an algebraic pattern in which every object is elementary.
    Then the inclusion of $t$ in $\cO^\elem_{t/}$ is initial for every $t$, hence the conditions of \cref{def:new robust} become vacuous.
    In particular, $\cO$ is atomically robust.
\end{example}

\begin{example}
    Recall the algebraic pattern structure on the tree category $\Omega[\cO]^{\op,\natural}$ from \cref{ex:iterated trees}.
    Even if $\cO$ is robust, the pattern $\Omega[\cO]^{\op,\natural}$ need not be robust.
    For example, if $\cO = \Span(\F)^\flat$, then $\Omega[\cO]^{\op,\natural}$ is not saturated and hence not robust.
    However, if $\cO$ is atomically robust, then the next lemma shows that the same is true for $\Omega[\cO]^{\op,\natural}$.
    In particular, the iterated tree construction $\Omega^n[\cO]^{\op,\natural}$ is atomically robust whenever $\cO$ is.
\end{example}

\begin{lemma}
\label{lemma: when tree is sound and robust}
Let $\O$ be an atomically robust pattern.
Then the pattern $\Tree[\O]^{\op,\natural}$ from \cref{ex:iterated trees} is also atomically robust.
\end{lemma}
\begin{proof}
    We will check the conditions from \cref{def:atomically robust}. We start by showing that $\O$ is sound.
    Let $\phi : [m] \actarrow [n]$ be an active map in $\Delta$. Suppose that $[n;t]$ is a tree. We will consider the cartesian map $f : [m; \phi^*t] \actarrow [n;t]$ in $\Tree[\O]$. This is an active map, and  \cref{proposition:explicit description of mapping space} implies that every active map is of this form. Suppose that $g : [k; u] \intarrow [n;t]$ is an inert map with $[k;u]$ elementary, i.e.\ $k \leq 1$.
    By \cite[Lemma 3.3.9]{BarkanHaugsengSteinebrunner}, we have to show that the category 
    \begin{equation}\label{eq:forest sound}
    \Tree[\O]^{\el}_{/[m;\phi^*t]} \times_{\Tree[\O]_{/[n;t]}^\inert} (\Tree[\O]_{/[n;t]}^\inert)_{g/}
    \end{equation}
    is weakly contractible. Recall from \cref{prop:segal tree spaces are segal objects} that the canonical map 
    $$
    \colim_{\alpha : [j] \intarrow [m] \in \Delta^\el_{/[m]}}  \Tree[\O]^{\el}_{/\left<j;\alpha^*\phi^*t\right>}  \to \Tree[\O]^{\el}_{/[m;\phi^*t]} 
    $$
    is an equivalence. As pulling back along the left fibration $(\Tree[\O]^\inert_{/[n;t]})_{g/} \to \Tree[\O]^\inert_{/[n;t]}$ preserves colimits, \cref{eq:forest sound} can be written as an iterated pushout of categories of the form
    \begin{equation}\label{eq: forest sound 2}
    \Tree[\O]^{\el}_{/\left<j; \alpha^*\phi^*t\right>} \times_{\Tree[\O]^\inert_{/[n;t]}} (\Tree[\O]_{/[n;t]}^\inert)_{g/}.
    \end{equation}
    for $\alpha \colon [j] \intarrow [m]$ with $j \leq 1$.
    To study these categories, consider the factorization 
    \[
        \begin{tikzcd}
           {[j]} \arrow[r, squiggly,"\gamma"]\arrow[d,tail, "\alpha"'] & {[l]}\arrow[d,tail, "\beta"] \\
        {[m]} \arrow[r, squiggly, "\phi"] & {[n]}
        \end{tikzcd}
    \]
    and let us write $\psi \colon [k] \to [n]$ for the underlying map of $g$.
    The pullback \cref{eq: forest sound 2} is non-empty if and only if $\psi$ factors through $\beta$, which is the case if and only if $\phi(\alpha(0)) \leq \psi(0) \leq \psi(k) \leq \phi(\alpha(j))$.
    Observe that there exists at least one inert map $\alpha \colon [1] \intarrow [m]$ for which this happens.
    We will show that the pullback \cref{eq: forest sound 2} is weakly contractible in this case.
    This then implies that the classifying space of the category \cref{eq:forest sound} is an iterated pushout of the form
    \[(\varnothing \cup_{\varnothing} \cdots \cup_{\varnothing} \varnothing) \cup_\varnothing (* \cup_* \cdots \cup_* *) \cup_\varnothing ( \varnothing \cup_\varnothing \cdots \cup_\varnothing \varnothing),\]
    hence weakly contractible.
    So suppose that $\psi$ factors through $\beta$.
    In this case $g$ factors uniquely as $[k;u] \xrightarrow{g'} \langle l; \beta^*t \rangle \intarrow [n;t]$ and we obtain an equivalence
    \begin{equation}\label{eq:forest sound 3}
    \Tree[\O]^{\el}_{/\left<j; \alpha^*\phi^*t\right>} \times_{\Tree[\O]^\inert_{/[n;t]}} (\Tree[\O]_{/[n;t]}^\inert)_{g/} \simeq \Tree[\O]^{\el}_{/\left<j; \gamma^*\beta^*t\right>} \times_{\Tree[\O]^\inert_{/\left<l;\beta^*t\right>}} (\Tree[\O]_{/\left<l,\beta^*t\right>}^\inert)_{g'/}
    \end{equation}
    To see that this is weakly contractible,
    consider the factorization 
    \[
    \begin{tikzcd}[column sep = tiny, row sep = small]
        (\O_j^{\el, \op})_{/\gamma^*\beta^*t} \arrow[rr]\arrow[dr] && \Tree[\O]^\el_{/\left<j;\gamma^*\beta^*t\right>}. \arrow[dl] \\
        & \Tree[\O]_{/\left<j;\gamma^*\beta^*t\right>}
    \end{tikzcd}
    \]
    The left slanted arrow is final since $\O$ is sound (see \cref{rem:forest decomposition sound}) and the right slanted arrow is final on account of \cref{prop:segal tree spaces are segal objects}.  Pulling back along a left fibration preserves final functors by \cite[Remark 4.1.2.10 \& Proposition 4.1.2.15]{HTT}. Since final functors induce equivalences on classifying spaces, 2-out-of-3 for weak equivalences implies that the induced functor 
    $$
    (\O_j^{\el, \op})_{/\gamma^*\beta^*t} \times_{\Tree[\O]^\inert_{/\left<l;\beta^*t\right>}} (\Tree[\O]_{/\left<l,\beta^*t\right>}^\inert)_{g'/} \to \Tree[\O]^{\el}_{/\left<j; \gamma^*\beta^*t\right>} \times_{\Tree[\O]^\inert_{/\left<l;\beta^*t\right>}} (\Tree[\O]_{/\left<l,\beta^*t\right>}^\inert)_{g'/}
    $$
    is a weak equivalence on classifying spaces.
    Now observe that the left-hand side fits in a diagram of pullback squares
    \[
        \begin{tikzcd}
            (\O_j^{\el, \op})_{/\gamma^*\beta^*t} \times_{\Tree[\O]^\inert_{/\left<l;\beta^*t\right>}} (\Tree[\O]_{/\left<l,\beta^*t\right>}^\inert)_{g'/} \arrow[r]\arrow[d] & ((\O_l^{\el, \op})_{/\beta^*t})_{g'/}\arrow[d]\arrow[r] & (\Tree[\O]_{/\left<l,\beta^*t\right>}^\inert)_{g'/}\arrow[d] \\
            (\O_j^{\el, \op})_{/\gamma^*\beta^*t} \arrow[r, "\simeq"] & (\O_l^{\el, \op})_{/\beta^*t} \arrow[r] & \Tree[\O]^\inert_{/\left<l;\beta^*t\right>},
        \end{tikzcd}
    \]
    Thus we must show that $((\O_l^{\el, \op})_{/\beta^*t})_{g'/}$ is weakly contractible.
    The equivalence $(\O_l^\elem)_{\beta^* t/} \simeq \O^\elem_{t_{\beta(l)}/}$ from \cref{remark:elementary-slices-On} induces an equivalence $((\O_l^{\el, \op})_{/\beta^*t})_{g'/} \simeq ((\O^\elem_{t_{\beta(l)}/})_{/u_k})^\op$.
    This category has an initial object since $u_k$ is elementary, hence it is weakly contractible.
    We conclude that \cref{eq:forest sound} is weakly contractible and hence that $\Tree[\O]$ is sound.
    
    We will now show that the elementary slices of $\Omega[\cO]$ are connected.
    We can write any elementary slice as an iterated pushout as above, so it suffices to show that for every forest $\left<k;t\right>$ with $k \leq 1$, the elementary slice $\Tree[\O]^\el_{/\left<k;t\right>}$ is connected. 
    By the same reasoning as \cref{eq:forest sound 3}, the map 
    $$
    (\O^{\el, \op}_{k})_{/t} \to \Tree[\O]^\el_{/\left<k;t\right>}
    $$
    induces a homotopy equivalence on classifying spaces. Since $\O$ has connected elementary slices by assumption, the result follows.

    Finally, observe that \cref{cor:fibrant forest} exactly says that the pattern $\Tree[\O]^{\op,\natural}$ is saturated. We are therefore reduced to showing that a map $[m;s] \to [n;t]$ is inert if for any inert map from an elementary $[i;u] \intarrow [m;s]$, the composite $[i;u] \intarrow [m;s] \to [n;t]$ is inert.
    Since $\Delta$ satisfies this property and a map in $\Tree[\cO]$ is inert precisely if it lies over an inert in $\Delta$, it suffices to show that any inert $\phi \colon [i] \intarrow [m]$ admits a lift with target $[m;s]$.
    This follows from \cref{cor:explicit description of mapping space} since $(\cO_0^\op)_{/s_{\phi(i)}} \simeq \cO^\elem_{s_{\phi(i)}/}$ connected (hence non-empty) by assumption.
\end{proof}

\begin{remark}
Let $\O$ be a robust pattern. Instead of considering simplices decorated by elements of $\O$ as we do when working with the category $\Tree[\O]$, we could consider non-layered trees, i.e., elements of the dendroidal category $\Omega$, decorated appropriately by atomic objects of $\O$. This would lead to a category $\Tree_{nl}[\O]$ of non-layered trees that we conjecture is a robust pattern whenever $\O$ is.
\end{remark}

\section{Necessity of the Conduch\'e criterion}\label{section:necessity}
This section is dedicated to the proof of \cref{thmB}, the converse of \cref{thmA}, for so-called \emph{robust} algebraic patterns:

\begin{theorem}\label{thm:main-thm-expo-CSeg-converse}
    Let $\O$ be a robust algebraic pattern. Then a map $p : X \to Y$ in $\CSeg(\Tree[\O])$ satisfies \cref{conditionCC} of \cref{thm:main-thm-expo-CSeg} if $p$ is exponentiable.
\end{theorem}

The converse of \cref{thm:main-thm-expo-factorization-Algd} then automatically follows from the dictionary of \cref{lemma:dictionary cc conditions}.

Throughout this section, we will fix a robust algebraic pattern $\O$.

\subsection{Relative Segality} 
By \cref{subsection:functoriality of algebrad cats}, $\pi_0$ induces a functor
$$
    q := \Tree[\pi_0] : \Tree[\O] \to \Tree[\Span(\F)^\flat]
$$
between tree categories. If $\left<n;t\right>$ is a forest, then we have a projection 
$$
    [n;t] \to q^*[n; \pi_0(t)]
$$
obtained by applying $i^* : \PSh(\Forest[\O]) \to \PSh(\Tree[\O])$ to the canonical map $\left<n;t\right> \to \Forest[\pi_0]^*\left<n;\pi_0(t)\right>$. The goal of this subsection is to show the following:

\begin{proposition}\label{prop:relative fibrancy over tree}
    The projection $[n;t] \to q^*[n;\pi_0(t)]$ is a complete Segal fibration for every forest $\left<n;t\right>$.
\end{proposition}

Recall from \cref{subsection:functoriality of algebrad cats} that we have a natural map $q_!(X\boxtimes_{[n]} [n;t]) \to X \boxtimes_{[n]} [n;\pi_0(t)]$ which is an equivalence if $q$ is strong Segal. In general, $q$ fails to be strong Segal. Still, the following holds:

\begin{lemma}\label{lem:q comparison for discrete}
    If $Y$ is a discrete $\Tree[\Span(\F)]$-space, then the induced map 
    $$\Hom_{\PSh(\Tree[\Span(\F)])}(X \boxtimes_{[n]} [n;\pi_0(t)], Y) \to \Hom_{\PSh(\Tree[\O])}(X \boxtimes_{[n]} [n; t], q^*Y)$$
    is an equivalence for every forest $\left<n;t\right>$ and every simplicial space $X \to [n]$.
\end{lemma}
\begin{proof}
    By \cref{lemma:box with forest is forest}, it suffices to handle the case that $X = [n]$. By the decomposition formula of \cref{rem:decomposition by atoms formula}, we may moreover assume that $(\O^{\el}_n)_{t/}$ is connected. The map $q_![n;t] \to [n;\pi_0(t)]$ is computed by $\colim_{s\in (\O^{\el}_n)_{t/}} [n;\pi_0(s)] \to [n;\pi_0(t)]$, as $\O$ is sound.
    Now, each map $[n;\pi_0(s)] \to [n;\pi_0(t)]$ is an equivalence as  it lies over $\id_{[n]}$ and $\pi_0(s)$ and $\pi_0(t)$ are equivalent to $\underline{1}$.
    Thus $q_![n;t] \to [n;\pi_0(t)]$ is equivalent to the projection $|(\O^{\el}_n)_{t/}| \times [n;\pi_0(t)] \to [n;\pi_0(t)]$. The map induced on hom-spaces from the statement of the lemma is thus given by 
    $$
    \Hom_{\PSh(\Tree[\Span(\F)])}([n;\pi_0(t)], Y) \to \Hom_\Spc(|\O^{\el}_{t/}|, \Hom_{\PSh(\Tree[\Span(\F)])}([n;\pi_0(t)], Y)).
    $$
    This is an equivalence since $\Hom_{\PSh(\Tree[\Span(\F)])}([n;\pi_0(t)], Y)$ is a set and $|(\O^{\el}_n)_{t/}|$ is connected.
\end{proof}

\begin{lemma}
\label{lemma:fiber of pi}
Let $\langle n;t\rangle$ and $\langle m;s\rangle$ be forests. Suppose that we have a map $f : \left<m;\pi_0(s)\right>\to \left<n;\pi_0(t)\right>$ of forests over $\phi :[m]\to [n]$ determined by a commutative diagram of pullback squares
\[
    \alpha = \begin{tikzcd}
        \pi_0(t_{\phi(0)}) \arrow[r] & \dotsb \arrow[r] & \pi_0(t_{\phi(m)}) \\
        \pi_0(s_0) \arrow[r]\arrow[u, "\alpha_0"] & \dotsb \arrow[r] & \pi_0(s_m),\arrow[u, "\alpha_m"']
    \end{tikzcd}
\] cf.\ \cref{exa:span-forests}. 
Then there is a pullback square
\[\begin{tikzcd}
	{\prod_{i\in\pi_0(s)}\Hom_{\O_m}((\phi^*t)^{\alpha_m(i)},s^i)} \arrow[r]\arrow[d] & {\Hom_{\PSh(\Tree[\O])}([m;s],[n;t])} \arrow[d]\\
	{\{\bar{f}\}} \arrow[r]&  {\Hom_{\PSh(\Tree[\Span(\F)])}([m;\pi_0(s)],[n;\pi_0(t)])},
\end{tikzcd}\]
where the right vertical arrow is induced by $[n;t] \to q^*[n;\pi_0(t)]$ (see \cref{lem:q comparison for discrete}), and $\bar{f}$ is the map $[m;\pi_0(s)] \to [n;\pi_0(t)]$ induced by $f$.
\end{lemma}

\begin{proof}
Suppose first that $[m;s]$ is a tree. The right vertical map of the square above may be identified with the top right vertical map induced by $\Forest[\pi_0]$ in the commutative diagram
\[\begin{tikzcd}
	{\Hom_{\O_m}(\phi^*t,s)} & {\Hom_{\Forest[\O]}(\left<m;s\right>,\left<n;t\right>)} \\
	{\Hom_{\Span(\F)_m}(\phi^*(\pi_0(t)),\pi_0(s))} & {\Hom_{\Forest[\Span(\F)]}(\left<m;\pi_0(s)\right>,\left<n;\pi_0(t)\right>)} \\
	{\{\phi\}} & {\Hom_{\Delta}([m],[n])}.
	\arrow[from=1-1, to=1-2]
	\arrow[from=1-1, to=2-1]
	\arrow[from=1-2, to=2-2]
	\arrow[from=2-1, to=2-2]
	\arrow[from=2-1, to=3-1]
	\arrow[from=2-2, to=3-2]
	\arrow[from=3-1, to=3-2]
\end{tikzcd}\]
The total and lower squares are pullback squares, so the top one is a pullback square as well by right cancellation. Pulling back this upper square along the morphism $\alpha_m: \pi_0(t_{\phi(m)}) \intarrow \pi_0(s_m) \simeq \underline{1}$ leads to the result by \cref{lem:pi0 levelwise cocart}.

Suppose now that $\langle m;s\rangle$ is a forest. Then the canonical horizontal arrows in the square
\[\begin{tikzcd}
     {\Hom_{\PSh(\Tree[\O])}([m;s],[n;t])} \arrow[r]\arrow[d] &  \prod_{i \in \pi_0(s)} \lim_{e \in (\O^{\el}_m)_{s^i/}} {\Hom_{\PSh(\Tree[\O])}([m;e],[n;t])} \arrow[d] \\
	{\Hom_{\PSh(\Tree[\O])}([m;s],q^*[n;\pi_0(t)])} \arrow[r] & \prod_{i \in \pi_0(s)} \lim_{e \in (\O^{\el}_m)_{s^i/}} {\Hom_{\PSh(\Tree[\O])}([m;e],q^*[n;\pi_0(t)])} 
\end{tikzcd}\]
are equivalences by \cref{rem:forest decomposition sound} and \cref{rem:decomposition by atoms formula} applied to $X=[n]$. It follows from the above case that the fiber of the left vertical arrow above $f$ is computed by the product
$
\prod_{i \in \pi_0(s)} \lim_{e \in (\O^{\el}_m)_{s^i/}} \Hom_{\O_m}((\phi^*t)^{\alpha_m(i)},e).
$
As $\O_m$ is saturated by \cref{prop:robust-On-is-saturated}, the canonical map 
$$
 \textstyle \prod_{i \in \pi_0(s)} \Hom((\phi^*t)^{\alpha_m(i)},s^i) \to \prod_{i \in \pi_0(s)} \lim_{e \in (\O^{\el}_m)_{s^i/}} \Hom_{\O_m}((\phi^*t)^{\alpha_m(i)},e)
$$
is an equivalence.
\end{proof}

Moreover, we will need the following observation to finish the proof of the main result \cref{prop:relative fibrancy over tree} of this subsection.

\begin{lemma}
\label{lemma:topos}
    Let $f_i:A_i\to B_i$ be a cartesian family of complete Segal fibrations between $\Tree[\O]$-spaces indexed by a category $I$.
    Then the induced morphism $\colim_I f_i:\colim_I A_i\to \colim_I B_i$ is a complete Segal fibration.
\end{lemma}

\begin{proof}
    As all generating complete Segal extensions have a representable codomain, it is sufficient to check that for every morphism $g:[n;t]\to \colim_IB_i$, the pullback $g^*\colim_If_i\to [n;t]$ is a complete Segal fibration. Let $g$ be such a morphism. Remark that there necessarily exists a (non-unique) $j$ in $I$ such that $g$ factors through $B_j\to \colim_IB_i$. As $\PSh(\Tree[\O])$ is a topos, and as the family of diagrams is cartesian, by \cite[Theorem 6.1.0.6]{HTT} we have a diagram of cartesian squares:
\[\begin{tikzcd}
	{g^*\colim_If_i}& {A_j} & {\colim_{I}A_i} \\
	{[n;t]} & {B_j} & {\colim_{I}B_i}
	\arrow[from=1-1, to=1-2]
	\arrow[from=1-1, to=2-1]
	\arrow[from=1-2, to=1-3]
	\arrow[from=1-2, to=2-2]
	\arrow["{\colim_If_i}", from=1-3, to=2-3]
	\arrow[from=2-1, to=2-2]
	\arrow[from=2-2, to=2-3]
\end{tikzcd}\]
As the assumption implies that the middle map is a complete Segal fibration, so is the left one, which concludes the proof.
\end{proof}

\begin{proof}[Proof of \cref{prop:relative fibrancy over tree}]
We need to show that the map 
\begin{equation}\label{eq:relative fibrancy over tree}
    \Hom_{/q^*[n;\pi_0(t)]}([m;s], [n;t]) \to \Hom_{/q^*[n;\pi_0(t)]}(\Gamma[m;s], [n;t]) 
\end{equation}
is an equivalence 
for every map $f : [m;s] \to q^*[n;\pi_0(t)]$ so that $[m;s]$ is a tree. As $[m;s]$ is a tree, this comes from a map $\left<m;\pi_0(s)\right> \to \left<n;\pi_0(t)\right>$ between forests by adjunction, that we will denote by $f$ as well. Let $\phi : [m] \to [n]$ be its underlying map in $\Delta$ and $\alpha \colon \pi_0(\phi^*(t)) \to \pi_0(s)$ the underlying map in $\Span(\F)_m$.
Note that the commutative square 
\[
    \begin{tikzcd}
        {[m;\phi^*t]} \arrow[r]\arrow[d] & {[n;t]} \arrow[d] \\
        q^*{[m;\pi_0(\phi^*t)]} \arrow[r] & q^*{[n; \pi_0(t)]}
    \end{tikzcd}
\]
is cartesian by pullback pasting. Therefore, we may reduce to the case that $[n] = [m]$ and $\phi = \id_{[m]}$ in what follows. Now, by the decomposition formula of \cref{rem:decomposition by atoms formula} (applied to both patterns $\O$ and $\Span(\F)$), we can identify the projection $[m;t] \to q^*[m;\pi_0(t)]$ with the coproduct of projections $\coprod_{i \in \pi_0(t_m)}[m;t^i] \to \coprod_{i \in \pi_0(t_m)}q^*[m; \pi_0(t)^i]$. Thus by \cref{lemma:topos}, we may further reduce to the case that $[n]=[m]$, $\phi=\id_{[m]}$ and $\pi_0(t_m) \simeq \underline{1}$. In particular, the component $\alpha_m$ is an equivalence. 

By \cref{lemma:fiber of pi}, the map \cref{eq:relative fibrancy over tree} is then computed as the gap map in the top square of the commutative diagram
\begin{equation}\label{eq:rectangle-relative-fibrancy}
    \begin{tikzcd}
        \Hom_{\O_m}(t,s) \arrow[r]\arrow[d] & \Hom_{\O_1}(t_{\leq m-1}, s_{\leq m-1}) \arrow[d] \\
        {\prod_{i\in \pi_0(s_{m-1})}\Hom_{\O_{m-1}}(t_{\leq m-1}^{\alpha_{m-1}(i)},s_{\leq m-1}^i)}\arrow[d]\arrow[r] & {\prod_{i\in \pi_0(s_{m-1})}\Hom_{\O_0}(t_{m-1}^{\alpha_{m-1}(i)},s_{ m-1}^i)}\arrow[d] \\
	{\prod_{i\in \pi_0(s_{m-1})}\Hom_{\O_{m-1}}(t_{\leq m-1},s_{\leq m-1}^i)} \arrow[r]& {\prod_{i\in \pi_0(s_{m-1})}\Hom_{\O_0}(t_{m-1},s_{ m-1}^i).}
    \end{tikzcd}
\end{equation}
As $\O_{m-1}$ is saturated by \cref{prop:robust-On-is-saturated}, the bottom map is equivalent to 
$$\Hom_{\O_{m-1}}(t_{\leq m-1},s_{\leq m-1})\to \Hom_{\O_{0}}(t_{m-1},s_{ m-1})$$ and so the outer rectangle is cartesian as $\O_\bullet$ is a double category.
By pullback pasting, it therefore suffices to show that the bottom square of \cref{eq:rectangle-relative-fibrancy} is cartesian.
We will fix $i \in \pi_0(s_{m-1})$ and show this for the $i$-th factor.
Note that we have a cube
\[\begin{tikzcd}
	& {(\cO_{m-1})_{t^{\alpha_{m-1}(i)}_{\leq m-1}/}} && {(\cO_{0})_{t^{\alpha_{m-1}(i)}_{m-1}/}} \\
	{(\cO_{m-1})_{t_{\leq m-1}/}} && {(\cO_{0})_{t_{m-1}/}} \\
	& {\cO_{m-1}} && {\cO_0} \\
	{\cO_{m-1}} && {\cO_0}
	\arrow["\sim", from=1-2, to=1-4]
	\arrow[from=1-2, to=3-2]
	\arrow[from=1-4, to=3-4]
	\arrow[from=2-1, to=1-2]
	\arrow[from=2-1, to=4-1]
	\arrow[from=2-3, to=1-4]
	\arrow[from=3-2, to=3-4]
	\arrow["\sim", from=4-1, to=3-2]
	\arrow[from=4-1, to=4-3]
	\arrow["\sim", from=4-3, to=3-4]
	\arrow[from=2-3, to=4-3, crossing over]
	\arrow["\sim"{pos=0.7}, from=2-1, to=2-3, crossing over]
\end{tikzcd}\]
Both the bottom and top square are cartesian. Taking fibers of the vertical maps over $s^i_{\leq m-1}$, it follows that the $i$-th factor of the bottom square of \cref{eq:rectangle-relative-fibrancy} is cartesian.
We conclude that the top square of \cref{eq:rectangle-relative-fibrancy} is cartesian.

The proof of the unique left lifting property against generating completeness extension is shown in a similar way using the fact that the double category $\O_\bullet$ is complete.
\end{proof}

\begin{corollary}
\label{cor:fibrant forest}
Let $\O$ be an atomically robust pattern. Then the $\Tree[\O]$-Segal space $[n;t]$ is complete Segal for every forest $\langle n;t\rangle$.

\end{corollary}

\begin{proof}
    As $\O$ is atomically robust, $[n;\pi_0(t)]$ corresponds to the tree $\coprod_{\pi_0(t_n)}L_n$ where $L_n$ is the tree $[n;\underline{1}=\underline{1}=\ldots=\underline{1}]$. The unary tree $L_n$ is precisely the image of $[n]$ under the functor $\iota : \CSeg(\Delta) \to \CSeg(\Tree[\Span(\F)])$ of \cref{sec:Underlying-graph}. Thus $[n;\pi_0(t)]$ is complete Segal. As $[n;\pi_0(t)]$ is also discrete, $q^*[n;\pi_0(t)]$ is a complete Segal $\Tree[\O]$-space by \cref{lem:q comparison for discrete}. The desired conclusion now follows from \cref{prop:relative fibrancy over tree}.
\end{proof}

\subsection{The grafting construction}\label{subsec:grafting}
Recall that we are working with a robust pattern $\O$. The key ingredient for our proof of \cref{thmB} is the so-called \textit{grafting construction}:

\begin{construction}\label{cons:grafting}
Let $u:u_0\actarrow u_1$ and $t:t_0\actarrow u_0^i$ be two active morphisms with $i\in \pi_0(u_0)$. Then there exists a unique active morphism $\underline{t}\actarrow u_0$ whose projection by the functor $\O^{\act}_{/u_0}\to \O^{\act}_{/u_0^j}$ is given by 
\[
    \begin{cases}
        t : t_0 \actarrow u_0^i & \text{if $i=j$}, \\
        \id : u_0^j \to u_0^j & \text{otherwise}.
    \end{cases}
\]
We then obtain a forest 
$$
\textstyle \left<2; u \star^i t\right> := \left<2; \underline{t}\actarrow u_0 \actarrow u_1\right>.
$$
For $[k] \in \Delta$, we will consider the $\Tree[\O]$-space defined by the pushout square
\[\begin{tikzcd}
	{\coprod_{j\neq i\in \pi_0(u_0)}  (T_k[1]\times_{[2]}[1])\boxtimes_{[1]}[1;u_0^j=u_0^j]} & {T_k[1] \boxtimes_{[2]}[2;u\star^i t]} \\
	{\coprod_{j\neq i\in \pi_0(u_0)} [0;u_0^j]} & {T_k[u\circ^i t]}.
	\arrow[from=1-1, to=1-2]
	\arrow[from=1-1, to=2-1]
	\arrow[from=1-2, to=2-2]
	\arrow[from=2-1, to=2-2]
\end{tikzcd}\]
We then set 
$$[u\circ^i t]:=T_0[u\circ^i t], \quad \quad \Gamma[u\circ^i t]:=[1;t^i] \cup_{[1;u_0^i]} [1;u].$$
\end{construction}

The goal of this subsection is to show the following:

\begin{proposition}
    \label{prop:fibrancy of the grafting construction}
    Let $u, t$ and $\gamma$ be as in \cref{cons:grafting}.
    Then the map $\Gamma[u \circ^i t] \to [u \circ^i t]$ is a Segal extension and $[u\circ^i t]$ is a complete Segal space.
    Moreover, for any forest $\langle n;s \rangle$ with $n \leq 1$, the $\Tree[\cO]$-space $[n;s]$ is complete Segal.
\end{proposition}

The name of the grafting construction comes from the fact that it recovers the grafting construction for the prototypical case of operads. This is explained during the demonstration of the following lemma:

\begin{lemma}\label{lem:fibrancy of the grafting construction operad case}
    If $\O$ is the pattern $\Span(\F)^{\flat}$, then the $\Tree[\Span(\F)^\flat]$-space $[u\circ^i t]$ is complete Segal, or equivalently, is an operad.
    Moreover, for any forest $\langle n;t \rangle$ with $n \leq 1$, the $\Tree[\Span(\F)^\flat]$-space $[n;t]$ is complete Segal.
\end{lemma}

\begin{proof}
    Note that we have a composite $\Psi \colon \Tree[\Span(\F)^\flat] \to \Omega \hookrightarrow \mathrm{Op}$ sending a tree to the free operad generated by this tree; here $\Omega$ is the dendroidal category of Moerdijk--Weiss \cite{MoerdijkWeiss}.
    Given an operad $P$, its \emph{nerve} $N(P) \coloneqq \Hom(\Psi(-),P)$ defines a complete Segal $\Tree[\Span(\F)^\flat]$-space by \cite[Theorem 1.1]{ChuHaugsengea2018TwoModelsHomotopy}.
    We will prove the first part of the lemma by showing that $[u\circ^i t]$ is of the form $N(P)$ for some operad $P$.
    Suppose that $|u_0| = m$ and $|t_0| = n$; then we will show that $[u\circ^i t] = N(P)$ where $P$ is the operad freely generated by the tree
    \[\scalebox{0.8}
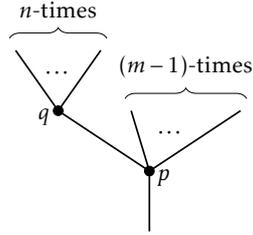
{\begin{tikzpicture}
        \draw[fill] (0,1) circle [radius=0.08];
        \draw[fill] (-1.5,2) circle [radius=0.08];
        \node[right,yshift=-3pt] at (0,1) {$p$};
        \node[left,yshift=-3pt] at (-1.5,2) {$q$};
        \node at (-1.5,2.6) {$\cdots$};
        \node at (0.35,1.6) {$\cdots$};
        
        \draw[thick] (0,0) -- (0,1);
        \draw[thick] (0,1) -- (-1.5,2);
        \draw[thick] (0,1) -- (-0.3,2);
        \draw[thick] (0,1) -- (1.5,2);
        \draw[thick] (-1.5,2) -- (-2.2,3);
        \draw[thick] (-1.5,2) -- (-0.8,3);
        \draw[decorate,decoration={brace,amplitude=5pt}]
            (-2.3,3.2) -- (-0.7,3.2)
            node[midway,above=6pt] {$n\text{-times}$};
        \draw[decorate,decoration={brace,amplitude=5pt}]
            (-0.4,2.2) -- (1.6,2.2)
            node[midway,above=6pt] {$(m-1)\text{-times}$};
    \end{tikzpicture}}\]
    where $p$ has $m$ ingoing edges and $q$ has $n$ ingoing edges. To be precise, the edges of this tree are the colors of $P$, while the (non-identity) operations of $P$ are the vertices $p$ and $q$ and their composite $p \circ^i q$.
    Note that we have a canonical comparison map $[2;u \star^i t] \to N(P)$ given by the operad map
    \[\scalebox{0.8}
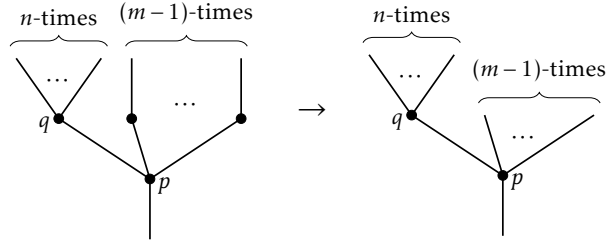
{\begin{tikzpicture}[baseline={(current bounding box.center)}]
        \draw[fill] (0,1) circle [radius=0.08];
        \draw[fill] (-1.5,2) circle [radius=0.08];
        \draw[fill] (-0.3,2) circle [radius=0.08];
        \draw[fill] (1.5,2) circle [radius=0.08];
        \node[right,yshift=-3pt] at (0,1) {$p$};
        \node[left,yshift=-3pt] at (-1.5,2) {$q$};
        \node at (-1.5,2.6) {$\cdots$};
        \node at (0.6,2.2) {$\cdots$};
        
        \draw[thick] (0,0) -- (0,1);
        \draw[thick] (0,1) -- (-1.5,2);
        \draw[thick] (0,1) -- (-0.3,2);
        \draw[thick] (0,1) -- (1.5,2);
        \draw[thick] (-0.3,2) -- (-0.3,3);
        \draw[thick] (1.5,2) -- (1.5,3);
        \draw[thick] (-1.5,2) -- (-2.2,3);
        \draw[thick] (-1.5,2) -- (-0.8,3);
        \draw[decorate,decoration={brace,amplitude=5pt}]
            (-2.3,3.2) -- (-0.7,3.2)
            node[midway,above=6pt] {$n\text{-times}$};
        \draw[decorate,decoration={brace,amplitude=5pt}]
            (-0.4,3.2) -- (1.6,3.2)
            node[midway,above=6pt] {$(m-1)\text{-times}$};
    \end{tikzpicture}}
    \quad \to \quad 
    \scalebox{0.8}{\begin{tikzpicture}[baseline={(current bounding box.center)}]
        \draw[fill] (0,1) circle [radius=0.08];
        \draw[fill] (-1.5,2) circle [radius=0.08];
        \node[right,yshift=-3pt] at (0,1) {$p$};
        \node[left,yshift=-3pt] at (-1.5,2) {$q$};
        \node at (-1.5,2.6) {$\cdots$};
        \node at (0.35,1.6) {$\cdots$};
        
        \draw[thick] (0,0) -- (0,1);
        \draw[thick] (0,1) -- (-1.5,2);
        \draw[thick] (0,1) -- (-0.3,2);
        \draw[thick] (0,1) -- (1.5,2);
        \draw[thick] (-1.5,2) -- (-2.2,3);
        \draw[thick] (-1.5,2) -- (-0.8,3);
        \draw[decorate,decoration={brace,amplitude=5pt}]
            (-2.3,3.2) -- (-0.7,3.2)
            node[midway,above=6pt] {$n\text{-times}$};
        \draw[decorate,decoration={brace,amplitude=5pt}]
            (-0.4,2.2) -- (1.6,2.2)
            node[midway,above=6pt] {$(m-1)\text{-times}$};
    \end{tikzpicture}}\]
    that sends $p$ and $q$ to themselves, and all other vertices to the identity operation of the corresponding color.
    This map factors through the pushout $[u\circ^i t]$ by construction.
    We will show that this map $\chi \colon [u\circ^i t] \to N(P)$ is an equivalence.
    This comes down to showing that for any layered tree $[n;s] \in \Omega[\cO]$ and any operad map $f \colon \Psi([n;t]) \to P$, there is a unique map $\tilde{f} \colon [n;s] \to [u\circ^i t]$ such that the composite $\chi \circ \tilde{f} \colon [n;s] \to N(P)$ is the adjunct of $f$.
    
    First consider the case $m = 1$.
    In this case $[u\circ^i t] = [2; u_0 \actarrow \underline{1} \actarrow \underline{1}]$, so the claim reduces to showing that any map from (the operad associated with) a layered tree to $P$ automatically respects the layering. This is clear.
    Now assume $n=1$ and $m \neq 1$.
    If a map $f \colon [n;s] \to N(P)$ exists, then either $[n;s]$ is a linear tree, meaning that $s_0 = s_1 = \cdots = s_n = \underline{1}$, or the tree $[n;s]$ contains precisely one vertex with $m$ ingoing edges and all other vertices have exactly one ingoing edge.
    In the first case, there can potentially be multiple lifts of $f$ to a map $[n;s] \to [2; u \star^i t]$, but the pushout defining $[u\circ^i t]$ will always identify these with each other.
    In the second case, the vertex with $m$ ingoing edges needs to be sent to $p$. Such a map always lifts uniquely to a map of layered trees $[n;s] \to [2; u \star^i t]$.
    Now suppose $m,n \neq 1$ and let $f \colon [n;s] \to N(P)$ be given.
    If $[n;s]$ is a linear tree, then as in the previous case a lift of $f$ to $[2; u \star^i t]$ exists, but need not be unique if $[n;s]$ hits one of the leaves attached to $p$.
    However, the pushout defining $[2; u \star^i t]$ identifies all these lifts.
    On the other hand, if $[n;s]$ is not a linear tree, then any lift of $f$ to $[2;u \star^i t]$ and hence $[u\circ^i t]$ is necessarily unique.
    We conclude that $[u\circ^i t] \to N(P)$ is an equivalence.
    
    For the second part of the lemma, a similar (but simpler) argument shows that $[n;t]$ is complete Segal for any tree of length $n \leq 1$.
    If $\langle n;t \rangle$ is a forest instead, then the claim follows since $[n;t]$ is a coproduct of trees of length $n$ in $\PSh(\Tree[\cO])$, and the operadic nerve $\mathrm{Op} \to \PSh(\Omega) \to \PSh(\Tree[\cO])$ preserves coproducts.
\end{proof}

\begin{proof}[Proof of \cref{prop:fibrancy of the grafting construction}]
By construction, we have a pushout square 
\[\begin{tikzcd}
	{\Gamma[2;u\star^i t]} & {[2;u\star^i t]} \\
	{\Gamma[u\circ^i t]} & {[u\circ^i t]}
	\arrow[from=1-1, to=1-2]
	\arrow[from=1-1, to=2-1]
	\arrow[from=1-2, to=2-2]
	\arrow[from=2-1, to=2-2]
\end{tikzcd}\]
so that the bottom arrow is a complete Segal extension as well.

To show that $[u\circ^it]$ is complete Segal, we first note that $[u\circ^it]\to q^*[\pi_0(u)\circ^i \pi_0(t)]$ is a complete Segal fibration.
Namely, we have pullback squares
\[\begin{tikzcd}[column sep=small]
	{\coprod_{j\neq i\in \pi_0(u_0)} [0;u_0^j]} & {\coprod_{j\neq i\in \pi_0(u_0)} [1;u_0^j=u_0^j]} & {[2;u\star^i t]} \\
	{\coprod_{j\neq i\in \pi_0(u_0)}q^*[0;\pi_0(u)_0^j]} & {\coprod_{j\neq i\in \pi_0(u_0)}q^*[1;\pi_0(u)_0^j=\pi_0(u)_0^j]} & {q^*[2;\pi_0(u)\star^i \pi_0(t)]}.
	\arrow[from=1-1, to=2-1]
	\arrow[from=1-2, to=1-1]
	\arrow[from=1-2, to=1-3]
	\arrow[from=1-2, to=2-2]
	\arrow[from=1-3, to=2-3]
	\arrow[from=2-2, to=2-1]
	\arrow[from=2-2, to=2-3]
\end{tikzcd}\]
By construction, the map $[u\circ^i t]\to q^*[\pi_0(u)\circ^i \pi_0(t)]$ is obtained by taking pushouts horizontally in this diagram. Thus it is a complete Segal fibration by
\cref{prop:relative fibrancy over tree} and \cref{lemma:topos}. 
As $q^*$ preserves discrete complete Segal objects by \cref{lem:q comparison for discrete}, the prototypical case of \cref{lem:fibrancy of the grafting construction operad case} implies that $q^*[\pi_0(u)\circ^i \pi_0(t)]$ is complete Segal. We conclude that $[u\circ^i t]$ is complete Segal.

The second part of the proposition follows immediately by combining \cref{prop:relative fibrancy over tree,lem:fibrancy of the grafting construction operad case}.
\end{proof}

\subsection{Another explicit replacement}
Suppose that $X \to [u\circ^i t]$ is a complete Segal fibration. Then we will prove \cref{thm:main-thm-expo-CSeg-converse} by constructing an explicit replacement for  $X \times_{[u\circ^i t]} \Gamma[u\circ^i t] \to [u\circ^i t]$ using the $Q$ functor of \cref{cons:Q}.

\begin{construction}
Let $X \in \PSh(\Tree[\O])_{/[u\circ^it]}$. We will write $X_0 := X \times_{[u\circ^i t]} [2;u\star^i t]$. Using \cref{cons:Q}, we can construct the following commutative diagram:
\[\begin{tikzcd}[column sep=tiny,row sep=small]
	X_3 && X_1 && {QX_0} & \\
	& X_4 && X_2 && RX \\
	{\coprod_{j\neq i\in \pi_0(u_0)}[1;u_0^j=u_0^j]} && {\Gamma [2;u\star^i t]} && {[2;u\star^i t]} \\
	& {\coprod_{j\neq i\in \pi_0(u_0)}[0;u_0^j]} && {\Gamma[u\circ^i t]} && {[u\circ^i t]}
	\arrow[from=1-1, to=1-3]
	\arrow[from=1-1, to=2-2]
	\arrow[from=1-1, to=3-1]
	\arrow[from=1-3, to=1-5]
	\arrow[from=1-3, to=2-4]
	\arrow[from=1-3, to=3-3]
	\arrow[from=1-5, to=2-6]
	\arrow[from=1-5, to=3-5]
	\arrow[from=2-2, to=2-4]
	\arrow[from=2-2, to=4-2]
	\arrow[from=2-4, to=2-6]
	\arrow[from=2-4, to=4-4]
	\arrow[from=2-6, to=4-6]
	\arrow[from=3-1, to=3-3]
	\arrow[from=3-1, to=4-2]
	\arrow[from=3-3, to=3-5]
	\arrow[from=3-3, to=4-4]
	\arrow[from=3-5, to=4-6]
	\arrow[from=4-2, to=4-4]
	\arrow[from=4-4, to=4-6]
 \end{tikzcd}\]
where the $X_i$'s are pullbacks of $X$, and $RX \coloneqq X_2 \cup_{X_1} QX_0$.
The back square involving the map $X_1 \to QX_0$ is cartesian as remarked in \cref{cons:Q}. 
We note that there is a canonical factorization 
$$
X_2 = X \times_{[u\circ^it]} \Gamma[u\circ^it] \to RX \to X.
$$
\end{construction}

\begin{proposition}
\label{lemma:fibrancy of RX}
The map $X \times_{[u\circ^it]} \Gamma[u\circ^it] \to RX$ is a Segal extension for all maps $p : X \to [u\circ^i t]$, and $RX \to [u\circ^it]$ is a complete Segal fibration if $p$ is a complete Segal fibration.
\end{proposition}
\begin{proof}
    The first assertion follows directly from \cref{prop:Q is a replacement}. If $p$ is a complete Segal fibration, then \cref{prop:Q is a replacement} implies that the morphism $QX_0\to [2;u\star^i t]$ is a complete Segal fibration. The result then follows from \cref{lemma:topos} as $RX\to [u\circ^i t]$ is the colimit of a cartesian diagram of complete Segal fibrations.
\end{proof}

The next goal is to provide a computation of $R$ that is similar to the computation of $Q$ established in \cref{prop:computation Q}.

\begin{lemma}\label{lemma:a cartesian fib}
Let $X \to [2;u\star^i t]$ be a complete Segal fibration. Then the induced map
$$
\Hom_{/[2;u\star^i t]}(T_\bullet[1]\boxtimes_{[2]}[2;u\star^i t],X) \to \prod_{i\neq j}\Hom_{/[2;u\star^i t]}([1+\bullet]\boxtimes_{[1]}[1;u_0^j =u_0^j],X)
$$
is map of complete Segal spaces, and can thus be viewed as a functor of categories. As such, it is a cartesian fibration.
\end{lemma}
\begin{proof}
 It follows from \cref{prop:simplicial segal extension} that this is indeed a map between complete Segal spaces.
Note that $[1+\bullet] \simeq T_\bullet[1] \times_{[2]} [1]$, where the pullback is taken along $d_2 \colon [1] \to [2]$.
We get an induced map 
$$\Hom_{/[2;u\star^i t]}(T_\bullet[1]\boxtimes_{[2]}[2;u\star^i t],X)\xrightarrow{\alpha} \Hom_{/[2;u\star^i t]}((T_\bullet[1] \times_{[2]} [1])\boxtimes_{[2]}[2;u\star^i t],X).$$
We claim that this is a right fibration.
To this end, we have to show that the square 
\[\begin{tikzcd}
	\Hom_{/[2;u\star^i t]}({T_1[1]\boxtimes_{[2]}[2;u\star^it]}, X)& \Hom_{/[2;u\star^i t]}({T_0[1]\boxtimes_{[2]}[2;u\star^it]},X)\\
	\Hom_{/[2;u\star^i t]}({(T_1[1]\times_{[2]}[1])\boxtimes_{[2]}[2;u\star^it]},X) & \Hom_{/[2;u\star^i t]}({(T_0[1]\times_{[2]}[1])\boxtimes_{[2]}[2;u\star^it]} , X)
	\arrow[from=1-1, to=1-2]
	\arrow[from=1-1, to=2-1]
	\arrow[from=1-2, to=2-2]
	\arrow[from=2-1, to=2-2]
\end{tikzcd}\]
induced by $\{1\} \to [1]$ is a pullback square.
But this follows from \cref{prop:simplicial segal extension} and the observation that the following square is a pushout in the category of complete Segal spaces over $[2]$:
\[\begin{tikzcd}
	 {T_0[1]\times_{[2]}[1]} & {T_0[1]} \\
	{T_1[1]\times_{[2]}[1]} & {T_1[1]}
	\arrow[from=1-1, to=1-2]
	\arrow[from=1-1, to=2-1]
	\arrow[from=1-2, to=2-2]
	\arrow[from=2-1, to=2-2]
\end{tikzcd}
\quad \simeq \quad
\begin{tikzcd}
	{[1]} & {[2]} \\
	{[2]} & {[3]}.
	\arrow[from=1-1, to=1-2, "d_2"]
	\arrow[from=1-1, to=2-1, "d_1"']
	\arrow[from=1-2, to=2-2, "d_1"]
	\arrow[from=2-1, to=2-2, "d_3"]
\end{tikzcd}
\]
Now, using the decomposition formula of \cref{rem:decomposition by atoms formula}, the following square is a pullback
\[\begin{tikzcd}
	{\Hom_{/[2;u\star^i t]}([1+\bullet]\boxtimes_{[2]}[2;u\star^i t],X)} & {\Hom_{/[2;u\star^i t]}([1+\bullet]\boxtimes_{[1]}[1;t],X)} \\
	{\prod_{i\neq j}\Hom_{/[2;u\star^i t]}([1+\bullet]\boxtimes_{[1]}[1;u_0^j =u_0^j],X)} & \ast.
	\arrow[from=1-1, to=1-2]
	\arrow["\beta"', from=1-1, to=2-1]
	\arrow[from=1-2, to=2-2]
	\arrow[from=2-1, to=2-2]
\end{tikzcd}\]
As the right vertical morphism is a  cartesian fibration, so is the left vertical one. This implies that $\beta\alpha$ is a cartesian fibration, as desired.
\end{proof}

\begin{lemma}
\label{lemma:an initiality result}
Suppose that $X \to [u\circ^i t]$ is a complete Segal fibration. 
Then the morphism 
$$\colim_{[k]\in \Delta^{\op}}\map_{/[u\circ^i t]}(T_k[u\circ^i t],X)\to \colim_{[k]\in \Delta^{\op}}\map_{/[u\circ^i t]}(T_k[1]\boxtimes_{[2]}[2;u\star^i t],X)$$
is an equivalence.
\end{lemma}

\begin{proof}
By construction, we have a pullback square
\[
\begin{tikzcd}[column sep=tiny,row sep=small]
	{\Hom_{/[u\circ^i t]}(T_\bullet[u\circ^i t],X)}\arrow[r]\arrow[d] & {\Hom_{/[u\circ^i t]}(T_\bullet[1]\boxtimes_{[2]}[2;u\star^i t],X)}\arrow[d] \\
	{\prod_{i\neq j}\Hom_{/[u\circ^i t]}([0;u_0^j],X)}\arrow[r] & {\prod_{i\neq j}\Hom_{/[u\circ^i t]}([1+\bullet]\boxtimes_{[1]}[1;u_0^j =u_0^j],X)}
\end{tikzcd}
\]
in $\PSh(\Delta)$. All vertices of this square are complete Segal spaces, so we can view this as a pullback square in $\Cat$. We have to show that the top functor induces an equivalence on classifying spaces. To this end, we will show that it is initial. Note that the right vertical map is a cartesian fibration by \cref{lemma:a cartesian fib}. So it suffices to check that the bottom horizontal map is initial.
The inclusions $[0;u_0^j]\to [1+k]\boxtimes_{[1]}[1;u_0^j=u_0^j]$ induced by $\{0\}\to [1+k]$ provide a left deformation retract for the bottom horizontal morphism, and that implies that it is initial.
\end{proof}

\begin{proposition}
\label{lemma:computing the hom of RX}
Let $X \to [u\circ^i t]$ be a complete Segal fibration.
Then the maps in the canonical span
 \[
    \begin{tikzcd}[column sep = {10em, between origins}, row sep = small]
            & \colim_{[k] \in \Delta^\op}\map_{/[u\circ^i t]}(T_k[u\circ^i t ],RX) \arrow[dl] \arrow[dr] & \\ 
            \map_{/[u\circ^i t]}([1; \underline t \actarrow u_1], RX) && \colim_{[k] \in \Delta^\op}\map_{/[u\circ^i t]}(T_k[u\circ^i t],X)
    \end{tikzcd}
\]
are equivalences. Here $\underline{t}\actarrow u_1$ is the active arrow defined in \cref{cons:grafting}.
\end{proposition}
\begin{proof}
This now readily follows from \cref{prop:computation Q} and \cref{lemma:an initiality result}.
\end{proof}

\subsection{The proof of \cref{thmB}}\label{subsec:multigrafting}
The proof of \cref{thm:main-thm-expo-CSeg-converse} proceeds by inductively grafting. To formulate this procedure, we need a more elaborate version of \cref{cons:grafting}:

\begin{construction}[The multigrafting construction]
\label{cons:multigraphting}
Let $u:u_0\actarrow u_1$ be an active arrow. Suppose that $E$ is a subset of $\pi_0(u_0)$, and that we have an active arrow $t^i:t_0^i\actarrow u_0^i$ for every $i \in E$.
We consider the unique active morphism $\underline{t}\actarrow u_0$ whose projection by the functor $\O^{\act}_{/u_0}\to \O^{\act}_{/u_0^i}$ is given by
\[
    \begin{cases}
        t^i : t_0^i \actarrow u_0^i & \text{if $i \in E$}, \\
        \id : u_0^i \to u_0^i & \text{otherwise}.
    \end{cases}
\]
We then obtain a forest 
$$
\textstyle \left<2; u \star^E (t^i)^{i\in E}\right> := \left<2; \underline{t}\actarrow u_0 \actarrow u_1\right>.
$$
For $[k] \in \Delta$, we will consider the $\Tree[\O]$-space defined by the pushout square
\[\begin{tikzcd}
	{\coprod_{i\notin E}(T_k[1]\times_{[2]}[1])\boxtimes_{[1]}[1;u_0^i = u_0^i]} & {T_k[1]\boxtimes_{[2]}[2;u\star^{E} (t^i)^{i\in E}]} \\
	{\coprod_{i\notin E}[0; u_0^i]} & {T_k[u\circ^{E} (t^i)^{i\in E}],}
	\arrow[from=1-1, to=1-2]
	\arrow[from=1-1, to=2-1]
	\arrow[from=1-2, to=2-2]
	\arrow[from=2-1, to=2-2]
\end{tikzcd}\]
and set $[u\circ^{E} (t^i)^{i\in E}] := T_0[u\circ^{E} (t^i)^{i\in E}]$.
\end{construction}

\begin{example}\label{ex:multigrafting everything}
    If $|E| = 1$, then \cref{cons:multigraphting} specializes to the previous grafting construction of \cref{cons:grafting}. We now consider the maximal case $E =\pi_0(u_0)$.
    Suppose that $u : u_0 \actarrow u_1$ is an active arrow. If $t : t_0 \actarrow u_0$ is an active arrow, then we can consider the collection of actives $(t^i : t_0^i \actarrow u_0^i)^{i \in E}$. It readily follows from the formula that
    $[u\circ^E (t^i)^{i \in E})] \simeq [2; t_0\actarrow u_0 \actarrow u_1]$.
\end{example}

We construct an auxiliary object that encodes coherences concerning the associativity of multigrafting: 

\begin{construction}
\label{cons:life saving}

    Let $E$, $u$, $(t^{i})^{i\in E}$ as in \cref{cons:multigraphting}.  Suppose then that $|E|>1$, and let $x\in E$. We will write $F := E \setminus \{x\}$.
    We consider the unique commutative diagram of active arrows
    \[
        \begin{tikzcd}
            v_{00} \arrow[r,squiggly]\arrow[d,squiggly] & v_{10} \arrow[d, squiggly] \\ 
            v_{01} \arrow[r,squiggly] & v_{11} \simeq u_0
        \end{tikzcd}
    \]
    over $u_1$ so that for $i \in \pi_0(u_0)$ its projection by the functor $\O^{\act}_{/u_0} \to \O^{\act}_{/u_{0}^i}$ is given by the square
    \[
         \begin{tikzcd}[column sep = small, row sep = small]
           u_0^i \arrow[r,equal]\arrow[d,equal]  & u_0^i \arrow[d, equal] \\
           u_0^i \arrow[r,equal] & u_0^i 
        \end{tikzcd}
        ~ \text{if $i \notin E$}, \quad
         \begin{tikzcd}[column sep = small, row sep = small]
           t_0^i \arrow[r,squiggly]\arrow[d,equal]  & u_0^i \arrow[d, equal] \\
           t_0^i \arrow[r,squiggly] & u_0^i 
        \end{tikzcd}
        ~ \text{if $i \in F$,} \quad \text{and}~
         \begin{tikzcd}[column sep = small, row sep = small]
           t_0^x \arrow[r,equal]\arrow[d,squiggly]  & t_0^x \arrow[d, squiggly] \\
           u_0^x \arrow[r,equal] & u_0^x
        \end{tikzcd}
        ~ \text{if $i =x$}.
    \]
    
    We define an auxiliary $\Tree[\O]$-space $A$ by the pushout square 
    \[
        \begin{tikzcd}
            {[2;v_{00} \actarrow v_{11} \actarrow u_1]} \arrow[r]\arrow[d] & {[3;v_{00} \actarrow v_{10} \actarrow v_{11} \actarrow u_1]} \arrow[d]  \\
           {[3;v_{00} \actarrow v_{01} \actarrow v_{11} \actarrow u_1]} \arrow[r] & A.
        \end{tikzcd}
    \]
    As $([1] \times [1])*[0] \simeq [3] \cup_{[2]} [3]$,
    the pushout $A$ is canonically fibered over $p^*(([1] \times [1])*[0])$. We then define $$
    A_{k,l} := (([1+k] \times [1+l]) * [0]) \boxtimes_{([1] \times [1])*[0]} A
    $$
    for $[k], [l] \in \Delta$. Here $[1+k], [1+l] \to [1]$ are the maps that carry $0$ to $0$ and all other elements to $1$. Moreover, for $i \in \pi_0(u_0)$, we consider the projection
     \[
        (B^i_{k,l} \to C^i_{k,l}) := \begin{cases}
            ~{([1+k] \times [1+l]) \boxtimes_{[0]} [0; u^i_0]} \rightarrow {[0;u^i_0]} & \text{if $i\notin E$}, \\
            ~{([1+k]\times[1+l])\boxtimes_{[1]}[1;t^i]} \rightarrow
        {[1+k]\boxtimes_{[1]}[1;t^i]} & \text{if $i\in F$, and} \\
            ~{([1+k]\times[1+l])\boxtimes_{[1]}[1;t^x]} \rightarrow
        {[1+l]\boxtimes_{[1]}[1;t^x]} & \text{if $i=x$}.
        \end{cases}
     \]
    We now define $\Upsilon_{k,l}$ by the pushout square
    \[
        \begin{tikzcd}
            \coprod_{i\in\pi_0(u_0)}B^i_{k,l} \arrow[r]\arrow[d] & A_{k,l} \arrow[d]\\ 
            \coprod_{i\in\pi_0(u_0)}C^i_{k,l} \arrow[r] & \Upsilon_{k,l},
        \end{tikzcd}
    \]
    where the top map is the canonical map $$\coprod_{i\in \pi_0(u_0)} B^i_{k,l} \simeq ([1+k] \times [1+l]) \boxtimes_{([1+k] \times [1+l]) \ast [0]} A_{k,l} \to A_{k,l}.$$
\end{construction}

\begin{lemma}\label{lem:easy expression upsilon}
    Let $k,l \geq 0$.
    The canonical inclusion 
    $$
     \left(\coprod_{i \in F} [1+k] \boxtimes_{[1]} [1;t^i] \sqcup [1+l] \boxtimes_{[1]} [1;t^x]\right) \bigcup_{\coprod_{i\in E}[0;u_0^i]} [1;u] \to \Upsilon_{k,l}
    $$ 
    is a Segal extension.
\end{lemma}
\begin{proof}
If we define $A'$ by the pushout square
\[
    \begin{tikzcd}
        (\coprod_{i \in \pi_0(u_0)}{[1; v_{00}^i \actarrow v_{11}^i]}) \cup_{[0; u_0]} [1;u] \arrow[r]\arrow[d] & (\coprod_{i \in \pi_0(u_0)}{[2; v_{00}^i \actarrow v_{10}^i \actarrow v_{11}^i]}) \cup_{[0; u_0]} [1;u] \arrow[d] \\ 
        (\coprod_{i \in \pi_0(u_0)}{[2; v_{00}^i \actarrow v_{01}^i \actarrow v_{11}^i]}) \cup_{[0; u_0]} [1;u] \arrow[r] & A',
    \end{tikzcd}
\]
then the canonical inclusion $A' \to A$ is a Segal extension by the decomposition formula of \cref{rem:decomposition by atoms formula}. Note that the surjective map $([1 +k] \times [1+l]) * [0] \to ([1] \times [1]) * [0]$ is exponentiable in $\Cat$. Thus it follows from \cref{prop:boxtimes with exponentiable map} that
$$
A'_{k,l} := (([1 +k] \times [1+l]) * [0]) \boxtimes_{([1] \times [1]) * [0]} A' \to A_{k,l}
$$
is a Segal extension as well. The map from the lemma is obtained after pushing out along $\coprod_{i \in \pi_0(u_0)} B^{i}_{k,l} \to \coprod_{i \in \pi_0(u_0)} C^{i}_{k,l}$.
\end{proof}

\begin{lemma}
\label{lemma:simplification of upsilon}
    There is a Segal extension 
    $$
    T_k[u\circ^E (t^i)^{i \in E}] \to \Upsilon_{k,k}
    $$
    that is natural in $[k] \in \Delta$.
\end{lemma}
\begin{proof}
    There is an obvious commutative square
    \[
        \begin{tikzcd}
            T_k[1] = {[1+k] * [0]} \arrow[r]\arrow[d] & {[1]} * [0] \arrow[d] \\
            {([1+k] \times [1+k]) * [0]} \arrow[r] & ([1]\times[1])*[0],
        \end{tikzcd}
    \]
    where the vertical maps are induced by the diagonals.
    This gives rise to a natural map $T_k[1] \boxtimes_{[2]} [2; u \star^E (t^i)^{i \in E}] \to (([1+k] \times [1+k]) * [0]) \boxtimes_{([1] \times [1]) * [0]} [2;v_{00} \actarrow v_{11} \actarrow u_1]$, where we use the notation from \cref{cons:life saving}. This map factors through the defining pushout of $T_k[u \circ^E (t^i)^{i \in E}]$ to give the natural map $T_k[u\circ^E (t^i)^{i \in E}] \to \Upsilon_{k,k}$ from the statement.
    The Segal extension of \cref{lem:easy expression upsilon} factors through this natural map via the inclusion $(\coprod_{i \in E} [1+k] \boxtimes_{[1]}[1;t^i]) \cup_{\coprod_{i \in E} [0;u_0^i]} [1;u] \to T_k[u\circ^E (t^i)^{i \in E}]$, which is also a Segal extension for similar reasons.
\end{proof}

\begin{lemma}
\label{lemma:simplification of upsilon 2}
Let $w$ denote the composite active arrow
$$
v_{01} = (u\star^F(t^i)^{i\in F})_0 \actarrow u_0 \actarrow u_1.
$$
There is a commutative square 
\[\begin{tikzcd}
	{[1;w]} & {T_k[u\circ^F(t^i)^{i\in F}]} \\
	{T_l[w\circ^xt^x]} & {\Upsilon_{k,l}}
	\arrow[from=1-1, to=1-2]
	\arrow[from=1-1, to=2-1]
	\arrow[from=1-2, to=2-2]
	\arrow[from=2-1, to=2-2]
\end{tikzcd}\]
that is natural in $[k], [l] \in \Delta$. The associated cogap map of this square is a Segal extension.
\end{lemma}
\begin{proof}
    It follows from construction that we have inert arrows $v_{01} \intarrow t_0^i$, $i \in F$, and inert arrow $v_{01} \intarrow u_0^i$, $i \in E$, which are $\pi_0$-cocartesian again by assumption (4d) of \cref{def:new robust}. The induce an equivalence $\pi_0(v_{01}) \simeq E \sqcup \coprod_{i \in F}\pi_0(t_0^i)$. From this, one deduces that $w\star^x t^x$ is given by the string of actives $v_{00}\actarrow v_{01} \actarrow u_1$ (and in particular, one deduces that $x \in E \subset \pi_0(v_{01})$ so that the statement of the lemma makes sense). 
    Let $A''$ be the $\Tree[\O]$-space defined by the pushout square 
     \[
        \begin{tikzcd}[column sep = tiny]
            {[1;w]} \arrow[r,phantom, "="]& {[1;v_{01} \actarrow u_1]} \arrow[r]\arrow[d] & {[2;v_{01} \actarrow v_{11} \actarrow u_1]} \arrow[d] \arrow[r,phantom,"="] & {[2;u\star^F (t^i)^{i\in F}]}  \\
           {[2;w \star^x t^x]} \arrow[r,phantom,"="] &{[2;v_{00} \actarrow v_{01} \actarrow u_1]} \arrow[r] & A''.
        \end{tikzcd}
    \]
    There is a canonical inclusion $A'' \to [3;v_{00} \actarrow v_{01} \actarrow v_{11} \actarrow u_1] \to A$. This lies over the map $([1] \cup_{[0]} [1]) * [0] \to ([1] \times [1]) * [0]$ coming from the morphism $[1] \cup_{[0]} [1] \to [1]\times [1]$ that selects the arrows $00 \to 01$ and $01 \to 11$. There is an obvious commutative square 
    \[
        \begin{tikzcd}
            T_k[1] \cup_{[1]} T_l[1] = ([1+k] \cup_{[0]} [1+l])*[0]\arrow[d] \arrow[r] & ([1] \cup_{[0]} [1]) * [0] \arrow[d] \\
            ([1+k] \times [1+l]) * [0] \arrow[r] & ([1] \times [1])*[0]
        \end{tikzcd}
    \] that, combined with the map $A''\to A$, eventually gives rise to the desired square of the lemma. One readily sees that the Segal extension of \cref{lem:easy expression upsilon} factors through the associated cogap map via the dashed inclusion 
    \[ 
    \begin{tikzcd}
    \textstyle \left(\coprod_{i \in F} [1+k] \boxtimes_{[1]} [1;t^i] \sqcup [1+l] \boxtimes_{[1]} [1;t^x]\right) \bigcup_{\coprod_{i\in E}[0;u_0^i]} [1;u] \arrow[d,dashed]\arrow[r]& \Upsilon_{k,l}. \\
     T_k[u\circ^F(t^i)^{i\in F}] \cup_{[1;w]} T_l[w \circ^x t^x] \arrow[ur]
    \end{tikzcd}
    \]
    By applying \cref{lem:easy expression upsilon} twice, one readily verifies that the dashed inclusion is a Segal extension as well. Thus the diagonal map must be a Segal extension as well by cancellation.
\end{proof}

\begin{lemma}
\label{lemma:the last technical lemma}
Let $p:X\to Y$ be a map between complete Segal $\Tree[\O]$-spaces. Suppose that for every active arrow $u : u_0 \actarrow u_1$ ending in an elementary, and every active arrow $t : t_0 \actarrow t_1 = u_0^i$, with $i \in \pi_0(u_0)$, the morphism
$$ \colim_{[k]\in\Delta^{\op}}\Hom_{/Y}(T_k[u\circ^i t],X)\to \Hom_{/Y}([1;\underline{t}\actarrow u_1],X)$$
is an equivalence. Then $p$ satisfies $\cref{CondCrit}$.
\end{lemma}

\begin{proof}
Let $u : u_0 \actarrow u_1$ be an active arrow ending with an elementary. Suppose that $t : t_0 \actarrow u_0$ is an active arrow. Then we must show that the map 
\begin{equation}\label{eq:the last technical lemma}
\colim_{[k]\in\Delta^{\op}}\Hom_{/Y}(T_k[1]\boxtimes_{[2]}[2;t_0\actarrow u_0 \actarrow u_1],X)\to \Hom_{/Y}([1;t_0 \actarrow u_1],X)
\end{equation}
is an equivalence.

Suppose first that $\pi_0(u_0)$ is empty. On account of condition (4c) of \cref{def:new robust}, the map $t_0 \actarrow u_0$ is an equivalence. Thus the above map is then induced by applying $(-)\boxtimes_{[1]} [1;u]$ to the natural sections $[1] \to [1+k+1]$ in $\Delta_{/[1]}$, where $[1+k+1] \to [1]$ carries only the maximal element to $1$. This (augmented) cosimplicial object admits a splitting, so \cref{eq:the last technical lemma} is an equivalence. 

Henceforth, we may assume that $\pi_0(u_0)$ is non-empty. Suppose that $E$ is a non-empty subset of $\pi_0(u_0)$ and consider the collection of actives $(t^{i} : t_0^i \actarrow u_0^i)^{i\in E}$. We will show that for every $f : [u\circ^E (t^i)^{i\in E}] \to Y$, the map 
$$\colim_{[k]\in\Delta^{\op}}\Hom_{/Y}(T_k[u\circ^E (t^i)^{i\in E}],X)\to \Hom_{/Y}([1;\underline{t}\actarrow u_1],X)$$ is an equivalence,
where $\underline{t}$ is defined as in \cref{cons:multigraphting}. The maximal case that $E =\pi_0(u_0)$ recovers precisely \cref{eq:the last technical lemma} (see \cref{ex:multigrafting everything}).
Note that the case $|E| = 1$ corresponds to the hypothesis. We now proceed by induction on the cardinality of $E$.

Suppose now that $|E|>1$, and let $x\in E$. We define $F := E \setminus \{x\}$ and we suppose that the claim holds for $F$. Consider the bisimplicial $\Tree[\O]$-space $\Upsilon_{k,l}$ of \cref{cons:life saving}. Note that we can view $f$ as a map $\Upsilon_{0,0} \to Y$ by \cref{lemma:simplification of upsilon}. We have a natural map $[1;v_{00} \actarrow u_1] = [1;\underline{t} \actarrow u_1] \to \Upsilon_{k,l}$. By \cref{lemma:simplification of upsilon}, we must show that the induced map 
$$
\colim_{[k] \in \Delta^\op} \Hom_{/Y}(\Upsilon_{k,k}, Y) \to \Hom_{/Y}([1;v_{00} \actarrow u_1], Y)
$$
is an equivalence.

We will now use the presentation of \cref{lemma:simplification of upsilon 2} to obtain a cartesian square
\[\begin{tikzcd}
	{\Hom_{/Y}(\Upsilon_{k,l},X)} & {\Hom_{/Y}( T_k[w\circ^xt^x],X)} \\
	{\Hom_{/Y}(  T_l[u\circ^F(t^i)^{i\in F}],X)} & {\Hom_{/Y}(  [1;w],X)}
	\arrow[from=1-1, to=1-2]
	\arrow[from=1-1, to=2-1]
	\arrow[from=1-2, to=2-2]
	\arrow[from=2-1, to=2-2]
\end{tikzcd}\]
where $w : v_{01} \actarrow u_1$ is the composite of the sequence $u\star^F(t^i)^{i\in F}$.
Using the assumption and the induction hypothesis, we then deduce that the natural maps $[1; v_{00} \actarrow u_1] \to \Upsilon_{k,l}$ 
induce an equivalence
\begin{align*}
    \colim_{([k],[l]) \in \Delta^\op \times \Delta^\op}\Hom_{/Y}(\Upsilon_{k,l},X)&\xrightarrow{\simeq}\colim_{[k]\in \Delta^\op }\Hom_{/Y}( T_k[w\circ^xt^x],X) \\&\xrightarrow{\simeq} \Hom_{/Y}([1;v_{00}\actarrow u_1],Y).
\end{align*}
As $\Delta^\op$ is sifted, the desired conclusion follows.
\end{proof}

\begin{proof}[Proof of \cref{thm:main-thm-expo-CSeg-converse}]
Suppose that $p:X\to Y$ is an exponentiable map between complete $\Tree[\O]$-spaces. Let $u_0 \actarrow u_1$ be an active arrow ending in an elementary. Suppose that $t : t_0 \actarrow u_0^i$ is an active map with $i \in \pi_0(u_0)$. Then we consider the following diagram of pullback squares:
\[\begin{tikzcd}
	{X''} & {X'} & X \\
	{\Gamma[u\circ^i t]} & {[u\circ^i t]} & Y.
	\arrow["v", from=1-1, to=1-2]
	\arrow[from=1-1, to=2-1]
	\arrow[from=1-2, to=1-3]
	\arrow[from=1-2, to=2-2]
	\arrow[from=1-3, to=2-3]
	\arrow[from=2-1, to=2-2]
	\arrow[from=2-2, to=2-3]
\end{tikzcd}\]
We claim that $v$ is a Segal extension. To this end, let $L : \PSh(\Tree[\O]) \to \CSeg(\Tree[\O])$ denote the reflection.
The map $[1;t_0^i \actarrow u_0^i] \cup_{[0;u_0^i]} [1;u] \simeq \Gamma[u\circ^i t]\to [u\circ^it]$ is a Segal extension and $[u\circ^it]$ is complete Segal; see \cref{prop:fibrancy of the grafting construction}.
Applying $L$, we obtain the colimit expression $[1;t_0^i] \cup_{[0;u_0^i]} [1;u] \simeq [u\circ^i t]$ in $\CSeg(\Tree[\O])$. Note that all objects in this colimit expression lie in $\CSeg(\Tree[\cO])$ by \cref{prop:fibrancy of the grafting construction}. As $p^* : \CSeg(\Tree[\O])_{/Y} \to \CSeg(\Tree[\O])_{/X}$ is cocontinuous, we thus deduce that $v$ is carried to an equivalence by $L$. Thus \cref{remark:saturated-vs-strongly-saturated} implies that $v$ is a complete Segal extension. 

Now, using \cref{lemma:fibrancy of RX} and reasoning similarly as in \cref{lemma:exponentiable Segal fibrations}, we have a factorization $X'' \to RX' \to X'$ of $v$ into a complete Segal extension followed by a complete Segal fibration. Thus $RX' \to X'$ must be an equivalence.  By \cref{lemma:computing the hom of RX}, the map $p$ fulfills the hypothesis of \cref{lemma:the last technical lemma} and thus verifies \cref{conditionCC}. 
\end{proof}

\subsection{Examples}\label{subsection:examples of CC converse}
We now highlight some sample applications of \cref{thmB}. 

\begin{example}\label{exa:CC-necessary-VDC}
    By \cref{thm:main-thm-expo-CSeg-converse} (or \cref{rem:thmB easy for atomic}) the conditions of \cref{example:CC-criterion-VDC} for a map in $\Algd(\Delta^{\op,\natural})$ to be exponentiable are also necessary.
    That is, a map $\cP \to \cQ$ of virtual double categories is exponentiable if and only if the following conditions hold:
    \begin{enumerate}[(1)]
        \item the functor $\cP_0 \to \cQ_0$ is exponentiable in $\Cat$, and
        \item for any $t \colon [2] \to \cQ^\activ$ such that $t(2)$ lies over $[1]$, the base-change $\cP^\activ \times_{\cQ^\activ} [2] \to [2]$ is exponentiable in $\Cat$.
    \end{enumerate}
\end{example}

\begin{example}\label{exa:CC-necessary-operad}
    Let $\cO = \F_*^\flat$ be the pattern describing operads. Suppose that $f \colon \cP \to \cQ$ is a map between operads.
    Then we claim that $f$ is exponentiable if and only if the equivalent conditions (1) and (2) of \cref{ex:CC-criterion-operads} hold. As explained in \cref{example:fin*-not-robust}, $\F_*^\flat$ is not robust, so we cannot directly apply \cref{thm:main-thm-expo-CSeg-converse}.
    
    However, note that there is an inclusion $\phi \colon \F_* \to \Span(\F)$ as the wide subcategory of $\Span(\F)$ containing the spans whose backward arrows are injective; concretely, a morphism $f \colon X \sqcup \{*\} \to Y \sqcup \{*\}$ is sent by $f$ to the span
    \[\begin{tikzcd}
    	X & {f^{-1}(Y)} & Y 
    	\arrow[hook', from=1-2, to=1-1]
    	\arrow["f", from=1-2, to=1-3]
    \end{tikzcd}\]
    By \cite[Corollary B]{BarkanHaugsengSteinebrunner}, pullback along $\phi$ induces an equivalence
    \[\phi^* \colon \Algd(\Span(\F)^\flat) \xrightarrow{\simeq} \Algd(\F_*^\flat) = \mathrm{Op},\]
    where $\Span(\F)^\flat$ has the pattern structure from \cref{examples:examples of patterns}.
    Alternatively, this equivalence follows from the fact that the tree categories of $\Span(\F)^\flat$ and $\F_*^\flat$ agree, see \cref{exa:span-forests}.
    Since $\Span(\F)^\flat$ is robust by \cref{example:Span-robust}, it follows that items (1) and (2) are necessary when $\F_*$ is replaced by $\Span(\F)$.
    Because the functor $\phi \colon \F_* \to \Span(\F)$ is an equivalence on active morphisms, it follows that (1) and (2) are also necessary for $\F_*^\flat$.
\end{example}

\begin{example}\label{exa:CC-necessary-genoperad}
    The inclusion $\phi \colon \F_* \hookrightarrow \Span(\F)$ from the previous example also induces an equivalence
    \[ \Algd(\Span(\F)^\natural) \simeq \Algd(\F^\natural_*).\]
    This follows by checking the conditions of \cite[Theorem A]{BarkanHaugsengSteinebrunner}, but can also be deduced as follows:
    By \cref{exa:span-forests}, the inclusion $\phi$ yields an equivalence $\Tree[\F_*^\natural] \simeq \Tree[\Span(\F)^\natural]$.
    This is an equivalence of algebraic patterns for the pattern structure of \cref{ex:iterated trees}, hence we obtain an equivalence
    \[\CSeg(\Omega[\Span(\F)^\natural]) \simeq \CSeg(\Omega[\F^\natural_*])\]
    by \cref{prop:segal tree spaces are segal objects}.
    \Cref{prop:comparison algebrads vs segal spaces} then shows that pulling back along $\phi$ gives an equivalence between categories of algebrads.
    Because the functor $\phi \colon \F_* \to \Span(\F)$ is an equivalence on active morphisms, it follows that conditions (1) and (2) of \cref{example:CC-criterion-gen-operads} are both necessary and sufficient for a map in $\Algd(\F_*^\natural)$ to be exponentiable.
    In particular, combining this example with the previous one, it follows that a map of operads is exponentiable in the category of operads if and only if it is exponentiable in the category of generalized operads.
\end{example}

\begin{example}\label{exa:CC-necessary-G-operad}
    Let $G$ be a finite group and consider the pattern $\Span(\F_G)^\flat$ from \cref{examples:examples of patterns} describing $G$-operads. This is robust by \cref{example:Span-robust}. Thus conditions (1) and (2) of \cref{example:CC-criterion-G-operads} are also necessary for a map of $G$-operads to be exponentiable. 
    
    As explained in \cref{example:CC-criterion-G-operads}, Nardin--Shah \cite{NardinShah} modelled $G$-operads as algebrad for a different pattern $\underline{\F}_{G,*}$ and showed that a map $\cP \to \cQ$ between such algebrads is exponentiable if $\cP^\activ \to \cQ^\activ$ is exponentiable in $\Cat$.
    While $\underline{\F}_{G,*}$ is not robust, we can show that their condition is necessary by comparing the pattern to $\underline{\F}_{G,*}$ to $\Span(\F_G)^\flat$.
    We first briefly recall the equivalence of $\Algd(\underline{\mathbb{F}}_{G,*})$ with $\Algd(\Span(\F_G)^\flat)$ proved in \cite[Corollary B]{BarkanHaugsengSteinebrunner}.
    
    The category $\underline{\F}_{G,*}$ is defined as the cocartesian unstraightening $p \colon \underline{\F}_{G,*} \to \mathrm{Orb}_G^\op$ of the functor 
    \[\mathrm{Orb}_G^\op \to \Cat; \quad G/H \mapsto \F_{H,*},\]
    where $\mathrm{Orb}_G^\op$ is the category of transitive $G$-sets and $\F_{H,*}$ the category of pointed finite $H$-sets.
    A map in $\underline{\F}_{G,*}$ is inert if it is $p$-cocartesian and active if it lies over an equivalence in $\mathrm{\Orb}_G$.\footnote{In \cite{NardinShah}, a slightly larger class of actives is used, but their inerts and actives don't form a factorization system. Moreover, their exponentiability criterion \cite[Definition 3.1.1]{NardinShah} uses the actives we define here.}
    There is a functor $\phi \colon \underline{\F}_{G,*} \to \Span(\F_G)$ that sends a pointed $H$-set $X \sqcup \{*\}$ to the $G$-set $\mathrm{ind}_H^G X$. By \cite[Proposition 5.2.14]{BarkanHaugsengSteinebrunner}, pullback along $\phi$ induces an equivalence
    \[\phi^* \colon \Algd(\Span(\F_G)) \to \Algd(\underline{\F}_{G,*}).\]
    The inclusion $\{G/G\} \hookrightarrow \mathrm{Orb}^\op_G$ induces an inclusion $\F_G \hookrightarrow \underline{\F}_{G,*}^\activ$ whose composition with $\phi^\activ \colon \underline{\F}_{G,*}^\activ \to \Span(\F_G)^\activ = \F_G$ is the identity.
    We therefore see that given a map $\cQ \to \cP$ in $\Algd(\Span(\F_G))$, the functor $\cQ^\activ \to \cP^\activ$ is exponentiable in $\Cat$ if and only if $(\phi^* \cP)^\activ \to (\phi^* \cQ)^\activ$ is exponentiable in $\Cat$.
    This is precisely the exponentiability condition of \cite[Corollary 3.1.5]{NardinShah}.
\end{example}

\begin{example}
    The algebraic patterns $\Theta_n$ and $\Delta^{\times n, \op, \natural}$ from \cref{example:theta-robust} and \cref{example:Deltan-robust} are (atomically) robust, so \cref{thm:main-thm-expo-factorization-Algd} gives a complete characterization of the exponentiable morphisms in $\Algd(\Theta_n)$ and $\Algd(\Delta^{\times n, \op, \natural})$.
\end{example}

\section{Examples of exponentiable maps}\label{sec:examples}
We will now give concrete examples of exponentiable maps between algebrads using \cref{thmA}.
We start with observing that any Segal $\cO$-category $\cO \to \Cat$ is exponentiable when viewed as an $\cO$-algebrad.
Secondly, we consider the case where every object of $\cO$ is elementary.
In this case $\Algd(\cO) \simeq \Cocart^\inert(\cO)$, and we obtain a complete characterization of the exponentiable morphisms in $\Cocart^\inert(\cO)$.
We then give a very small example of an exponentiable morphism between virtual double categories that does not fall in the previous classes.
This is also the counterexample mentioned in \cref{intro:other results}.
Finally, we will consider the example of the (virtual) cospan double category $\DCospan(\cC)$ of a category $\cC$.
If $\cC$ admits pushouts, then this is a double category and hence it is exponentiable within the category of virtual double categories.
We will show that if $\cC$ does not admit pushouts, then $\DCospan(\cC)$ still exists as a virtual double category, and in that it is always exponentiable.

\subsection{Segal \texorpdfstring{$\cO$}{O}-categories} Recall that Segal $\O$-categories (\cref{def:segal objects}) can be viewed as particular examples of $\O$-algebrads via unstraightening. It turns out that these are always exponentiable in $\Algd(\O)$:

\begin{proposition}\label{prop:segal-o-cat-expo}
    Let $\cO$ be an algebraic pattern and $X \colon \cO \to \Cat$ a Segal $\cO$-category.
    Then its unstraightening $\smallint X \to \cO$ is an exponentiable object in $\Algd(\cO)$.
\end{proposition}

\begin{proof}
    By \cite[Lemma 3.2.1]{AyalaFrancis2020FibrationsInftyCategories}, $\smallint X \to \cO$ is exponentiable in $\Cat$.
    In particular, it satisfies the criterion \cref{CondCrit} from \cref{thmA}.
\end{proof}

\begin{example}\label{exa:double-cat-is-expo}
    The category of double categories is equivalent to $\DblCat = \Seg(\Delta^{\op,\natural},\Cat)$.
    In particular, any double category, when viewed as an object in the category $\VirtDblCat = \Algd(\Delta^{\op,\natural})$ of virtual double categories, is exponentiable.
    This is an $\infty$-categorical version of \cite[Theorem 3.9]{Arkor}.
\end{example}

\begin{warning}\label{warning:segal-o-cat-expo}
    It follows from \cref{prop:segal-o-cat-expo} that any two Segal $\cO$-categories $X$ and $Y$ admit an exponential object $[X,Y]$ in $\Algd(\cO)$.
    However, $[X,Y] \to \cO$ is generally not a cocartesian fibration and hence not a Segal $\cO$-category.
    Even if it is a Segal $\cO$-category, it is usually not an exponential object in $\Seg(\cO,\Cat)$ since the morphisms in $\Algd(\cO)$ and $\Seg(\cO,\Cat)$ don't agree.
    For example, any double category $X$ is exponentiable in both $\DblCat$ and $\VirtDblCat$, but the exponential objects generally don't agree:
    Given a double category $Y$, the objects of the exponential object $[[1]_h,Y]_\DblCat$ in $\DblCat$ are horizontal arrows in $Y$, while the objects of the exponential object $[[1]_h,Y]_\VirtDblCat$ are pairs of objects in $Y$ (cf.\ \cref{exa:underlying-graph-vdc}). Here $[1]_h$ denotes the free-living double category containing a horizontal arrow.
\end{warning}

\begin{example}
    Let $\O$ be an algebraic pattern. On account of \cref{prop:segal tree spaces are segal objects} there is a composite inclusion of (non-full) subcategories
    \begin{align*}
        \Algd(\O) &\simeq \CSeg(\Tree[\O]) \hookrightarrow \Seg(\Tree[\O]) = \Seg(\Tree[\O]^{\op, \natural}, \Spc) \\
        &\hookrightarrow \Seg(\Tree[\O]^{\op,\natural}, \Cat) \hookrightarrow \Algd(\Tree[\O]^{\op,\natural}).
    \end{align*}
    On account of \cref{prop:segal-o-cat-expo}, the inclusion carries each $\O$-algebrad to an exponentiable $\Tree[\O]^{\op, \natural}$-algebrad (but beware of \cref{warning:segal-o-cat-expo}).

    By iterating the tree construction (see \cref{ex:iterated trees}), we now produce a sequence of inclusions 
    $$
    \Algd(\O) \to \Algd(\Tree[\O]^{\op, \natural}) \to \Algd(\Tree^2[\O]^{\op, \natural})\to \dotsb,
    $$
    whose images consist of exponentiable objects.
\end{example}


\subsection{Exponentiable morphisms in \texorpdfstring{$\Cocart^\inert(\cO)$}{Cocart\textasciicircum int(O)}}

Let $\cO$ be a category with a factorization system $(\cO^\inert,\cO^\activ)$.
Then we can make $\cO$ into an algebraic pattern by declaring that every object is elementary.
By \cref{remark:algebrads-all-elementary}, it follows that $\Cocart^\inert(\cO) \simeq \Algd(\cO)$, while $\cO$ is robust by \cref{exa:all-elementary-then-robust}.
In particular, we obtain the following complete characterization of the exponentiable morphisms in $\Cocart^\inert(\cO)$.

\begin{proposition}\label{prop:expo-inert-fibrations}
    Let $\cO$ be a category with a factorization system $(\cO^\inert,\cO^\activ)$.
    Then a morphism $\cP \to \cQ$ is exponentiable in $\Cocart^\inert(\cO)$ precisely if the functor
    \[\cP^\activ \coloneqq \cP \times_\cO \cO^\activ \to \cQ \times_\cO \cO^\activ \eqqcolon \cQ^\activ\]
    is exponentiable in $\Cat$.
\end{proposition}

\begin{proof}
    Since every object in $\cQ$ lies over an elementary object in $\cO$, the condition from \cref{thm:main-thm-expo-factorization-Algd} is equivalent to the Conduché criterion for the functor $\cP^\activ \to \cQ^\activ$ from \cite[Lemma 2.2.8]{AyalaFrancis2020FibrationsInftyCategories}.
\end{proof}

\begin{example}
    Suppose that $\C$ is a category. We may then apply \cref{prop:expo-inert-fibrations} to the factorization system whose left class is given by all maps in $\C$. Then it follows from \cref{prop:expo-inert-fibrations} that a map $F \to G$ in $\Fun(\C,\Cat)$ is exponentiable if and only if each component $Fc\to Gc$ is exponentiable.
\end{example}

\subsection{A non-trivial exponentiable map between virtual double categories}\label{subsection:non-triv exp map between vdc} 
We will discuss a toy example of an exponentiable map in $\VirtDblCat \simeq \Algd(\Delta^{\op,\natural})$, where $\Delta^{\op,\natural}$ is the algebraic pattern defined in \cref{examples:examples of patterns}. This is also the counterexample mentioned in \cref{intro:other results}.

\begin{construction}
    We will write $(s,t \colon \{01,12\} \rightrightarrows \{0,1,2\}) \in \Fun(\mathbb{G},\Cat)$ for the underlying graph of $[0;[2]]$ (see \cref{exa:underlying-graph-vdc}). We define $$f \colon A \to B \coloneqq [2; [2] = [2]=[2]]$$ to be the map of unary virtual double categories that is classified by the map of graphs
    \[\begin{tikzcd}
	{[2] \times \{01,12\}} & {[2] \times\{01, 12\}} \\
	{[3] \times\{0,1,2\}} & {[2] \times\{0,1,2\}}.
	\arrow["\id", from=1-1, to=1-2]
	\arrow["{d_1 \times s}"', from=1-1, to=2-1]
	\arrow["{d_2 \times t}", shift left=3, from=1-1, to=2-1]
	\arrow["{\id \times s}"', from=1-2, to=2-2]
	\arrow["{\id \times t}", shift left=3, from=1-2, to=2-2]
	\arrow["{s_1 \times \id}", from=2-1, to=2-2]
    \end{tikzcd}\]
    Pictorially, it corresponds to a map
    \[
    f \colon A = \begin{tikzcd}
	\bullet & \bullet & \bullet \\
	\bullet & \bullet & \bullet \\
	\bullet & \bullet & \bullet \\
	\bullet & \bullet & \bullet
	\arrow[""{name=0, anchor=center, inner sep=0}, "\shortmid"{marking}, from=1-1, to=1-2]
	\arrow[from=1-1, to=2-1]
	\arrow[""{name=1, anchor=center, inner sep=0}, "\shortmid"{marking}, from=1-2, to=1-3]
	\arrow[from=1-2, to=2-2]
	\arrow[from=1-3, to=2-3]
	\arrow[from=2-1, to=3-1]
	\arrow[""{name=2, anchor=center, inner sep=0}, "\shortmid"{marking}, from=2-2, to=2-3]
	\arrow[from=2-2, to=3-2]
	\arrow[from=2-3, to=3-3]
	\arrow[""{name=3, anchor=center, inner sep=0}, "\shortmid"{marking}, from=3-1, to=3-2]
	\arrow[from=3-1, to=4-1]
	\arrow[from=3-2, to=4-2]
	\arrow[from=3-3, to=4-3]
	\arrow[""{name=4, anchor=center, inner sep=0}, "\shortmid"{marking}, from=4-1, to=4-2]
	\arrow[""{name=5, anchor=center, inner sep=0}, "\shortmid"{marking}, from=4-2, to=4-3]
	\arrow[between={0.2}{0.8}, Rightarrow, from=0, to=3]
	\arrow[between={0.2}{0.8}, Rightarrow, from=1, to=2]
	\arrow[between={0.2}{0.8}, Rightarrow, from=2, to=5]
	\arrow[between={0.2}{0.8}, Rightarrow, from=3, to=4]
\end{tikzcd}
    \quad \longrightarrow \quad  
    \begin{tikzcd}
	\bullet & \bullet & \bullet \\
	\bullet & \bullet & \bullet \\
	\bullet & \bullet & \bullet
	\arrow[""{name=0, anchor=center, inner sep=0}, "\shortmid"{marking}, from=1-1, to=1-2]
	\arrow[from=1-1, to=2-1]
	\arrow[""{name=1, anchor=center, inner sep=0}, "\shortmid"{marking}, from=1-2, to=1-3]
	\arrow[from=1-2, to=2-2]
	\arrow[from=1-3, to=2-3]
	\arrow[""{name=2, anchor=center, inner sep=0}, "\shortmid"{marking}, from=2-1, to=2-2]
	\arrow[from=2-1, to=3-1]
	\arrow[""{name=3, anchor=center, inner sep=0}, "\shortmid"{marking}, from=2-2, to=2-3]
	\arrow[from=2-2, to=3-2]
	\arrow[from=2-3, to=3-3]
	\arrow[""{name=4, anchor=center, inner sep=0}, "\shortmid"{marking}, from=3-1, to=3-2]
	\arrow[""{name=5, anchor=center, inner sep=0}, "\shortmid"{marking}, from=3-2, to=3-3]
	\arrow[between={0.2}{0.8}, Rightarrow, from=0, to=2]
	\arrow[between={0.2}{0.8}, Rightarrow, from=1, to=3]
	\arrow[between={0.2}{0.8}, Rightarrow, from=2, to=4]
	\arrow[between={0.2}{0.8}, Rightarrow, from=3, to=5]
\end{tikzcd}
= B.
    \]
    The map $f$ is unique in the sense that the \textit{total vertical pasting} of $A$, given by the unique inclusion $[1; [2] = [2]] \to A$, is carried by $f$ to the \textit{total vertical pasting} $[1;[2] = [2]] \to B$ of $B$, which is classified by the tuple of maps $(d_1 \colon [1] \to [2], \id_{[2]} \colon [2] \to [2])$ in $\Delta$.
\end{construction}

The following shows that it is necessary in \cref{thm:main-thm-expo-CSeg} to only restrict to trees $[2;t]$ and that moving to \textit{forests} $\left<2;t\right>$ would be too broad.

\begin{proposition}\label{prop:conter-example-forest-conduche}
    The map $f \colon A\to B$ is exponentiable in $\VirtDblCat$. However, for the identity map $[2;[2] = [2] = [2]] \to B$, the similarly defined comparison map of \cref{thm:main-thm-expo-CSeg} is not an equivalence.
\end{proposition}
\begin{proof}
    For the first assertion, we note that there is an equivalence 
    $$\iota \colon \Fun(\mathbb{G},\Cat)_{/\Gamma B} \xrightarrow{\simeq} \VirtDblCat_{/B}$$
    on account of \cref{prop:unary algebrads} and \cref{cor:unary only accepts unary} (cf.\ \cref{rem:trivial algebrad}). Hence, it suffices to show that $\Gamma f \colon \Gamma A \to \Gamma B$ is an exponentiable map in $\Fun(\mathbb{G},\Cat)$. This can be checked pointwise, and follows from the fact that $s_1 \colon [3] \to [2]$ and $\id_{[2]} \colon [2]\to[2]$ are exponentiable in $\Cat$.
    
    For the second assertion, we note that the total vertical pasting of $B$ extends to the identity $[2;[2] = [2] = [2]] \to B$. By construction of $f$, the total vertical pasting of $A$ lifts the total vertical pasting of $B$, yet the total vertical pasting of $A$ cannot be extended along $[1;[2]=[2]] \to [2;[2]=[2]=[2]]$ in a compatible way.
\end{proof}

\begin{remark}\label{rem:VDC-expo-vs-actives-conduche}
    When viewing $A$ and $B$ as $\Delta^{\op,\natural}$-algebrads in the sense of \cref{defi:of algebrad for factactization system}, \cref{prop:conter-example-forest-conduche} says that while $f \colon A \to B$ satisfies the criterion of \cref{thm:main-thm-expo-factorization-Algd}, the functor on actives $A^\activ \to B^\activ$ is not exponentiable in $\Cat$.
    Namely, it shows that there exists a map $\alpha$ in $A^\activ$ together with a factorization $f(\alpha) = \beta \circ \gamma$ in $B^\activ$ such that $\beta$ and $\gamma$ do not lift to $A^\activ$.
\end{remark}

\subsection{The virtual cospan double category}\label{subsection:cospans}

We will now construct, for any category $\cC$, its \emph{virtual double category of cospans} $\DCospan^\mathrm{virt}(\cC)$ and show that it is exponentiable in $\VirtDblCat$.

Recall that for a category $\cC$ that admits pushouts, its cospan double category $\DCospan(\cC)$ is constructed by first constructing a bigger simplicial category
\[\overline{\DCospan}(\cC) \colon \Delta^\op \to \Cat; \quad \overline{\DCospan}(\cC)_n = \Fun(\Delta^{\inert}_{/[n]},\cC)\]
and then restricting, for every $n$, to the full subcategory $\DCospan(\cC)_n \subset \overline{\DCospan}(\cC)_n$ spanned by those functors $\Delta^\inert_{/[n]} \to \cC$ that are left Kan extended from $\Delta^{\elem}_{/[n]}$. For details, including the proof that this defines a double category, we refer the reader to \cite[\S 5]{HaugsengIterated} where the dual case of spans is discussed.
Since $\Delta^\inert_{/[1]} = \Delta^\elem_{/[1]} \simeq (\bullet \rightarrow \bullet \leftarrow \bullet)$, we see that the horizontal morphisms of $\DCospan(\cC)$ are indeed cospans.

If $\cC$ does not admit pushouts, then we embed $\cC$ into $\PSh(\cC^\op)^\op$ via the Yoneda embedding.\footnote{We may always assume $\cC$ is small by considering presheaves in a larger universe.}
Then $\PSh(\cC^\op)^\op$ admits pushouts and the inclusion $\cC \hookrightarrow \PSh(\cC^\op)^\op$ preserves all pushouts that exist in $\cC$.
We will define the virtual double category $\DCospan^\mathrm{virt}(\cC)$ as a subobject of the double category $\DCospan(\PSh(\cC^\op)^\op)$.
In what follows, we write $\widehat{\cC}$ for $\PSh(\cC^\op)^\op$

\begin{construction}
    Let $\cC$ be a category and write $\mathscr{E} \coloneqq \smallint \DCospan(\widehat{\cC}) \to \Delta^\op$ for the unstraightening of the double category $\DCospan(\widehat{\cC})$.
    Then we may identify the fiber of $\mathscr{E}$ over $[0]$ with $\widehat{\cC}$ and over $[1]$ with the category of cospans $\Fun(\bullet \rightarrow \bullet \leftarrow \bullet, \widehat{\cC})$.
    Using the Segal condition, we may identify the objects in the fiber $\DCospan(\widehat{\cC})_n$ over $[n]$ with iterated cospans
    \[\begin{tikzcd}[sep={2.3em, between origins}]
    	& {c'_1} && \cdots && {c'_n} & \\
    	{c_0} && {c_1} && {c_{n-1}} && {c_n}
    	\arrow[to=1-2, from=2-1]
    	\arrow[to=1-2, from=2-3]
    	\arrow[to=1-4, from=2-3]
    	\arrow[to=1-4, from=2-5]
    	\arrow[to=1-6, from=2-5]
    	\arrow[to=1-6, from=2-7]
    \end{tikzcd}\]
    in $\widehat{\cC}$.
    We define $\mathscr{E}'$ as the full subcategory of $\mathscr{E}$ spanned by those iterated cospans for which the objects $c_{0},\ldots,c_n,c'_1, \ldots, c'_n$ all lie in the image of $\cC \hookrightarrow \widehat{\cC}$.
    We leave it to the reader to verify that $\mathscr{E}' \to \Delta^\op$ defines a $\Delta^{\op,\natural}$-algebrad, i.e.\ a virtual double category.
    We will denote this virtual double category by $\DCospan^\mathrm{virt}(\cC)$.
\end{construction}

\begin{remark}
    Suppose that $\cC$ admits pushouts. Since the inclusion $\cC \hookrightarrow \widehat{\cC}$ preserves all pushouts that exist in $\cC$, it follows that $\DCospan^\mathrm{virt}(\cC) \simeq \smallint \DCospan(\cC)$ in this case.
\end{remark}

\begin{remark}
    One can in fact show that the construction of $\DCospan^\mathrm{virt}(\cC)$ is independent of the choice of embedding of $\cC$ into a category that admits pushouts.
\end{remark}

We will now show that $\DCospan^\mathrm{virt}(\cC)$ is always exponentiable in the category of virtual double categories, even if $\cC$ does not admit pushouts.
The main ingredient is the following lemma.

\begin{lemma}\label{lem:cartesian-recognition-expo}
    Let $\cO$ be an algebraic pattern and $p \colon \cP \to \cO$ an algebrad. 
    Suppose that $\cP$ admits $p$-cartesian lifts for every active morphism $x \actarrow e$ in $\cO$ whose target is elementary.
    Then $\cP$ is exponentiable in $\Algd(\cO)$.
\end{lemma}

\begin{proof}
We will check the criterion \cref{CondCrit} from \cref{thmA}, so let $h \colon x \actarrow y$ and $g \colon y \actarrow e$ be active morphisms in $\cO$ with $e$ elementary, and let $f \colon \overline{x} \to \overline{e}$ be a lift of $g \circ h$ in $\cP$.
Choose a $p$-cartesian lift $\overline{g} \colon \overline{y} \to \overline{e}$ of $g$; then there exists a unique lift $\overline{h} \colon \overline{x} \to \overline{y}$ such that $\overline{g} \circ \overline{h} = f$.
An argument dual to \cite[Lemma 3.2.1]{AyalaFrancis2020FibrationsInftyCategories} shows that $\overline{x} \overset{\overline{h}}{\actarrow} \overline{y} \overset{\overline{g}}{\actarrow} \overline{e}$
is a terminal object of $\Fact(f \mid g \circ h)$, hence this category is weakly contractible.
\end{proof}

\begin{proposition}
    Let $\cC$ be any category.
    Then $\DCospan^\mathrm{virt}(\cC)$ is an exponentiable virtual double category.
\end{proposition}

\begin{proof}
    We need to show that $\DCospan^\mathrm{virt}(\cC) \to \Delta^\op$ is exponentiable in $\Algd(\Delta^{\op,\natural})$.
    By construction, $\DCospan^\mathrm{virt}(\cC)$ is a subobject of $\smallint \DCospan(\widehat{\cC})$.
    Let $\phi \colon [1] \actarrow [n]$ be an active morphism in $\Delta$.
    Then the cocartesian transport functor $\DCospan(\widehat{\cC})_n \to \DCospan(\widehat{\cC})_1$ may be identified with the left Kan extension functor
    \[\mathrm{LKan}_{\phi^*} \colon \Fun(\Delta^\elem_{/[n]},\widehat{\cC}) \to \Fun(\Delta^\elem_{/[1]},\widehat{\cC})\]
    along $\phi^* \colon \Delta^\elem_{/[n]} \to \Delta^\elem_{/[1]}$.
    In particular, it admits a right adjoint which is given by the restriction functor
    \[\mathrm{res}_{\phi^*} \colon \Fun(\Delta^\elem_{/[1]},\widehat{\cC}) \to \Fun(\Delta^\elem_{/[n]},\widehat{\cC}) \]
    It follows as in the proof of (the dual of) \cite[Corollary 5.2.2.5]{HTT} that $\DCospan(\widehat{\cC}) \to \Delta^\op$ admits cartesian lifts of $\phi$.
    Since $\mathrm{res}_{\phi^*}$ takes the full subcategory $\DCospan^\mathrm{virt}(\cC)_1$ to $\DCospan^\mathrm{virt}(\cC)_n$, it follows that $\DCospan^\mathrm{virt}(\cC) \to \Delta^\op$ admits cartesian lifts of $\phi$ as well.
    Since the only active morphism in $\Delta^\op$ with target $[0]$ is the identity, it follows that $\DCospan^\mathrm{virt}(\cC) \to \Delta^\op$ admits cartesian lifts of every active morphism whose target is elementary.
    The result now follows from \cref{lem:cartesian-recognition-expo}.
\end{proof}

\begin{remark}
    The cartesian transport functor $\DCospan^\mathrm{virt}(\cC)_1 \to \DCospan^\mathrm{virt}(\cC)_n$ can explicitly be described as follows: it takes a cospan $c_0 \rightarrow c_1 \leftarrow c_2$ to the iterated cospan
    \[\begin{tikzcd}[sep={2.3em, between origins}]
    	& {c_1} && \cdots && {c_1} & \\
    	{c_0} && {c_1} && {c_1} && {c_2.}
    	\arrow[to=1-2, from=2-1]
    	\arrow[to=1-2, from=2-3, equals]
    	\arrow[to=1-4, from=2-3, equals]
    	\arrow[to=1-4, from=2-5, equals]
    	\arrow[to=1-6, from=2-5, equals]
    	\arrow[to=1-6, from=2-7]
    \end{tikzcd}\]
\end{remark}

\bibliographystyle{amsalpha}
\bibliography{main}

@article{MoerdijkWeiss,
 author = {Moerdijk, Ieke and Weiss, Ittay},
 title = {Dendroidal sets},
 journal = {Algebr. Geom. Topol.},
 volume = {7},
 pages = {1441--1470},
 year = {2007}
}

@misc{Kern,
 author = {Kern, David},
 title = {Monoidal envelopes and {Grothendieck} construction for dendroidal {Segal} objects},
 year = {2023},
 howpublished = {{arXiv}:2301.10751}
}

@article{Barwick,
 author = {Barwick, Clark},
 title = {From operator categories to higher operads},
 journal = {Geometry \& Topology},
 volume = {22},
 number = {4},
 pages = {1893--1959},
 year = {2018}
}

@article{Rezk,
  author = {Rezk, Charles},
  title = {A model for the homotopy theory of homotopy theory},
  journal = {Transactions of the American Mathematical Society},
  volume = {353},
  year = {2001},
  number = {3},
  pages = {973-1007}
}

@misc{NardinShah,
 author = {Nardin, Denis and Shah, Jay},
 title = {Parametrized and equivariant higher algebra},
 year = {2022},
 howpublished = {{arXiv}:2203.00072}
}

@misc{HA,
    author = {Lurie, Jacob},
    title = {Higher algebra},
    howpublished = {\url{https://www.math.ias.edu/~lurie/papers/HA.pdf}},
    year = {2017}
}

@article{BarkanHaugsengSteinebrunner,
 author = {Barkan, Shaul and Haugseng, Rune and Steinebrunner, Jan},
 title = {Envelopes for algebraic patterns},
 journal = {Algebraic \& Geometric Topology},
 volume = {25},
 number = {9},
 pages = {5319--5388},
 year = {2025}
}

@misc{Arkor,
 author = {Arkor, Nathanael},
 title = {Exponentiable virtual double categories and presheaves for double categories},
 year = {2025},
 howpublished = {{arXiv}:2508.11611}
}

@misc{EaEThompson,
    title = {Pro-representable virtual double categories},
    author = {Thompson, Ea E},
    year = {2025},
    howpublished = {talk at the Topos Institute Berkeley Seminar, slides, \url{https://categorytheory.zulipchat.com/user_uploads/21317/kFFsTvEkBvYe_xZOhIkeem8f/Pro-representableVDCs.pdf}},
    note = {accessed on 9 Sep.\ 2025}
}

@misc{Arkor_slide,
    author = {Arkor, Nathanael},
    title = {The three-dimensional structures formed by monoidal categories, bicategories, double categories, etc.},
    howpublished = {talk at the International Category Theory Conference, slides, \url{https://arkor.co/files/Three-dimensional%20structures%20(CT%202025).pdf}},
    year = {2025},
    note = {accessed on 12 Feb.\ 2026}
}

@article{ChuHaugseng2021HomotopycoherentAlgebraSegal,
  title = {Homotopy-coherent algebra via {S}egal conditions},
  author = {Chu, Hongyi and Haugseng, Rune},
  year = {2021},
  journal = {Advances in Mathematics},
  volume = {385},
  pages = {107733},
  doi = {10.1016/j.aim.2021.107733}
}

@article{GepnerHaugseng2015EnrichedCategoriesNonsymmetric,
  title = {Enriched $\infty$-categories via non-symmetric $\infty$-operads},
  author = {Gepner, David and Haugseng, Rune},
  year = {2015},
  journal = {Advances in Mathematics},
  volume = {279},
  pages = {575--716},
  doi = {10.1016/j.aim.2015.02.007}
}

@article{ChuHaugsengea2018TwoModelsHomotopy,
  title = {Two models for the homotopy theory of $\infty$-operads},
  author = {Chu, Hongyi and Haugseng, Rune and Heuts, Gijs},
  year = {2018},
  journal = {Journal of Topology},
  volume = {11},
  number = {4},
  pages = {857--873},
  doi = {10.1112/topo.12071}
}

@misc{BrankoJuran,
  title={On orthogonal factorization systems and double categories},
  author={Juran, Branko},
  howpublished={arXiv:2501.01363},
  year={2025}
}

@article{MoerdijkWeiss2007DendroidalSets,
  title = {Dendroidal sets},
  author = {Moerdijk, Ieke and Weiss, Ittay},
  year = 2007,
  journal = {Algebraic \& Geometric Topology},
  volume = {7},
  number = {3},
  pages = {1441--1470},
  doi = {10.2140/agt.2007.7.1441}
}

@book{HTT,
  title = {Higher topos theory},
  author = {Lurie, Jacob},
  year = 2009,
  series = {Annals of Mathematics Studies},
  number = {170},
  publisher = {Princeton University Press},
  address = {Princeton, N.J},
}

@article{BonventrePereira2020EquivariantDendroidalSegal,
  title = {Equivariant dendroidal Segal spaces and {G}–{$\infty$}–operads},
  author = {Bonventre, Peter and Pereira, Luís A.},
  year = 2020,
  month = dec,
  journal = {Algebraic \& Geometric Topology},
  volume = {20},
  number = {6},
  pages = {2687--2778},
  publisher = {Mathematical Sciences Publishers},
  doi = {10.2140/agt.2020.20.2687}
}

@article{BonventrePereira2022EquivariantDendroidalSets,
  title = {Equivariant dendroidal sets and simplicial operads},
  author = {Bonventre, Peter and Pereira, Luís A.},
  year = 2022,
  journal = {Journal of Topology},
  volume = {15},
  number = {2},
  pages = {745--805},
  doi = {10.1112/topo.12223}
}

@article{Pereira2018EquivariantDendroidalSetsa,
  title = {Equivariant dendroidal sets},
  author = {Pereira, Luís A.},
  year = 2018,
  month = apr,
  journal = {Algebraic \& Geometric Topology},
  volume = {18},
  number = {4},
  pages = {2179--2244},
  publisher = {Mathematical Sciences Publishers},
  doi = {10.2140/agt.2018.18.2179}
}

@article{AyalaFrancis2020FibrationsInftyCategories,
  author = {Ayala, David and Francis, John},
  year = 2020,
  month = feb,
  journal = {Higher Structures},
  volume = {4},
  number = {1},
  pages = {168--265},
  doi = {10.21136/HS.2020.05},
  title = {Fibrations of {$\infty$}-categories}
}

@article{AyalaFrancisRozenblyum,
  author = {Ayala, David and Francis, John and Rozenblyum, Nick},
  title = {Factorization homology {I}: {H}igher categories},
  journal = {Advances in Mathematics},
  volume = {333},
  year = {2018},
  pages = {1042-1177}
}

@incollection{Day1970ClosedCategoriesFunctors,
  title = {On closed categories of functors},
  booktitle = {Reports of the Midwest Category Seminar IV},
  author = {Day, Brian},
  editor = {MacLane, S. and Applegate, H. and Barr, M. and Day, B. and Dubuc, E. and {Phreilambud} and Pultr, A. and Street, R. and Tierney, M. and Swierczkowski, S.},
  year = 1970,
  volume = {137},
  pages = {1--38},
  publisher = {Springer Berlin Heidelberg},
}

@article{Glasman2016DayConvolution,
  author = {Glasman, Saul},
  year = 2016,
  journal = {Mathematical Research Letters},
  volume = {23},
  number = {5},
  pages = {1369--1385},
  doi = {10.4310/MRL.2016.v23.n5.a6},
  title = {Day convolution for {$\infty$}-categories}
}

@article{Hinich2020YonedaLemmaEnriched,
  title = {Yoneda lemma for enriched $\infty$-categories},
  author = {Hinich, Vladimir},
  year = 2020,
  journal = {Advances in Mathematics},
  volume = {367},
  pages = {107129},
  doi = {10.1016/j.aim.2020.107129}
}

@article{HaugsengIterated,
  title = {Iterated spans and classical topological field theories},
  author = {Haugseng, Rune},
  year = 2018,
  journal = {Mathematische Zeitschrift},
  volume = {289},
  number = {3},
  pages = {1427--1488},
  doi = {10.1007/s00209-017-2005-x}
}

@article{ChuHaugseng2023FreeAlgebrasDay,
  title = {Free algebras through {D}ay convolution},
  author = {Chu, Hongyi and Haugseng, Rune},
  year = 2023,
  journal = {Algebraic \& Geometric Topology},
  volume = {22},
  number = {7},
  pages = {3401--3458},
  publisher = {Mathematical Sciences Publishers},
  doi = {10.2140/agt.2022.22.3401}
}

@article{DawsonParePronk,
  author = {Dawson, Robert and Par\'e, Robert and Pronk, Dorette},
  title = {The span construction},
  journal = {Theory and Applications of Categories},
  volume = {24},
  year = {2010},
  number = {13},
  pages = {302-377}
}

@misc{FactSystems,
    author = {Barkan, Shaul and Steinebrunner, Jan},
    title = {On the definition of factorization systems},
    howpublished = {note, \url{https://www.jan-steinebrunner.com/wp-content/uploads/2025/12/On_the_definition_of_factorization_systems.pdf}},
    note = {accessed on 18 Feb.\ 2026},
    year = {2025}
}

@article{HaugsengTheta,
  author = {Haugseng, Rune},
  year = 2018,
  journal = {Proceedings of the American Mathematical Society},
  volume = {146},
  number = {4},
  pages = {1401--1415},
  doi = {10.1090/proc/13695},
  title = {On the equivalence between {$\Theta _{n}$}-spaces and iterated Segal spaces}
}

@misc{Blom2024StraighteningEveryFunctor,
  title = {On the straightening of every functor},
  author = {Blom, Thomas},
  year = 2024,
  howpublished = {arXiv:2408.16539}
}

@misc{Ruit2025Companions,
  title = {Homotopy coherent companionships and conjunctions},
  author = {Ruit, Jaco},
  year = 2025,
  howpublished = {arXiv:2408.14335}
}

@phdthesis{RuitThesis,
title = "On multiple $\infty$-categories and formal category theory via $\infty$-equipments",
author = "Ruit, Jaco",
year = "2025",
publisher = "Utrecht University",
type = "Ph.{D}.\ thesis",
school = "Universiteit Utrecht",
}

@article{Pisani2014SeqMulticategories,
 author = {Pisani, Claudio},
 title = {Sequential multicategories},
 journal = {Theory and Applications of Categories},
 volume = {29},
 pages = {496--541},
 year = {2014}
}

@misc{HaineRamziSteinPushouts,
  author = {Haine, Peter J. and Ramzi, Maxime and Steinebrunner, Jan},
  year = 2025,
  eprint = {2503.03916},
  archiveprefix = {arXiv},
  title = {Fully faithful functors and pushouts of {$\infty$}-categories}
}

\end{document}